\documentclass[11pt]{amsart}
\usepackage[normalem]{ulem}
\usepackage{cancel} % for striking forulas out
\usepackage{enumerate}
\usepackage{amssymb}
\usepackage{amsmath}
\usepackage{mathrsfs}
\usepackage{amsthm}
\usepackage{bm}
\usepackage{verbatim}
\usepackage{pgf}
\usepackage[font=scriptsize, labelfont=bf]{caption}
\usepackage{float}
\usepackage[square,numbers]{natbib}
\usepackage{moreverb,url}
\usepackage{amssymb,amsmath,amsfonts}
\usepackage{bm}
\usepackage[font=scriptsize, labelfont=bf]{caption}
\usepackage{float}
\usepackage{varwidth}
\usepackage{booktabs}
\usepackage{afterpage}
\usepackage{mathrsfs}
\usepackage{subcaption}
\usepackage{verbatim}
\usepackage[square,numbers]{natbib}
\usepackage{stmaryrd}
\usepackage{fancyhdr}
\usepackage{syntax,etoolbox}
\usepackage{graphicx}
\usepackage{grffile}
\usepackage[square,numbers]{natbib}
\usepackage[left=1in, right=1in, bottom=1in,top=1in]{geometry}
\usepackage{lineno}
\usepackage{syntax,etoolbox}
\usepackage{algorithmicx}
\usepackage{algorithm}% http://ctan.org/pkg/algorithms
\usepackage{algpseudocode}% http://ctan.org/pkg/algorithmicx
\usepackage{tikz}
\usepackage{gensymb}
\usepackage{hyperref}
\usepackage[nomain,symbols,nogroupskip]{glossaries}
\hypersetup{colorlinks=true,allcolors=blue}
\setglossarystyle{list}
\newglossary[mlg]{mathsymbols}{msy}{msn}{List of Mathematical Symbols}
\newglossary[plg]{physsymbols}{psy}{psn}{List of Physics Symbols}
\makeglossaries
\newglossaryentry{Omega}{
	type=mathsymbols,
	name=\ensuremath{\Omega},
	description={domain in $\mathbb{R}^2$},
	sort={A2Omega}
}
\newglossaryentry{nuOmega}{
	type=mathsymbols,
	name=\ensuremath{\bm{\nu}_{\partial\Omega}},
	description={outer unit normal vector along $\partial\Omega$},
	sort={A2nuOmega}
}
\newglossaryentry{nabla}{
	type=mathsymbols,
	name=\ensuremath{\nabla},
	description={gradient operator with respect to $(x_1,x_2)\in\Omega$},
	sort={A1nabla}
}
\newglossaryentry{T}{
	type=mathsymbols,
	name=\ensuremath{T},
	description={terminal observation time},
	sort={A1T}
}
\newglossaryentry{C00TX}{
	type=mathsymbols,
	name=\ensuremath{\mathcal{C}^0(\overline{I};X)},
	description={space of $X$-valued functions that are continuous on $\overline{I}$},
	sort={A1C0}
}
\newglossaryentry{Lp0TX}{
	type=mathsymbols,
	name=\ensuremath{L^p(I;X)},
	description={space of $X$-valued functions that are $L^p$ integrable in the interval $I$},
	sort={A1Lp}
}
\newglossaryentry{Wmp0TX}{
	type=mathsymbols,
	name=\ensuremath{W^{m,p}(I;X)},
	description={space of $X$-valued functions $v:I\to X$ such that $v, v', \dots, v^{(m)}\in L^p(I;X)$},
	sort={A1Wmp}
}
\newglossaryentry{W1pq0TX}{
	type=mathsymbols,
	name=\ensuremath{W^{1,p,q}(I;X)},
	description={space of $X$-valued functions $v:I\to X$ such that $v\in L^p(I;X)$ and $v'\in L^q(I;X)$},
	sort={A1Wmpq}
}
\newglossaryentry{Xprime}{
	type=mathsymbols,
	name=\ensuremath{X'},
	description={dual space of the normed vector space $X$},
	sort={A1Xprime}
}
\newglossaryentry{div}{
	type=mathsymbols,
	name=\ensuremath{\nabla\, \cdot},
	description={divergence operator},
	sort={A1div}
}
\newglossaryentry{Domega}{
	type=mathsymbols,
	name=\ensuremath{\mathcal{D}(\omega)},
	description={space of smooth functions with compact support in the open set $\omega$},
	sort={A1D}
}
\newglossaryentry{cP}{
	type=mathsymbols,
	name=\ensuremath{c_{P}},
	description={constant of the Poincar\'{e}-Friedrichs inequality. It results $c_{P}=\frac{\textup{diam }\Omega}{p^{1/p}}$ for~\eqref{penalty:seq} and $c_{P}=\sqrt{2}L$ for~\eqref{FD:temp}},
	sort={A1cP}
}
\newglossaryentry{Ctr}{
	type=mathsymbols,
	name=\normalfont\ensuremath{\hat{C}_{\textup{tr}}},
	description={continuity constant for the classical trace operator defined over the space $V$},
	sort={A1Csurf}
}
\newglossaryentry{Csurf}{
	type=mathsymbols,
	name=\normalfont\ensuremath{\hat{C}_{\textup{surf}}},
	description={bounding constant for the source term in~\eqref{penalty-thetatilde}},
	sort={A1Ctr}
}
\newglossaryentry{Ksurf}{
	type=mathsymbols,
	name=\normalfont\ensuremath{K_{\textup{surf}}},
	description={set of admissible transformed ice surface elevations satisfying the associated physical constraints},
	sort={A1Ksurf}
}
\newglossaryentry{Ksurf2}{
	type=mathsymbols,
	name=\normalfont\ensuremath{\mathcal{K}_{\textup{surf}}},
	description={time-dependent analogue of $K_{\textup{surf}}$},
	sort={A1Ksurf2}
}
\newglossaryentry{V}{
	type=mathsymbols,
	name=\ensuremath{V},
	description={vector space of admissible temperatures},
	sort={A1V}
}
\newglossaryentry{Ktemp}{
	type=mathsymbols,
	name=\normalfont\ensuremath{K_{\textup{temp}}},
	description={set of admissible temperatures in $V$ satisfying the associated physical constraints},
	sort={A1Ktemp}
}
\newglossaryentry{Ktemp2}{
	type=mathsymbols,
	name=\normalfont\ensuremath{\mathcal{K}_{\textup{temp}}},
	description={time-dependent analogue of $K_{\textup{temp}}$},
	sort={A1Ktemp2}
}
\newglossaryentry{neg-part}{
	type=mathsymbols,
	name=\ensuremath{-\{\cdot\}^{-}},
	description={negative part operator defined pointwise by $-\{f\}^{-}=-\min\{f,0\}$},
	sort={A1neg-part}
}
\newglossaryentry{pos-part}{
	type=mathsymbols,
	name=\ensuremath{\{\cdot\}^{+}},
	description={positive part operator defined pointwise by $\{f\}^{+}=\max\{f,0\}$},
	sort={A1pos-part}
}
\newglossaryentry{expf}{
	type=mathsymbols,
	name=\ensuremath{\exp(f)},
	description={exponential of a non-positive integrable function $f$},
	sort={A1expf}
}
\newglossaryentry{left-shift-u}{
	type=mathsymbols,
	name=\ensuremath{\Lambda_k \bm{u}_{\ell,k}},
	description={left-shift in time of the vector $\bm{u}_{\ell,k}$ defined in~\eqref{PilL}},
	sort={A2left-shift-u}
}
\newglossaryentry{right-shift-u}{
	type=mathsymbols,
	name=\ensuremath{\Pi_k \bm{u}_{\ell,k}},
	description={right-shift in time of the vector $\bm{u}_{\ell,k}$ defined in~\eqref{Pil}},
	sort={A2right-shift-u}
}
\newglossaryentry{left-shift-T}{
	type=mathsymbols,
	name=\ensuremath{\Lambda_k \tilde{\bm{\Theta}}_{\ell,k}},
	description={left-shift in time of the vector $\tilde{\bm{\Theta}}_{\ell,k}$ defined in~\eqref{PilTL}},
	sort={A2left-shift-T}
}
\newglossaryentry{right-shift-T}{
	type=mathsymbols,
	name=\ensuremath{\Pi_k \tilde{\bm{\Theta}}_{\ell,k}},
	description={right-shift in time of the vector $\tilde{\bm{\Theta}}_{\ell,k}$ defined in~\eqref{PilT}},
	sort={A2right-shift-T}
}
\newglossaryentry{Nodes}{
	type=mathsymbols,
	name=\ensuremath{N},
	description={number of equally-spaced intervals the interval $[0,T]$ is divided into},
	sort={A1Nodes}
}
\newglossaryentry{k}{
	type=mathsymbols,
	name=\ensuremath{k},
	description={time-step defined by $k:=T/N$},
	sort={A1k}
}
\newglossaryentry{Dk}{
	type=mathsymbols,
	name=\ensuremath{D_k},
	description={finite difference quotient in time with time-step $k$},
	sort={A1Dk}
}
\newglossaryentry{un+1}{
	type=mathsymbols,
	name=\ensuremath{u_{\ell,k}^{n+1}},
	description={Unique solution for~\eqref{penalty:seq} in Problem~\ref{Pkappa:seq}},
	sort={A1un+1}
}
\newglossaryentry{Tn+1}{
	type=mathsymbols,
	name=\ensuremath{\tilde{\Theta}_{\ell,k}^{n+1}},
	description={Unique solution for~\eqref{FD:temp} in Problem~\ref{Pkappa:seq}},
	sort={A2Tn+1}
}
\newglossaryentry{uell}{
	type=mathsymbols,
	name=\ensuremath{u_\ell},
	description={Solution for~\eqref{penalty-u} in Problem~\ref{Pkappa}},
	sort={A1uell}
}
\newglossaryentry{Tell}{
	type=mathsymbols,
	name=\ensuremath{\tilde{\Theta}_\ell},
	description={Solution for~\eqref{penalty-thetatilde} in Problem~\ref{Pkappa}},
	sort={A2Tell}
}
\newglossaryentry{elle}{
	type=mathsymbols,
	name=\ensuremath{\ell},
	description={positive penalty parameter meant to approach zero},
	sort={A1elle}
}
\newglossaryentry{G}{
	type=mathsymbols,
	name=\ensuremath{G},
	description={a weak-star limit for the sequence $\{-\nabla\cdot\left(|\nabla u_{\ell_n}|^{p-1}\nabla u_{\ell_n}\right)\}_{n=1}^\infty$},
	sort={A1G}
}
\newglossaryentry{b}{
	type=physsymbols,
	name=\ensuremath{b},
	description={bedrock},
	sort={A1b}
}
\newglossaryentry{h}{
	type=physsymbols,
	name=\ensuremath{h},
	description={ice surface elevation/ice thickness},
	sort={A1h}
}
\newglossaryentry{Omega+}{
	type=physsymbols,
	name=\ensuremath{\Omega_{t}^{+}},
	description={region of $\Omega$ where $h(t,\cdot)>0$},
	sort={A2Omega+}
}
\newglossaryentry{Omega-}{
	type=physsymbols,
	name=\ensuremath{\Omega_{t}^{-}},
	description={region of $\Omega$ where $h(t,\cdot)=0$},
	sort={A2Omega-}
}
\newglossaryentry{Gammaft}{
	type=physsymbols,
	name=\ensuremath{\Gamma_{f,t}},
	description={free boundary of the ice region},
	sort={A2Gammaft}
}
\newglossaryentry{as}{
	type=physsymbols,
	name=\ensuremath{a_s},
	description={surface mass balance},
	sort={A1as}
}
\newglossaryentry{ab}{
	type=physsymbols,
	name=\ensuremath{a_b},
	description={basal melting rate},
	sort={A1ab}
}
\newglossaryentry{a}{
	type=physsymbols,
	name=\ensuremath{a},
	description={\ensuremath{a_s-a_b}},
	sort={A1a0}
}
\newglossaryentry{atilde}{
	type=physsymbols,
	name=\ensuremath{\tilde{a}},
	description={function defined in~\eqref{aNEW}},
	sort={A1atilde}
}
\newglossaryentry{Ub}{
	type=physsymbols,
	name=\ensuremath{\bm{U}_{b}},
	description={basal sliding velocity},
	sort={A1Ub}
}
\newglossaryentry{U}{
	type=physsymbols,
	name=\ensuremath{\bm{U}},
	description={horizontal ice flow velocity},
	sort={A1U}
}
\newglossaryentry{Q}{
	type=physsymbols,
	name=\ensuremath{\bm{Q}},
	description={vertically integrated ice volume flux},
	sort={A1Q}
}
\newglossaryentry{h0}{
	type=physsymbols,
	name=\ensuremath{h_0},
	description={initial ice surface elevation/ice thickness},
	sort={A1h0}
}
\newglossaryentry{nuGammaft}{
	type=physsymbols,
	name=\ensuremath{\bm{\nu}_{\Gamma_{f,t}}},
	description={outer unit normal vector along the free boundary $\Gamma_{f,t}$},
	sort={A2nuGammaft}
}
\newglossaryentry{nubase}{
	type=physsymbols,
	name=\ensuremath{\bm{\nu}_{b}},
	description={unit vector normal to the bedrock which points towards the portion of three-dimensional space that does not contain the bedrock},
	sort={A2nubase}
}
\newglossaryentry{nuice}{
	type=physsymbols,
	name={\normalfont\ensuremath{\bm{\nu}_{\text{ice}}}},
	description={outer unit normal vector applied along the surface $\{(x,h(t,x));x\in\overline{\Omega_{t}^{+}}\}$ and pointing towards to portion of space complementary to the one occupied by ice at time $t\in[0,T]$},
	sort={A2nuice}
}
\newglossaryentry{Tice}{
	type=physsymbols,
	name={\normalfont\ensuremath{\Theta_{\textup{ice}}}},
	description={shallow ice sheet internal temperature},
	sort={A2Tice}
}
\newglossaryentry{Tm}{
	type=physsymbols,
	name={\normalfont\ensuremath{\Theta_{\textup{m}}}},
	description={temperature that is sufficiently below the pressure melting point},
	sort={A2Tm}
}
\newglossaryentry{Ts}{
	type=physsymbols,
	name={\normalfont\ensuremath{\Theta_{\textup{surf}}}},
	description={ice temperature on the ice surface},
	sort={A2Ts}
}
\newglossaryentry{Tbase}{
	type=physsymbols,
	name={\normalfont\ensuremath{\Theta_{\textup{base}}}},
	description={ice temperature on the bedrock},
	sort={A2Tbase}
}
\newglossaryentry{Tbasetilde}{
	type=physsymbols,
	name={\normalfont\ensuremath{\tilde{\Theta}_{\textup{base}}}},
	description={auxiliary basal temperature},
	sort={A2Tbasetilde}
}
\newglossaryentry{T0}{
	type=physsymbols,
	name=\ensuremath{\Theta_0},
	description={initial shallow ice sheet internal temperature},
	sort={A2T0}
}
\newglossaryentry{T0tilde}{
	type=physsymbols,
	name=\ensuremath{\tilde{\Theta}_0},
	description={auxiliary initial temperature},
	sort={A2T0tilde}
}
\newglossaryentry{Text}{
	type=physsymbols,
	name={\normalfont\ensuremath{\Theta_{\textup{ext}}}},
	description={external auxiliary temperature},
	sort={A2Text}
}
\newglossaryentry{Tint}{
	type=physsymbols,
	name={\normalfont\ensuremath{\Theta_{\textup{int}}}},
	description={internal auxiliary temperature},
	sort={A2Tint}
}
\newglossaryentry{Taux}{
	type=physsymbols,
	name={\normalfont\ensuremath{\Theta_{\textup{aux}}}},
	description={auxiliary temperature defined in $\Omega\times (0,L)$},
	sort={A2Taux}
}
\newglossaryentry{Theta}{
	type=physsymbols,
	name=\ensuremath{\Theta},
	description={function defined in~\eqref{Theta-final-1}},
	sort={A2Theta}
}
\newglossaryentry{Thetatilde}{
	type=physsymbols,
	name=\ensuremath{\tilde{\Theta}},
	description={function defined in~\eqref{Theta-final-2}},
	sort={A2Thetatilde}
}
\newglossaryentry{qgeoperp}{
	type=physsymbols,
	name={\normalfont\ensuremath{q_\textup{geo}^\perp}},
	description={positive geothermal heat flux from bedrock to ice bulk},
	sort={A1qgeoperp}
}
\newglossaryentry{kappa}{
	type=physsymbols,
	name=\ensuremath{\kappa},
	description={ice heat conductivity},
	sort={A2kappa}
}
\newglossaryentry{A}{
	type=physsymbols,
	name=\ensuremath{A},
	description={ice softness},
	sort={A1AA}
}
\newglossaryentry{A0}{
	type=physsymbols,
	name={\normalfont\ensuremath{A_0}},
	description={positive parameter in the Arrhenius law},
	sort={A1AA0}
}
\newglossaryentry{Amin}{
	type=physsymbols,
	name={\normalfont\ensuremath{A_{\textup{min}}}},
	description={minimal ice softness},
	sort={A1AAmin}
}
\newglossaryentry{beta}{
	type=physsymbols,
	name={\normalfont\ensuremath{\beta}},
	description={positive parameter in the Arrhenius law},
	sort={A2beta}
}
\newglossaryentry{rho}{
	type=physsymbols,
	name=\ensuremath{\rho},
	description={ice mass density},
	sort={A2rho}
}
\newglossaryentry{g}{
	type=physsymbols,
	name=\ensuremath{g},
	description={gravitational acceleration},
	sort={A1g}
}
\newglossaryentry{p}{
	type=physsymbols,
	name=\ensuremath{p},
	description={Glen's power law index},
	sort={A1p}
}
\newglossaryentry{vz}{
	type=physsymbols,
	name=\ensuremath{v_{z}},
	description={vertical ice flow velocity},
	sort={A1vz}
}
\newglossaryentry{u}{
	type=physsymbols,
	name=\ensuremath{u},
	description={function appearing in the change of variable~\eqref{transfo}},
	sort={A1u}
}
\newglossaryentry{u0}{
	type=physsymbols,
	name=\ensuremath{u_0},
	description={initial datum for the transformed variational problem~\eqref{vi:ice}},
	sort={A1u0}
}
\newglossaryentry{mu1}{
	type=physsymbols,
	name=\ensuremath{\mu_1},
	description={essential infimum of the function $\mu$ defined},
	sort={A2mu1}
}
\newglossaryentry{mu}{
	type=physsymbols,
	name=\ensuremath{\mu},
	description={function defined in~\eqref{mu-def}},
	sort={A2mu}
}
\newglossaryentry{mu2}{
	type=physsymbols,
	name=\ensuremath{\mu_2},
	description={essential supremum of the function $\mu$},
	sort={A2mu2}
}
\newglossaryentry{c}{
	type=physsymbols,
	name=\ensuremath{c},
	description={ice specific heat},
	sort={A1c}
}
\newglossaryentry{C0}{
	type=physsymbols,
	name=\ensuremath{C_0},
	description={regularisation of the effective stress appearing in~\eqref{regularisation}},
	sort={A1C0}
}
\newglossaryentry{r0}{
	type=physsymbols,
	name=\ensuremath{r_0},
	description={diffuse interface parameter},
	sort={A1r0}
}
\newglossaryentry{L}{
	type=physsymbols,
	name=\ensuremath{L},
	description={positive constant associated with the maximum ice thickness},
	sort={A1L}
}
\newglossaryentry{xi}{
	type=physsymbols,
	name=\ensuremath{\xi},
	description={cut-off function used for defining adiabatic solutions},
	sort={A2xi}
}
\newglossaryentry{Dice}{
	type=physsymbols,
	name={\normalfont\ensuremath{D_{\textup{ice}}(t)}},
	description={portion of three-dimensional space occupied by ice at time $t$},
	sort={A1Dice}
}
\newglossaryentry{alpha}{
	type=physsymbols,
	name=\ensuremath{\alpha},
	description={constant introduced in~\eqref{alpha}},
	sort={A2alpha}
}
\newtheorem{theorem}{Theorem}[section]
\newtheorem{lemma}[theorem]{Lemma}

\theoremstyle{definition}

\newtheorem{remark}[theorem]{Remark}

\numberwithin{equation}{section}

\def\dd{\, \mathrm{d}}

\newcommand{\wsc}{\overset{\ast}{\rightharpoonup}}%

\newtheorem{innercustomgeneric}{\customgenericname}
\providecommand{\customgenericname}{}
\newcommand{\newcustomproblem}[2]{%
	\newenvironment{#1}[1]
	{%
		\renewcommand\customgenericname{#2}%
		\renewcommand\theinnercustomgeneric{##1}%
		\innercustomgeneric
	}
	{\endinnercustomgeneric}
}

\newcustomproblem{customprob}{\textbf{Problem}}
\newcustomproblem{customalgo}{Algorithm}
\newcommand*{\bqed}{\hfill\ensuremath{\blacksquare}}%

\allowdisplaybreaks

\makeatletter
\newcommand{\addresseshere}{%
	\enddoc@text\let\enddoc@text\relax
}
\makeatother

\begin{document}
	
	\today
	
	%%
	%% The title of the paper goes here.  Edit to your title.
	%%
	
	\title[Adiabatic flows in grounded shallow ice sheets]{Diffuse Adiabatic Flows in Thermally Coupled Grounded Shallow Ice Sheets: Modelling and Analysis}
	
	%%
	%% Now edit the following to give your name and address:
	%% 
	
	\author{Paolo Piersanti}
	\address{School of Science and Engineering, The Chinese University of Hong Kong (Shenzhen), 2001 Longxiang Blvd., Longgang District, Shenzhen, China}
	\email{ppiersanti@cuhk.edu.cn}
	
\begin{abstract}
In this article we propose a novel thermodynamical model which couples the evolution of the thickness of a grounded shallow ice sheet with the evolution of its internal temperature. Both the grounded shallow ice sheet surface elevation and the ice internal temperature are subjected to physical constraints. The equations governing the evolution of the grounded shallow ice sheet thickness are degenerate, and the ice internal temperature evolves in a moving domain.
First, we formally model the phenomenon under consideration by adopting strategies akin to those employed in the construction of diffuse-interface models. Second, we establish the existence of solutions for one such formal model by means of the penalty method, and we observe that the low regularity of the problem under consideration prevents us from obtaining a standard concept of solution.
\end{abstract}

\maketitle

\tableofcontents

\section{Introduction}
\label{Intro}

The mathematical modelling of ice sheets is fundamentally characterised by the interaction between a free boundary, governing the ice thickness, and physical fields, such as temperature, defined on the evolving domain itself.

In this article we propose and study a mathematical model describing the interaction between the evolution of the thickness of a grounded shallow ice sheet (from now on, simply \emph{shallow ice sheet}) and the evolution of the ice internal temperature within it.

\subsection*{Literature Review}

The rigorous mathematical study of models in Glaciology dates back, to our best knowledge, to 2002, when Calvo, D\'{i}az, Durany, Schiavi and V\'azquez~\cite{Diaz2002} published a paper discussing the evolution of the thickness of a shallow ice sheet as an obstacle problem. In their article, these authors only considered \emph{one spatial direction}, they assumed the basal velocity to be smooth, and assumed the bedrock to be flat. This formulation was also exploited to justify the generation of fast ice streams in ice sheets flowing along soft and deformable beds~\cite{Diaz2007}. In this direction, we also cite the related paper~\cite{Diaz1999} by D\'{i}az and his collaborators.

Ten years later, in 2012, Jouvet \& Bueler~\cite{JouvBuel2012} studied the steady (i.e., time-independent) version of the problem considered in~\cite{Diaz2002} where, this time, two spatial directions and a more general bedrock topography were taken into account.

In 2023, in~\cite{PT23}, the time-dependent version of the model considered in~\cite{JouvBuel2012} was mathematically analysed, under the assumption that the bedrock was flat. The latter assumption is physically realistic in the context of many ice sheets (e.g., the Greenland ice sheet~\cite{Bamber2001} and the Devon ice cap, described as a dome-shaped ice mass resting atop a bedrock plateau~\cite{DBGBS2004}). The elliptic bulk energy for the model considered in~\cite{JouvBuel2012} is expressed in terms of a tilted $p$-Laplacian, which is in general non-monotone. Requiring the bedrock to be flat restores the monotonicity and ensured the viability of the mathematical analysis of the time-dependent version of the model studied in~\cite{JouvBuel2012} by means of monotonicity methods~\cite{Brezis1972,Lions1969}.

Numerical studies in ice sheet dynamics were carried out by Hughes \& Ghattas and their collaborators in a series of papers. For instance, in~\cite{Ghattas12}, Hughes \& Ghattas and their collaborators proposed an infinite-dimensional adjoint-based inexact Gauss-Newton method for the solution of inverse problems governed by Stokes models of ice sheet flow with non-linear rheology and sliding law. In~\cite{Ghattas15a}, motivated by the need for efficient and accurate simulation of the dynamics of the polar ice sheets, Ghattas and his collaborators designed scalable solvers for the solution of non-linear incompressible Stokes equations. In~\cite{Ghattas15b}, Ghattas and his collaborators proposed efficient and scalable algorithms for this end-to-end, data-to-prediction process under the Gaussian approximation and in the context of modelling the flow of the Antarctic ice sheet and its effect on loss of grounded ice to the ocean. In~\cite{Ghattas16}, Hughes \& Ghattas and their collaborators studied inverse problems to recover the geothermal heat flux starting from ice flow observations.

Stemming from the analysis presented in~\cite{PT23}, numerical studies by Chawla, Holmes and Temam followed, some of which employed Physics-Informed Neural Networks to overcome the difficulties owing to the lack of regularity of solutions of the governing model~\cite{CH2025,CHT2024}.

Grounded ice sheets and sea ice are two related but physically different problems. The mathematical literature dealing with sea ice and glaciers has been growing rapidly since the early 2000s; in this direction we mention, for instance, the articles~\cite{Jouvet2015,Jouvet2015-2,Jouvet2014,Jouvet2013,Michel2014,Schoof2006,Schoof2006-2}, and the landmark paper by W.D. Hibler~\cite{Hibler1979} on sea ice modelling. The local-in-time well-posedness of Hibler's model has recently been established in the article by Liu, Thomas \& Titi~\cite{Titi2021}. The authors designed a regularizing scheme based on physical observations for deriving the sought local-in-time well-posedness.

In~\cite{Hieber2021}, Hieber and his collaborators investigated the local-in-time well-posedness of Hibler's model by means of a different regularisation approach. Precisely, the authors constructed the local strong solutions of the corresponding regularised system by investigating the analytic semi-group property of the parabolic operator of the regularised system in a bounded domain, and they also established the global well-posedness in time for initial data close to constant equilibria.
In this direction, we also mention the recent work by Hieber's research group~\cite{Brandt2022-3,Brandt2022-2,Brandt2022-1}.

In the recent paper~\cite{FRS2024}, Figalli, Ros-Oton \& Serra studied the phase transition of ice melting to water as a Stefan problem. The results established in~\cite{FRS2024} provide a refined understanding of the Stefan problem's singularities and answer some long-standing open questions in the field of free-boundary problems.

\subsection*{Novelties and main results}

The purpose of this paper is to extend the study initiated in~\cite{PT23} to the case where the evolution of the ice internal temperature is coupled with the evolution of the ice thickness. The \emph{main point of innovation} introduced here consists in considering the coupling between the ice thickness evolution and the internal ice temperature evolution. Since the ice thickness level is constrained to remain on or above the bedrock at all times, the variational problem describing the evolution of the ice thickness can be formulated as a \emph{time-dependent obstacle problem}.

On the one hand, the ice thickness of a shallow ice sheet evolves as a consequence of a manifold of factors like, for instance, the rate at which snow deposits, the rate at which melting occurs, as well as the velocity at which the glacier slides along the bedrock.

On the other hand, the ice internal temperature evolves as a result of the temperature of the surrounding environment, the heat irradiated through the bedrock surface, and the effective stress exerted by the ice mass.
According to Section~5.4 in~\cite{GreveBlatter2009}, the governing equations coupling the evolution of ice thickness and the evolution of ice internal temperature and the boundary conditions change if the ice internal temperature either drops below the \emph{pressure melting point} or reaches the pressure melting point, developing \emph{latent heat}.

In the context of this analysis, we consider the case where the ice bedrock is flat (as in~\cite{PT23}) and \emph{frozen}, in the sense that the ice internal temperature is constrained to remain below the pressure melting point. Also in this case, the variational problem associated with the evolution of the ice internal temperature can be formulated as a time-dependent obstacle problem.

The problem under consideration, in addition to being novel from both the modelling point of view and the analytical point of view, introduces challenges owing to the fact that the ice internal temperature is defined in the portion of space occupied by ice. This means that, in addition to having to deal with non-linearities associated with the presence of a \emph{free boundary}, for both the shallow ice sheet thickness and the ice internal temperature, substantial caution must be exercised as the governing equations for the ice internal temperature are posed over a \emph{moving domain}, where the motion of the domain boundary is expressed in terms of the ice thickness, which is one of the unknowns of the problem under consideration.
To summarise, the \textbf{main novelties} presented in this paper are the following:
\begin{itemize}
	\item[$(1)$] This paper introduces (viz. Section~\ref{Sec:1}) reduction techniques for transforming the original model coupling the evolution of the shallow ice sheet thickness with the evolution of the ice internal temperature (cf., e.g., Section~5.4 in~\cite{GreveBlatter2009}) to an equivalent model posed over a \emph{fixed} domain. These reduction techniques render the problem more tractable from the mathematical point of view at the cost of considering an approximation of the temperature evolution that vanishes near the ice surface. From the point of view of Physics, this feature characterises ice sheets covered with a continuum layer of debris. The latter feature of the thermal model, which is observable in the Miage ice sheet and the Karakorum ice sheet~\cite{MMDLST06,MBDDCKCS08,NB06,RHH21}, justifies the choice for the adjective \emph{adiabatic} in the title of this paper (cf. Remark~\ref{rem:3});
	\item[$(2)$] This paper establishes the existence of \emph{very weak solutions} of the initial boundary-value problem associated with the \emph{reduced model} obtained upon completion of item~$(1)$, coupling the evolution of the shallow ice sheet surface with the evolution of the ice internal temperature, in the case where both the unknown magnitudes are subjected to obeying or geometrical or physical constraints, i.e., the ice thickness is constrained to be greater than zero and the ice internal temperature is constrained to remain below the pressure melting point;
	\item[$(3)$] Beyond its glaciological application, this work introduces a new analytical framework for a class of coupled obstacle problems posed on moving domains.
\end{itemize}

\subsection*{The main result}
The main result of this paper is \textbf{Theorem~\ref{thm:4}}, where we establish the \emph{existence of very weak solutions} for the initial boundary-value problem coupling the evolution of the shallow ice sheet thickness with the evolution of its ice internal temperature in the case where the \emph{bedrock is frozen and flat}. In order to do so, we will need to improve the energy estimates in~\cite{PT23} by \emph{removing} the extra \emph{decrease assumption}~$(4.6)$ made there. In the context of the derivation of the rigorous concept of solution, in~\cite{PT23} extra regularity assumptions on the weak limit were resorted to. In this paper, we drop this extra regularity assumptions and we obtain an \emph{abstract} limit model. This conclusion is consistent with the results obtained in~\cite{PT23} where extra regularity had been assumed, while the solution concept aligns with the non-standard framework for evolutionary variational inequalities governed by pseudo-monotone operators~\cite{Pham2013}.

\subsection*{Structure of the paper}
This paper is divided into eight sections (including this one). In Section~\ref{Sec:1} we present a formal derivation of the model governing the evolution of the shallow ice sheet surface when the contribution of the ice internal temperature is taken into account.

In Section~\ref{Sec:1bis} we present a formal derivation of the model governing the ice temperature evolution, which takes the form of a \emph{moving boundary problem}, where the moving boundary is associated with the solution of the model presented in Section~\ref{Sec:1}.

In Section~\ref{Sec:2} we formulate the ``penalised'' version of the variational problem introduced in Section~\ref{Sec:1} and Section~\ref{Sec:1bis}, and we establish the existence of solutions for the time discretisation for this model. We also present new results that will be employed to study the \emph{ice softness term} coupling the model governing the evolution of the shallow ice sheet surface with the model governing the evolution of the ice internal temperature.

In Section~\ref{Sec:3} we establish the validity of suitable \emph{a priori} estimates that will allow us to resort to compactness methods.

In Section~\ref{Sec:3:bis} we establish other preparatory lemmas, that descend from the \emph{a priori} estimates derived in Section~\ref{Sec:3}.

In Section~\ref{Sec:3:ter}, we pass to the limit as the time step tends to zero in the discrete model introduced in Section~\ref{Sec:2}, and we establish the existence of solutions for the ``penalised'' version of the \emph{formal} obstacle problems introduced in Section~\ref{Sec:1} and Section~\ref{Sec:1bis}.

Finally, in Section~\ref{Sec:4}, we pass to the limit as the penalty parameter tends to zero in the ``penalised'' version of the \emph{formal} obstacle problems introduced in Section~\ref{Sec:1} and Section~\ref{Sec:1bis}, and we recover a \emph{rigorous concept of solution} for the model corresponding we formally derived at the end of Section~\ref{Sec:1} and Section~\ref{Sec:1bis}.

\subsection*{Notation}
For the sake of facilitating the reading of this paper, a Glossary was appended after the bibliographical references.

We consider the $N$-dimensional Euclidean space $(\mathbb{R}^N,\cdot)$ with its standard inner product. For an open subset $\omega$ of $\mathbb{R}^N$, the symbols $L^2(\omega)$, $H^m(\omega)$, and $H^m_0(\omega)$ (for $m \ge 1$) denote the standard Lebesgue and Sobolev spaces, while $\gls{Domega}$ represents the space of infinitely differentiable functions with compact support in $\omega$. The norm in a normed vector space $X$ is denoted by $\left\| \cdot \right\|_X$, its dual space is denoted by $\gls{Xprime}$, and the duality pairing between $X'$ and $X$ is denoted by $\langle \cdot, \cdot\rangle_{X',X}$. Boldface letters indicate spaces of vector-valued functions.

For a bounded open interval $I$ (see~\cite{Leoni2017}), the Lebesgue and Sobolev spaces of functions with values in a Banach space $X$ are denoted by $\gls{Lp0TX}$, $\gls{Wmp0TX}$ and $\gls{W1pq0TX}$, where $1 \le p \le \infty$ and $m\ge 1$. The space of continuous functions from the closure of a bounded interval $I$ to a Banach space $X$ is denoted by $\gls{C00TX}$. The respective norms of these spaces are $\left\|\cdot\right\|_{L^p(I;X)}$, $\left\|\cdot\right\|_{W^{m,p}(I;X)}$ and $\|\cdot\|_{\mathcal{C}^0(\overline{I};X)}$, respectively.

A \emph{Lipschitz domain in} $\mathbb{R}^N$ is defined as a bounded, connected open subset $\omega$ of $\mathbb{R}^N$ with a Lipschitz-continuous boundary $\partial\omega$, where $\omega$ lies locally on one side of its boundary (cf., e.g., Section~8.2 in~\cite{Ciarlet2025}).

In what follows, we let $\gls{Omega}$ be a Lipschitz domain in $\mathbb{R}^2$, and let $x=(x_1,x_2)$ denote a generic point in $\overline{\Omega}$. The outer unit normal vector field on $\partial\Omega$ is denoted by $\gls{nuOmega}$.
Let $0<\gls{T}<\infty$ be a given terminal time, and consider the time interval $(0,T)$, where $t\in (0,T)$ denotes an arbitrary time instant.

\section{Formal derivation of the shallow ice sheet surface evolution model}
\label{Sec:1}

The \emph{bedrock elevation} is described by the function $\gls{b}:\overline{\Omega} \to \mathbb{R}$. We assume that the bedrock topography remains constant over the observation period and, in agreement with~(5.85) in~\cite{GreveBlatter2009}, that $b\equiv 0$ in $\overline{\Omega}$ up to changing the coordinate system. The latter assumption is applicable to the Greenland ice sheet~\cite{JouvBuel2012} and the Devon ice cap~\cite{CH2025,CHT2024}.

The \emph{ice surface elevation} is given by $\gls{h}:[0,T]\times \overline{\Omega} \to \mathbb{R}$, which immediately implies that the following constraint has be taken into account:
\begin{equation}
	\label{constraint-h}
	h \ge 0 \quad\textup{ in }[0,T]\times \overline{\Omega}.
\end{equation}

This property frames the problem of ice thickness evolution as an obstacle problem, with the bedrock acting as the \emph{obstacle} and, for this particular choice of the framework, it results that the ice surface elevation coincides with the ice thickness.

We recall that in the Shallow Ice Approximation regime, the ice surface elevation is small in the bulk and might exhibit large slopes at the ice margins, leading to potential degeneracies in the governing equations~\cite{GreveBlatter2009,Hutter1983}.
The presence of this constraint introduces a \emph{free boundary} (see~\cite{Brezis1972,SalsaCaff2005,Fichera77}). For all, or almost every, $t \in (0,T)$, we define the set
\begin{equation*}
	\label{ice-region}
	\gls{Omega+}:=\{x \in \Omega; h(t,x) > 0\},
\end{equation*}
which represents the \emph{region of $\Omega$ covered by ice at time $t$}. The associated \emph{free boundary} is the set:
\begin{equation*}
	\label{free-boundary}
	\gls{Gammaft}:=\Omega \cap \partial\Omega_{t}^{+}.
\end{equation*}

The evolution of the ice thickness $h$ is driven by the \emph{surface-mass balance} $\gls{as}$ (accumulation/ablation rate) and the \emph{basal melting rate} $\gls{ab}$. The function $a_b$ is zero when the basal temperature is below the pressure melting point and positive otherwise. Typically, $a_s$ and $a_b$ depend on the horizontal position $x$ and the surface elevation $h(t,x)$, i.e., they are defined as follows:
\begin{equation*}
	\label{accumulation}
	\begin{aligned}
		a_s&:[0,T] \times \overline{\Omega} \times \mathbb{R}^+_{0} \to \mathbb{R},\\
		a_b&:[0,T] \times \overline{\Omega} \times \mathbb{R}^+_{0} \to \mathbb{R}^+_{0}.
	\end{aligned}
\end{equation*}

Since for the problem under consideration the temperature is constrained not to rise above the pressure melting point (cf., Section~5.4 in~\cite{GreveBlatter2009}), it results $a_b \equiv 0$, so that $\gls{a}:=a_s-a_b=a_s$. This function $a$ is assumed to satisfy the following properties (cf., e.g., \cite{GreveBlatter2009,JouvBuel2012,SchoofHewitt2013}):
\begin{gather}
	a\in\mathcal{C}^0([0,T]\times\overline{\Omega}\times\mathbb{R}^+_{0}),\label{a:continuous}\\
	a>0 \textup{ in accumulation areas within }\Omega_{t}^{+},\label{a:positive-1}\\
	a\ge 0 \textup{ in ablation areas within }\Omega_{t}^{+},\label{a:positive-2}\\
	a<0 \textup{ in }\gls{Omega-}:=\Omega\setminus(\Omega_{t}^{+} \cup \Gamma_{f,t}).\label{a:negative}
\end{gather}

Ice sheets are characterised by incompressible, non-Newtonian, gravity-driven flows~\cite{Fowler1997, GreveBlatter2009}, which transport ice from accumulation areas to ablation areas within $\Omega_{t}^{+}$.
The \emph{basal sliding velocity} is a given horizontal vector field $\gls{Ub}:\overline{\Omega}\to\mathbb{R}^2$. For a \emph{frozen} base (i.e., with bedrock temperature below the pressure melting point), it is reasonable to assume that $\bm{U}_b\equiv\bm{0}$ (cf., e.g., Section~5.4 in~\cite{GreveBlatter2009}). The \emph{horizontal ice flow velocity} is denoted by $\gls{U}:[0,T] \times \overline{\Omega} \times \mathbb{R} \to \mathbb{R}^2$, and the \emph{vertically integrated ice volume flux} is denoted and defined by:
\begin{equation}
	\label{flux}
	\gls{Q}:[0,T] \times \overline{\Omega}\to\mathbb{R}^2, \quad \bm{Q}:=\int_{0}^{h} \bm{U} \dd z.
\end{equation}

Following~\cite{JouvBuel2012}, we assume no ice flux into the ice-free region $\Omega_{t}^{-}$
\begin{equation*}
	\label{normal-flux}
	\bm{Q}\cdot\bm{\nu}_{\Gamma_{f,t}} =0, \quad \textup{ on }\Gamma_{f,t},
\end{equation*}
where $\gls{nuGammaft}$ denotes the \emph{outer unit normal vector to the free boundary $\Gamma_{f,t}$ pointing towards the region of $\Omega$ that is not covered with ice}. Consistently, we are in position to extend the volume flux by zero outside of $\Omega_{t}^{+}$:
\begin{equation}
	\label{outer-flux}
	\bm{Q}=\bm{0} \quad \textup{ in } \Omega_{t}^{-}.
\end{equation}

Furthermore, as in~\cite{JouvBuel2012}, we impose the boundary conditions
\begin{equation*}
	\label{BC}
	h(t,\cdot)=0 \textup{ on }\Gamma_{f,t} \qquad\textup{ and } \qquad h(t,\cdot)=0 \textup{ on } (0,T)\times\partial\Omega,
\end{equation*}
noting that $h(t,\cdot)$ can be extended by zero in $\Omega_{t}^{-}$.

The ice thickness evolution is governed by the \emph{ice thickness equation} (cf. equation~(5.55) in~\cite{GreveBlatter2009}):
\begin{equation}
	\label{sie}
	\dfrac{\partial h}{\partial t}= -\nabla \cdot \bm{Q} + a,\quad\textup{ in }\Omega_{t}^{+} \textup{ for a.a. }t\in (0,T).
\end{equation}

A proper description of the free boundary $\Gamma_{f,t}$ and the region $\Omega_{t}^{+}$ requires a weak formulation of the problem. To formally derive the associated boundary value problem, we first multiply~\eqref{sie} by $(v-h)$ for a sufficiently smooth test function $v$ satisfying $v \ge 0$ in $[0,T] \times \overline{\Omega}$, and then integrate over $\Omega$. Using the decomposition $\Omega = \Omega_{t}^{+} \sqcup \Omega_{t}^{-} \sqcup \Gamma_{f,t}$, applying the Green formula (cf., e.g., Theorem~8.2-2 in~\cite{Ciarlet2025}), incorporating conditions~\eqref{outer-flux}, resorting to the property that $h=0$ in $\Omega_{t}^{-}$ and -- thus -- that $\bm{Q}(t,\cdot)=\bm{0}$ in $\Omega_{t}^{-}$, resorting to the property that $\frac{\partial h}{\partial t}(t,\cdot)\ge 0$ in $\Omega_{t}^{-}$, and resorting to the properties of $a$ (viz.~\eqref{a:continuous}--\eqref{a:negative}) leads to
\begin{align*}
	&\int_{\Omega}\dfrac{\partial h}{\partial t} (v-h)\dd x-\int_{\Omega}\bm{Q}\cdot\nabla(v-h)\dd x\\
	&=\int_{\Omega_{t}^{+}}\dfrac{\partial h}{\partial t} (v-h)\dd x-\int_{\Omega_{t}^{+}}\bm{Q}\cdot\nabla(v-h)\dd x\\
	&+\int_{\Omega_{t}^{-}}\dfrac{\partial h}{\partial t}(v-h)\dd x-\int_{\Omega_{t}^{-}}\bm{Q}\cdot\nabla(v-h)\dd x\ge\int_{\Omega}a(v-h)\dd x,
\end{align*}
for a.a. $t\in (0,T)$.
We thus obtained the following weak variational formulation associated with~\eqref{sie}: \emph{Find $h \ge 0$ satisfying the variational inequalities}
\begin{equation}
	\label{weak-4}
	\int_{0}^{T} \int_{\Omega} \dfrac{\partial h}{\partial t} (v-h) \dd x\dd t -\int_{0}^{T} \int_{\Omega} \bm{Q} \cdot \nabla (v-h) \dd x \dd t \ge \int_{0}^{T}\int_{\Omega} a (v-h) \dd x\dd t,
\end{equation}
\emph{for all test functions $v$ with $v \ge 0$ a.e. in $(0,T) \times \Omega$, satisfying the boundary conditions}
$$
h=0\quad\textup{ on }(0,T)\times\partial\Omega,
$$
\emph{and satisfying the initial condition}
\begin{equation*}
	h(0,x) = h_0(x) \quad \textup{ for all }x\in\overline{\Omega},
\end{equation*}
\emph{where the initial datum $\gls{h0}=h(0)$ is given, non-negative, and not identically zero (cf.~\eqref{constraint-h}).}

Examples of initial values $h_0$ could be, for instance, the Bueler or Vialov profile~\cite{GreveBlatter2009}.
\begin{figure}[H]
	\centering
	\begin{subfigure}[t]{0.45\linewidth}
		\includegraphics[width=1.0\textwidth]{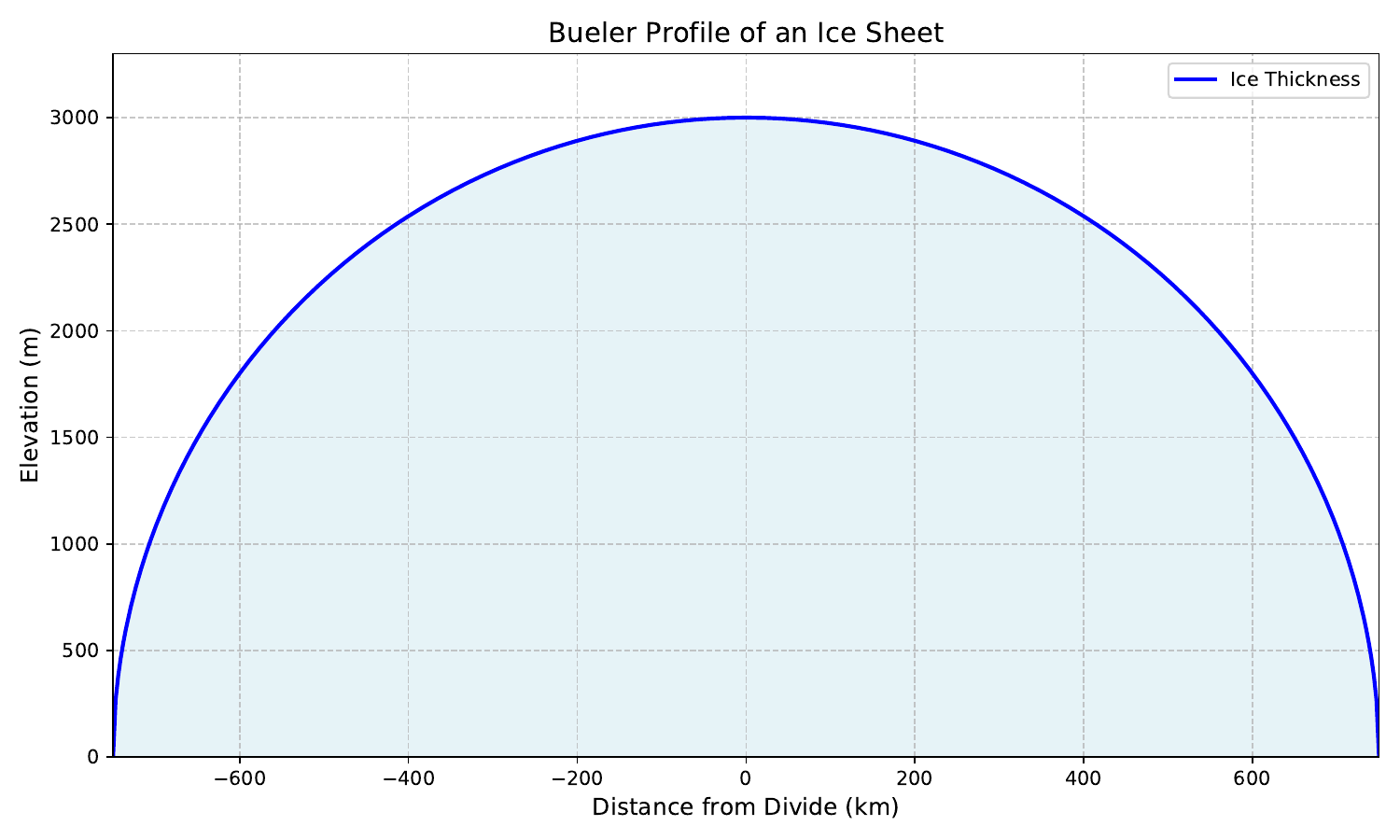}
		\subcaption{Bueler profile.}
	\end{subfigure}%
	\hspace{0.5cm}
	\begin{subfigure}[t]{0.45\linewidth}
		\includegraphics[width=1.0\textwidth]{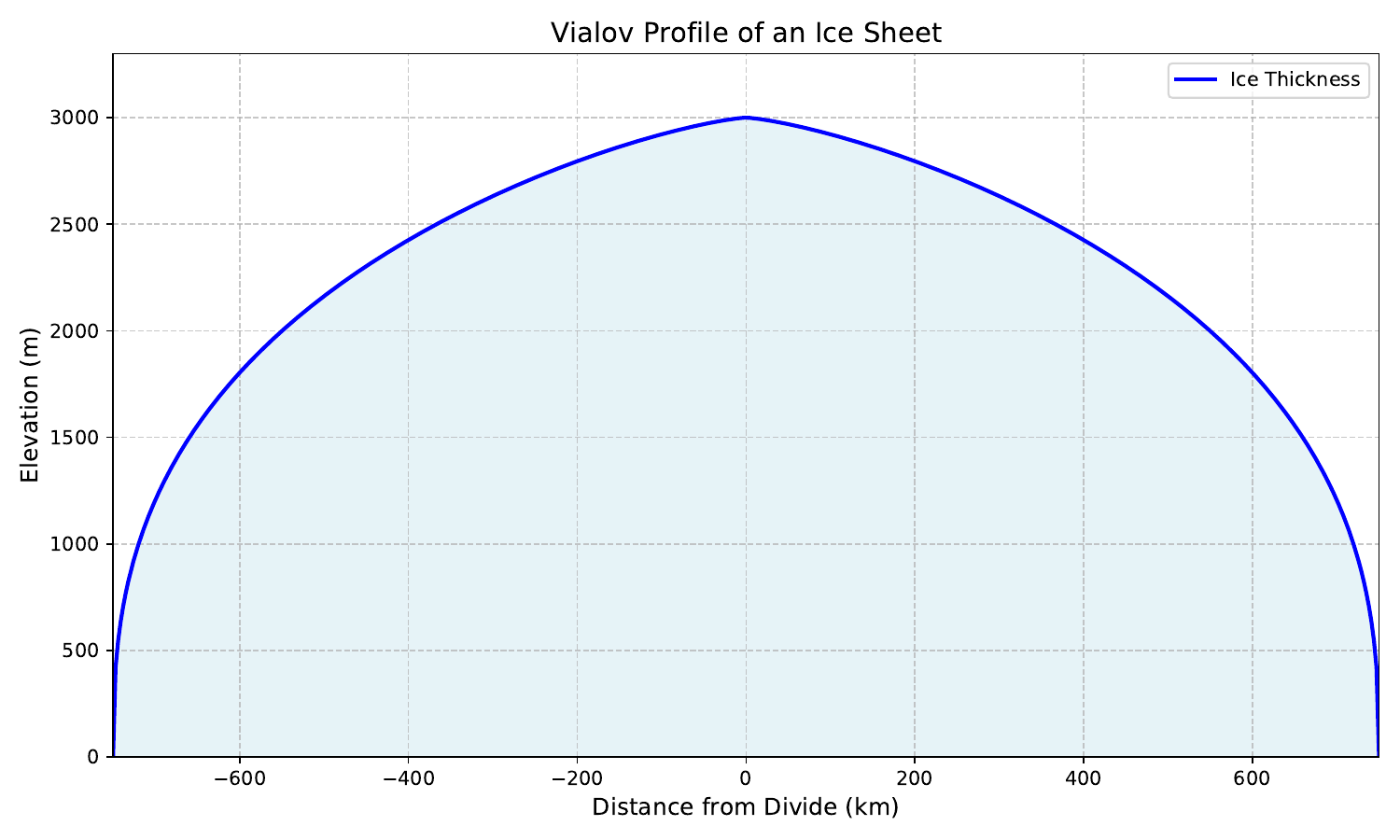}
		\subcaption{Vialov profile.}
	\end{subfigure}%
	\caption{Instances of Bueler profile and Vialov profile~\cite{GreveBlatter2009}.}
	\label{fig:BV}
\end{figure}

The precise definition of a solution, including the regularity requirements for $h$ and the data, will be addressed in Sections~\ref{Sec:2} and~\ref{Sec:3}.

Let $\gls{Tm}>0$ be a constant denoting \emph{a temperature sufficiently below the pressure melting point}\footnote{From the physical point of view, the pressure melting point represents a critical threshold which, when reached, latent heat is produced and the model describing the evolution of the ice surface elevation (viz., e.g., formula~(5.98) in~\cite{GreveBlatter2009}) changes. Since the thermal model we will consider in Section~\ref{Sec:1bis} (viz.~\eqref{eq:thermal}) is not designed to take into account upper bounds for the temperature, we must enforce this constraint via the variational inequality approach.}.
For each $t\in [0,T]$ and $x\in \overline{\Omega_{t}^{+}}$, we define the function
\begin{equation*}
	\gls{Tice}(t,x,\cdot):[0,h(t,x)]\to[0,\Theta_{\textup{m}}],
\end{equation*}
which represents the \emph{shallow ice sheet internal temperature at time $t$}.
We denote by $\gls{A}=A(x,z,\Theta_{\textup{ice}}(t,x,z))$ the \emph{ice softness}\footnote{The function $A$ is also known in the literature as \emph{rate factor}~\cite{Fowler2024,GreveBlatter2009}.} and we observe that a widely accepted expression for $A$ is given via the Arrhenius law (cf., e.g., Section~4.3 in~\cite{GreveBlatter2009}, or page~266 in~\cite{Fowler2024})
\begin{equation}
	\label{arrhenius}
	A(x,z,\Theta_{\textup{ice}}(t,x,z))=\max\left\{A_0\exp\left(-\beta\left(\Theta_{\textup{ice}}(t,x,z)-\Theta_{\textup{m}}\right)\right),A_{\textup{min}}\right\},
\end{equation}
for all $t\in[0,T]$, $x\in\overline{\Omega}$ and all $0\le z\le h(t,x)$, where $\gls{Amin}>0$ is the \emph{minimal ice softness}\footnote{In terrestrial ice sheets, temperatures range from about $200\degree\textup{K}$ (Antarctic cold regions) to $273\degree\textup{K}$. Even at the coldest natural conditions, $A$ remains positive but very small (e.g., $\sim 10^{-28}\,\mathrm{Pa}^{-3}\,\mathrm{s}^{-1}$ for $T=200\degree\textup{K}$)~\cite{JouvBuel2012}.}, and $\gls{A0}\ge A_{\textup{min}}$ and $\gls{beta}>0$ are positive parameters (cf., e.g. formula~(14) in~\cite{Fowler2024}).

We will see in the forthcoming construction of the model that the function $A$ will provide the coupling with the internal ice temperature in the equations governing the evolution of the shallow ice sheet thickness.

Ice exhibits viscous behaviour, with its viscosity characterised by the Glen power law~\cite{Glen1970} (see also equation (4.16) in~\cite{GreveBlatter2009}). This law incorporates the ice softness parameter $A(x,z,\Theta_{\textup{ice}}(t,x,z))$ and the \emph{Glen exponent} $p$, which satisfies $2.8 \le \gls{p} \le 5$. The range of values for $p$ is supported by laboratory experiments~\cite{GoldKohl2001}. Treating the ice softness as a function enables the coupling of the ice thickness equation with a thermodynamic model~\cite{GreveBlatter2009, Hooke2005}. From the point of view of ice modelling, the celebrated constitutive relations proposed by Hibler~\cite{Hibler1979}, itself inspired by the article of Coon~\cite{Coon1974}, are \emph{less accurate} in describing the horizontal ice flow velocity $\bm{U}$ than Glen's model, on the one hand. On the other hand, Hibler's approach achieves the goal of describing in the same instance the twofold nature of ice as a solid and as a fluid.

According to equation~{(5.84)} in~\cite{GreveBlatter2009}, the horizontal ice flow velocity $\bm{U}=(U_1,U_2)$ can be expressed in terms of the upper ice surface elevation $h$ via Glen's power law~\cite{Glen1970}
\begin{equation}
	\label{velocity}
	\bm{U}(t,x,z)=-2 (\rho g)^{p-1} \left(\int_{0}^{z} A(x,s,\Theta_{\textup{ice}}(t,x,s)) (h(t,x)-s)^{p-1} \dd s\right) |\nabla h(t,x)|^{p-2} \nabla h(t,x),
\end{equation}
for all $(t,x,z)\in [0,T]\times\overline{\Omega}\times [0,h(t,x)]$, where $\gls{rho}>0$ is a constant describing the \emph{ice mass density}, and $\gls{g}$ denotes the \emph{gravitational acceleration}.

We note that formula~\eqref{velocity} not only employs Glen's power law, but is also valid only within a specific shallow limit derived from the Stokes model (see, for example, \cite{SchoofHewitt2013}). The \emph{vertical ice flow velocity} (cf., e.g., formula~(5.72) in~\cite{GreveBlatter2009}) is given by:
\begin{equation}
	\label{vertical-velocity}
	\gls{vz}(t,x,z):=-\int_{0}^{z}\nabla\cdot\bm{U}(t,x,s)\dd s, \quad\textup{ for all } (t,x,z)\in [0,T]\times\overline{\Omega}\times [0,h(t,x)].
\end{equation}

Substituting expression~\eqref{velocity} into~\eqref{flux} gives:
\begin{equation}
	\label{flux-2}
	\begin{aligned}
		\bm{Q}&=-2 (\rho g)^{p-1}\left(\int_{0}^{h} \int_{0}^{z} A(\cdot,s,\Theta_{\textup{ice}}(\cdot,\cdot,s)) (h-s)^{p-1} \dd s \dd z\right) |\nabla h|^{p-2} \nabla h\\
		&=-2 (\rho g)^{p-1}\left(\int_{0}^{h} A(\cdot,s,\Theta_{\textup{ice}}(\cdot,\cdot,s)) (h-s)^{p} \dd s \right) |\nabla h|^{p-2} \nabla h,
	\end{aligned}
\end{equation}
where the latter equality is obtained by an integration by parts with respect to the variable $z$.

The weak formulation~\eqref{weak-4} constitutes a \emph{degenerate extension} of the $p$-Laplacian obstacle problem because the ice thickness $h$ approaches zero at the free boundary $\Gamma_{f,t}$ for each $t \in [0,T]$. Consequently, solutions to~\eqref{weak-4} exhibit infinite gradients at the ice margin $\Gamma_{f,t}$, as discussed in~\cite{Bueler2005,GreveBlatter2009,JouvBuel2012}. The boundary degeneracy was verified by means of laboratory experiments, the results of which are reported in~\cite{SayagWorster2013}.
A reformulation yielding a non-degenerate $p$-Laplacian structure is presented below (see~\eqref{weak-u}). Although this alternative formulation resolves the non-degeneracy, it would introduce a ``tilt'' that disrupts the monotonicity of the variational form. This is the very reason for which we limit ourselves to considering flat bedrocks in the context of this paper.

In order to overcome the difficulty arising as a result of the gradient degeneracy in the vicinity of the free boundary, we introduce the following \emph{change of variables}, originally suggested by P.-A.~Raviart in~\cite{Raviart1967} (see also~\cite{Diaz2002}):
\begin{equation}
	\label{transfo}
	h:=\gls{u}^{\frac{p-1}{2p}}.
\end{equation}

Observe that $h > 0$ in $[0,T] \times \overline{\Omega}$ if and only if $u > 0$ in $[0,T] \times \overline{\Omega}$, and that $h(t,x)=0$ in $[0,T] \times \overline{\Omega}$ if and only if $u=0$ in $[0,T]\times\overline{\Omega}$. As a result of the transformation~\eqref{transfo}, we \emph{first} obtain, formally,
\begin{equation}
	\label{dHdt}
	\dfrac{\partial h}{\partial t}=\dfrac{\partial}{\partial t}\left(|u|^{\frac{3p-1}{2p}-2} \, u\right),
\end{equation}
and, \emph{secondly}, we obtain
\begin{equation}
	\label{ice-soft}
	\begin{aligned}
		&\int_{0}^{h} A(x,s,\Theta_{\textup{ice}}(t,x,s)) (h-s)^p \dd s\\
		&=u^{\frac{(p+1)(p-1)}{2p}} \int_{0}^{1} A(x,u^{\frac{p-1}{2p}} \, s',\Theta_{\textup{ice}}(t,x,u^{\frac{p-1}{2p}}\,s')) (1-s')^p \dd s',
	\end{aligned}
\end{equation}
and, \emph{third}, we obtain:
\begin{equation}
	\label{grad}
	|\nabla h|^{p-2} \nabla h=\left(\dfrac{p-1}{2p}\right)^{p-1} u^{\frac{(-p-1)(p-1)}{2p}} |\nabla u|^{p-2} \nabla u.
\end{equation}

Define the function $\tilde{a}:[0,T] \times \overline{\Omega} \times \mathbb{R} \to \mathbb{R}$ by:
\begin{equation}
	\label{aNEW}
	\gls{atilde}(t,x,u):=a\left(t,x,u^{\frac{p-1}{2p}}\right).
\end{equation}

Inserting~\eqref{ice-soft}--\eqref{aNEW} into~\eqref{flux-2} gives:
\begin{equation}
	\label{flux-3}
	-\bm{Q}=2\left(\rho g \dfrac{p-1}{2p}\right)^{p-1}\left(\int_{0}^{1} A\left(x,u^{\frac{p-1}{2p}} s',\Theta_{\textup{ice}}(t,x,u^{\frac{p-1}{2p}}\,s')\right) (1-s')^p \dd s'\right)|\nabla u|^{p-2} \nabla u.
\end{equation}

For the sake of brevity, we define the function $\gls{mu}:\overline{\Omega} \times \mathbb{R} \to \mathbb{R}$ as a function of $u$ by:
\begin{equation}
	\label{mu-def}
	\begin{aligned}
		&\mu(x,u,\Theta_{\textup{ice}}(t,x,u)):=2\left(\rho g \dfrac{p-1}{2p}\right)^{p-1}\\
		&\quad\times\left(\int_{0}^{1} A\left(x,u^{\frac{p-1}{2p}} s',\Theta_{\textup{ice}}(t,x,u^{\frac{p-1}{2p}}\,s')\right) (1-s')^p \dd s'\right).
	\end{aligned}
\end{equation}

Since the ice softness $A$ is of the form~\eqref{arrhenius}, we obtain that
\begin{equation}
	\label{mu0}
	\gls{mu1}:=2\left(\rho g \dfrac{p-1}{2p}\right)^{p-1}\dfrac{A_\textup{min}}{p+1}\le\mu(x,u,\Theta_{\textup{ice}}(t,x,u))\le 2\left(\rho g \dfrac{p-1}{2p}\right)^{p-1}\dfrac{A_0}{p+1}=:\gls{mu2},
\end{equation}
for all $t\in [0,T]$ and all $x\in\overline{\Omega}$. Observe that plugging~\eqref{mu-def} into~\eqref{flux-3} gives:
\begin{equation}
	\label{flux-4}
	-\bm{Q}=\mu(x,u,\Theta_{\textup{ice}}(t,x,u)) |\nabla u|^{p-2} \nabla u.
\end{equation}

Plugging~\eqref{dHdt}--\eqref{aNEW} and~\eqref{flux-4} into~\eqref{weak-4} leads to the following weak formulation: \emph{Find} $u \ge 0$ \emph{satisfying the variational inequality}
\begin{equation}
	\label{vi:ice}
	\begin{aligned}
		&\int_{0}^{T} \int_{\Omega} \left(\dfrac{\partial}{\partial t}\left(|u|^{\frac{3p-1}{2p}-2} \, u\right)\right)(v-u) \dd x \dd t -\int_{0}^{T} \int_{\Omega} \mu(x,u,\Theta_{\textup{ice}}) |\nabla u|^{p-2} \nabla u \cdot \nabla(v-u) \dd x\dd t\\
		&\ge \int_{0}^{T} \int_{\Omega} \tilde{a} (v-u) \dd x \dd t,
	\end{aligned}
\end{equation}
\emph{for all test functions} $v$ \emph{such that $v \ge 0$ in $[0,T]\times\overline{\Omega}$, satisfying the following boundary conditions}
$$
u(t,\cdot)=0,\quad\textup{ on }\partial\Omega \textup{ for all }t\in[0,T],
$$
\emph{and satisfying the initial condition:}
$$
u(0,\cdot)=\gls{u0}=h_0^{\frac{2p}{p-1}},\quad\textup{ in }\overline{\Omega}.
$$

\section{Formal derivation of the ice internal temperature evolution model}
\label{Sec:1bis}

Let $\gls{c}>0$ be a constant denoting the \emph{ice specific heat}, and let $\gls{kappa}>0$ be a constant denoting the \emph{ice heat conductivity}. The internal ice temperature evolution is governed by the following advection--diffusion thermal equation (viz. equation~(5.105) in~\cite{GreveBlatter2009})
\begin{equation}
	\label{eq:thermal}
	\begin{aligned}
		&\rho c \left(\dfrac{\partial\Theta_{\textup{ice}}}{\partial t}+U_1(0,\cdot,\cdot)\dfrac{\partial\Theta_{\textup{ice}}}{\partial x_1}+U_2(0,\cdot,\cdot)\dfrac{\partial\Theta_{\textup{ice}}}{\partial x_2}+v_z(0,\cdot,\cdot)\dfrac{\partial\Theta_{\textup{ice}}}{\partial z}\right)\\
		&=\kappa\dfrac{\partial^2\Theta_{\textup{ice}}}{\partial z^2}+2(\rho g)^p A(x,z,\Theta_{\textup{ice}}) (h-z)^p |\nabla h|^p,
	\end{aligned}
\end{equation}
in $\Omega_{t}^{+}\times(0,h(t,x))$ for all $t\in (0,T)$. We observe that~\eqref{eq:thermal} is posed over a domain that depends on the ice thickness $h$, rendering the problem under consideration a \emph{moving boundary problem} (viz., e.g., \cite{CrankGupta1972,Ni2011}). The advection term is expressed in terms of the ice velocity, whose components are given in~\eqref{velocity} and~\eqref{vertical-velocity}.

We also assume, in agreement with Section~4.3.2 in~\cite{GreveBlatter2009}, that the viscosity has a finite limit as the effective stress tends to zero. This remark justifies the approximation:
\begin{equation}
	\label{regularisation}
	|\nabla h| \approx C_0, \quad \textup{ in } (0,T) \times\Omega_{t}^{+} \times (0, h(t,x)),
\end{equation}
for some $\gls{C0} > 0$. Conceptually, this constant plays a role analogous to the regularisation parameter $\sigma_0$ in the regularised Glen's flow law (cf. e.g., Section 4.3.2 in~\cite{GreveBlatter2009}), ensuring the stress term remains bounded. We note that while approximation~\eqref{regularisation} removes the spatial variability of the surface slope, it retains the essential vertical structure of the shear stress through the $(h-z)$ factor.

Since we are interested in studying the case where the bedrock is frozen, admissible ice internal temperatures vary in the following range:
\begin{equation}
	\label{const:temp}
	0\degree\textup{K}<\Theta_{\textup{ice}}(t,x,0)\le\Theta_{\textup{m}},\quad\textup{ for all } (t,x)\in [0,T]\times\overline{\Omega}.
\end{equation}

Let $\gls{Ts}:[0,T]\times\overline{\Omega}\to\mathbb{R}^+$ denote the \emph{surface ice temperature}, which is the temperature measured on the surface of the ice sheet, and whose unit measure is $\degree\textup{K}$. Measurements have shown that $\Theta_{\textup{surf}}$ can be well-approximated by the mean-annual surface air temperature, as long as the latter is not above $0\celsius$ (viz. page~67 in~\cite{GreveBlatter2009}). We thus require the following condition to hold:
\begin{equation}
	\label{bc:surf}
	\Theta_{\textup{ice}}(t,x,h(t,x))=\Theta_{\textup{surf}}(t,x),\quad\textup{ for all }(t,x)\in [0,T]\times\overline{\Omega_{t}^{+}}.
\end{equation}

Let us denote by $\gls{qgeoperp}$ a function in $L^2(\Omega)$ describing the \emph{geothermal heat flux} injected from the bedrock into the ice. This function is positive everywhere on the bedrock (viz., e.g., formula~(5.39) in~\cite{GreveBlatter2009}).

We denote by $\gls{nubase}$ the unit vector normal to the bedrock which points towards the portion of three-dimensional space that \emph{does not contain} the bedrock.
The surface ice temperature allows us to introduce a \emph{thermodynamical boundary condition} for $\Theta_{\textup{ice}}$ on the ice surface in contact with the air. For what concerns the ice surface in contact with the bedrock, the thermodynamical boundary condition that is to be considered in the case where the bedrock is frozen is given by (cf., e.g., \cite{GreveBlatter2009}):
\begin{equation}
	\label{bc:bedrock-1}
	\partial_{\bm{\nu}_{b}}\Theta_{\textup{ice}}(t,x,0)=\dfrac{q_{\textup{geo}}^\perp}{\kappa},\quad\textup{ for all }(t,x)\in [0,T]\times\overline{\Omega}.
\end{equation}

We observe that the partial differential equation governing the evolution of the internal ice temperature~\eqref{eq:thermal} is characterised by a coupling with the ice thickness that is \emph{stronger} than the coupling occurring in~\eqref{sie} when Glen's power law is in place (viz.~\eqref{flux-2}).

The initial condition the ice internal temperature has to satisfy takes the form
\begin{equation}
	\label{ic:thermal}
	\Theta_{\textup{ice}}(0,x,z)=\Theta_0(x,z), \quad\textup{ for all }(x,z)\in\overline{\Omega_{0}^{+}}\times [0,h_0(x)],
\end{equation}
where the \emph{initial shallow ice sheet internal temperature} \gls{T0} is given, and complies with~\eqref{const:temp} and~\eqref{bc:bedrock-1}.

The main difficulty in analysing~\eqref{eq:thermal} is due to the presence of a \emph{moving boundary}, represented by the ice surface $h$, and by the presence of a free boundary, represented by the set $\Omega_{t}^{+}$. In order to render this Partial Differential Equation more tractable from the analytical point of view, we make additional assumptions on the given data and we recast~\eqref{eq:thermal}--\eqref{ic:thermal} in a way that the temperature of the \emph{whole observed environment} is taken into account. The strategy we are going to employ is inspired by the \emph{Enthalpy Method}~\cite{Crank1984}.

Denote by $\gls{nuice}$ the outer unit normal vector applied along the surface $\{(x,h(t,x));x\in\overline{\Omega_{t}^{+}}\}$ and pointing towards to portion of space \emph{complementary to the one occupied by ice} at time $t\in[0,T]$. We require that there exists a positive constant $\gls{r0}$ such that:
\begin{equation}
	\label{r0}
	\textup{dist}(\partial\Omega_{t}^{+},\partial\Omega)>r_0,\quad\textup{ for all }t\in[0,T].
\end{equation}

Note that~\eqref{r0} introduces a \emph{restriction on the horizontal expansion} of the shallow ice sheet under consideration. More precisely, the assumption~\eqref{r0} means that $\textup{supp }h(t,\cdot)$ is compact in $\Omega$ for each $t\in [0,T]$. We denote by $L$ a positive constant, associated with the maximum ice surface elevation, that we will \emph{rigorously} determine in Section~\ref{Sec:2}.

Let $\gls{xi}=\xi(t,x,r)$ is a \emph{smooth cut-off function} varying between 0 and 1, which is equal to one for $0\le r\le \frac{r_0}{2}$ and is equal to zero for $\frac{3}{4}r_0\le r\le r_0$.

For a given $t\in[0,T]$ and $x\in\overline{\Omega}$, define an \emph{external auxiliary temperature} $\gls{Text}(t,x,\cdot):\{z\in\mathbb{R}^+\cup\{0\}; h(t,x)\le z\le L\}$, where we recall that $h(t,x)$ is identically zero for all $x\in\Omega_{t}^{-}$. For each $t\in [0,T]$, $x\in\overline{\Omega}$, and $0\le r\le r_0$, define:
\begin{equation}
	\label{theta-ext}
	\begin{aligned}
		&\Theta_{\textup{ext}}(t,(x,h(t,x))+r\bm{\nu}_{\textup{ice}}(t,x))\\
		&:=\xi(t,x,r)\left\{\Theta_{\textup{surf}}(t,x)+\left[\Theta_{\textup{surf}}(t,x)-\Theta_{\textup{ice}}(t,(x,h(t,x))-r\bm{\nu}_{\textup{ice}}(t,x))\right]\right\}.
	\end{aligned}
\end{equation}

Note that $\xi\equiv 1$ in a tubular neighbourhood of the ice surface $\{(x,h(t,x));x\in\overline{\Omega_{t}^{+}}\}$ of thickness $r_0$, and smoothly decreases to zero moving further away. Additionally, we observe that $\Theta_{\textup{ext}}(t,\cdot,\cdot)$ decreases moving far from the ice surface $\{(x,h(t,x));x\in\overline{\Omega_{t}^{+}}\}$.

Thanks to~\eqref{bc:surf}, it results that $\Theta_{\textup{ext}}$ satisfies the following conditions for all $(t,x)\in [0,T]\times\overline{\Omega_{t}^{+}}$:
\begin{gather}
	\Theta_{\textup{ext}}(t,x,h(t,x))=\Theta_{\textup{surf}}(t,x),\label{ice-surface}\\
	\partial_{\bm{\nu}_{\textup{ice}(t,x)}}\Theta_{\textup{ice}}(t,x,h(t,x))=\partial_{\bm{\nu}_{\textup{ice}(t,x)}}\Theta_{\textup{ext}}(t,x,h(t,x))=-\partial_{-\bm{\nu}_{\textup{ice}(t,x)}}\Theta_{\textup{ext}}(t,x,h(t,x)).\label{surf-decrease}
\end{gather}

Thanks to~\eqref{theta-ext}--\eqref{surf-decrease}, we infer that:
\begin{itemize}
	\item For each $t\in[0,T]$, the functions $\Theta_{\textup{ice}}(t,\cdot,\cdot)$ and $\Theta_{\textup{ext}}(t,\cdot,\cdot)$ are glued continuously along the surface $\{(x,h(t,x));x\in\overline{\Omega_{t}^{+}}\}$;
	\item For each $t\in[0,T]$, the normal derivatives of the functions $\Theta_{\textup{ice}}(t,\cdot,\cdot)$ and $\Theta_{\textup{ext}}(t,\cdot,\cdot)$ do not have jumps along the surface $\{(x,h(t,x));x\in\overline{\Omega_{t}^{+}}\}$;
	\item For each $t\in[0,T]$, the function $\Theta_{\textup{ext}}(t,\cdot,\cdot)$ has a controlled growth and decreases down to zero moving far from the ice surface. Note that the cut-off $\xi$ has indeed been introduced for this purpose.
\end{itemize}

Starting from the definition and properties of $\Theta_{\textup{ext}}$ listed beforehand (cf., \eqref{theta-ext}--\eqref{surf-decrease}), we construct an extension of $\Theta_{\textup{ext}}$ to the region containing ice, and we denote one such extension by $\gls{Tint}$.
The function $\Theta_{\textup{int}}$ is defined in a way that, in the region with ice, it is less or equal than $\Theta_{\textup{ice}}$ at each time. For each $t\in [0,T]$, for each $x\in\overline{\Omega_{t}^{+}}$, and for each $0\le r \le r_0$ we define $\Theta_{\textup{int}}$ close to the ice profile as follows
\begin{equation}
	\label{theta-int}
	\Theta_{\textup{int}}(t,(x,h(t,x))-r\bm{\nu}_{\textup{ice}(t,x)})=\xi(t,x,r)\Theta_{\textup{ice}}(t,(x,h(t,x))-r\bm{\nu}_{\textup{ice}(t,x)}),
\end{equation}
where $\xi=\xi(t,x,z)$ is the smooth cut-off function introduced beforehand, that varies between 0 and 1, and that is equal to one for $0\le r\le \frac{r_0}{2}$ and is equal to zero for $\frac{3}{4}r_0\le r\le r_0$.

\begin{remark}
	\label{rem:cutoff}
	We notice that this construction holds for ice sheets exhibiting vertical slopes at the ice margins (i.e., close to $\Gamma_{f,t}$) like, for instance, the Bueler and Vialov profiles. However, the construction can be adapted to more general ice surfaces by taking $\Theta_{\textup{int}}\equiv\Theta_{\textup{surf}}$ near the ice margins and applying another cut-off.
	\bqed
\end{remark}

Thanks to~\eqref{theta-int}, we are able to infer that $\Theta_{\textup{int}} \equiv 0$ in the interior of the ice region \emph{sufficiently far away} from the ice surface, so that it results:
\begin{equation}
	\label{theta-int-3}
	\Theta_{\textup{ice}}(t,x,z)\ge\Theta_{\textup{int}}(t,x,z),\quad\textup{ for all }t\in[0,T] \textup{ for all }x\in\overline{\Omega_{t}^{+}} \textup{ and for all }0\le z\le h(t,x).
\end{equation}

For each $t\in [0,T]$, we define the following \emph{auxiliary temperature}
\begin{equation*}
	\gls{Taux}(t,x,z):=
	\begin{cases}
		\Theta_{\textup{int}}(t,x,z)&,\quad\textup{ if } x\in\overline{\Omega_{t}^{+}} \textup{ and }0\le z\le h(t,x),\\
		\Theta_{\textup{ext}}(t,x,z)&,\quad\textup{ if } x\in\overline{\Omega_{t}^{-}} \textup{ or } z>h(t,x).
	\end{cases}
\end{equation*}

Thanks to~\eqref{surf-decrease}, it is then straightforward to observe that for each $t\in[0,T]$:
\begin{gather}
	\Theta_{\textup{aux}}(t,x,h(t,x))=\Theta_{\textup{surf}}(t,x), \textup{ for all }x\in\overline{\Omega_{t}^{+}},\label{th-aux-1}\\
	\begin{aligned}
		&\partial_{\bm{\nu}_{\textup{ice}(t,x)}}\Theta_{\textup{aux}}(t,x,h(t,x))=\partial_{\bm{\nu}_{\textup{ice}(t,x)}}\Theta_{\textup{ext}}(t,x,h(t,x))\\
		&=\partial_{\bm{\nu}_{\textup{ice}(t,x)}}\Theta_{\textup{ice}}(t,x,h(t,x)), \textup{ for all }x\in\overline{\Omega_{t}^{+}}.\label{th-aux-2}
	\end{aligned}
\end{gather}

The latter conditions mean that, at each $t\in [0,T]$, the function $\Theta_{\textup{aux}}$ is glued continuously along the curve $(x,h(t,x))$, $x\in\overline{\Omega_{t}^{+}}$, that the normal derivative of $\Theta_{\textup{aux}}$ has no jumps along the ice surface.

We now define the mapping $\gls{Theta}:[0,T]\times\overline{\Omega}\times[0,L]\to\mathbb{R}^+\cup\{0\}$ as follows
\begin{equation}
	\label{Theta-final-1}
	\Theta(t,x,z):=
	\begin{cases}
		\Theta_{\textup{ice}}(t,x,z)&, \quad\textup{ if } t\in[0,T], x\in\overline{\Omega_{t}^{+}} \textup{ and }0\le z\le h(t,x),\\
		\Theta_{\textup{ext}}(t,x,z)&,\quad\textup{ if } t\in[0,T], x\in\overline{\Omega_{t}^{-}} \textup{ or } z>h(t,x),
	\end{cases}
\end{equation}
and we define the mapping $\gls{Thetatilde}:[0,T]\times\overline{\Omega}\times[0,L]\to\mathbb{R}^+\cup\{0\}$ as follows:
\begin{equation}
	\label{Theta-final-2}
	\tilde{\Theta}(t,x,z):=\Theta(t,x,z)-\Theta_{\textup{aux}}(t,x,z),\quad\textup{ for all } (t,x,z)\in [0,T]\times\overline{\Omega}\times [0,L].
\end{equation}

We then define the initial condition for $\tilde{\Theta}$ as follows:
\begin{equation}
	\label{ic:thermal-2}
	\gls{T0tilde}(x,z):=\tilde{\Theta}(0,x,z)=
	\begin{cases}
		\Theta(0,x,z)-\Theta_{\textup{aux}}(0,x,z)&,\textup{ if }0<z<h_0(x),\\
		0&,\textup{ otherwise}.
	\end{cases}
\end{equation}

\begin{remark}
	\label{rem:2}
	Observe that one could forge an \emph{ad hoc} $\tilde{\Theta}_0$ that vanishes near the ice surface if $h_0$ describes, for instance, the shape of a Bueler profile or a Vialov profile.
	\bqed
\end{remark}

Thanks to the properties of $\Theta_{\textup{ext}}$ (see~\eqref{theta-ext}--\eqref{surf-decrease}), the properties of $\Theta_{\textup{int}}$ (see~\eqref{theta-int}--\eqref{theta-int-3}), and the properties of $\Theta_{\textup{aux}}$ (see~\eqref{th-aux-1}--\eqref{Theta-final-2}), it results:
\begin{gather}
	0\le\tilde{\Theta}(t,x,z)\le\Theta_{\textup{m}},\quad\textup{ for all }(t,x,z)\in [0,T]\times\overline{\Omega}\times [0,L],\label{Theta-final-3}\\
	\tilde{\Theta}(t,x,h(t,x))= 0,\quad\textup{ for all }(t,x)\in [0,T]\times\overline{\Omega_{t}^{+}},\label{Theta-final-4}\\
	\tilde{\Theta}(t,x,z)= 0,\quad\textup{ for all }(t,x,z)\in [0,T]\times\partial\Omega\times [0,L],\label{Theta-final-5}\\
	\tilde{\Theta}(t,x,L)= 0,\quad\textup{ for all }(t,x)\in [0,T]\times\overline{\Omega},\label{Theta-final-6}\\
	\partial_{\bm{\nu}_b}\tilde{\Theta}(t,x,0)= \dfrac{q_{\textup{geo}}^\perp}{\kappa},\quad\textup{ for all }(t,x)\in [0,T]\times\overline{\Omega},\label{Theta-final-6-bis}\\
	\partial_{\bm{\nu}_{\textup{ice}(t,x)}}\tilde{\Theta}(t,x,h(t,x))= 0,\quad\textup{ for all }(t,x)\in [0,T]\times\overline{\Omega_{t}^{+}}.\label{Theta-final-7}
\end{gather}

Note that the Neumann boundary condition~\eqref{Theta-final-6-bis} accounts for an error deriving from the tubular neighbourhood of width $r_0$ associated with $\Theta_{\textup{aux}}$. Since $r_0$ is small, it is licit to perform one such approximation.

In light of~\eqref{Theta-final-2}, we recast the left-hand side of the thermal equation~\eqref{eq:thermal} in terms of $\tilde{\Theta}$
\begin{equation}
	\label{eq:thermal-3}
	\begin{aligned}
		&\rho c\left(\dfrac{\partial\tilde{\Theta}}{\partial t}+U_1(0,\cdot,\cdot)\dfrac{\partial\tilde{\Theta}}{\partial x_1}+U_2(0,\cdot,\cdot)\dfrac{\partial\tilde{\Theta}}{\partial x_2}+v_z(0,\cdot,\cdot)\dfrac{\partial\tilde{\Theta}}{\partial z}\right)-\kappa\Delta\tilde{\Theta}\\
		&=2(\rho g C_0)^p A(x,z,\tilde{\Theta}) \left(\{h-z\}^{+}\right)^p,
	\end{aligned}
\end{equation}
for all $(t,x,z)\in [0,T]\times\overline{\Omega}\times[0,L]$.

\begin{remark}[Removal of the corrector near the ice surface]
	\label{rem:3}
	Rather than extending~\eqref{eq:thermal} rigorously to the fixed domain, we are going to define a new approximate model on the fixed cylinder $\Omega \times (0, L)$. The temperature field $\tilde{\Theta}$ satisfies~\eqref{eq:thermal-3} on the whole domain, with the source term supported only where ice is present. The moving boundary conditions are encoded in the auxiliary field $\Theta_{\textup{aux}}$, whose mismatch with the true solution is confined to a boundary layer of thickness $r_0$. For $r_0$ sufficiently small compared to the ice thickness, the bulk temperature is well-approximated. The dependence of the approximate model on the parameter $r_0$ still remains, as the initial condition is chosen in a way that it complies with the construction performed beforehand. The sharp-interface model is recovered, in the sense of formal asymptotics, in the limit $r_0 \to 0$, though a rigorous justification of this limit is beyond the scope of the present work.
	\bqed
\end{remark}

We now recover a suitable variational formulation associated with~\eqref{eq:thermal}, \eqref{Theta-final-3}--\eqref{Theta-final-7}, and the initial condition~\eqref{ic:thermal-2}. Before that, it is important to point out that the problem under consideration does not fall into the categories considered in the literature of moving interface problems. 

\begin{remark}[Classification of the variational problem]
	\label{rem:4}
	The problem under consideration is \emph{not} a Stefan problem~\cite{Crank1984}.
	In the prototypical Stefan problem, the moving boundary velocity appears \emph{directly} in the thermal boundary condition known as \emph{Stefan condition}.
	
	In the problem we are considering here, the surface motion comes from solving the shallow ice equation~\eqref{sie}, which is a mechanical model coupled with a thermodynamical model, and \emph{the Stefan condition is absent}. The surface boundary condition for temperature remains a Dirichlet or Neumann condition, but the location of that boundary evolves primarily as a result of snow accumulation or ablation.
	
	The problem here considered is also of a different category from the problems discussed in the literature connected with the Cahn--Hilliard equation (see, for instance, \cite{GGG17,GGM17,GGH22}). Indeed, in the problem we are considering here,only one phase is taken into account, in the sense that \emph{no model is prescribed in the portion of space complementary to the region occupied by ice}. In the problem considered here, the ice surface moves because mass is added or removed, or because ice flows at a rate depending on the ice temperature. In models of Cahn--Hilliard type, the interface is typically regarded as a known datum of the problem (cf., e.g., \cite{AM21-part1,AM21-part2,AGP25,CEGP23}). Besides, the choice of considering an elliptic equation over moving domains is crucial to derive the energy estimates leading to the sharp interface model as a result of a formal asymptotic analysis~\cite{AM21-part1}. In the model considered here, instead, the evolution of the ice surface (which plays a similar role to the one of the moving interfaces considered in~\cite{AM21-part1,AM21-part2,AGP25,CEGP23}) is an unknown of the problem that we find by solving the evolution equation~\eqref{vi:ice}.
	\bqed
\end{remark}

For each $t\in [0,T]$, define the open set $\gls{Dice}\subset\Omega\times (0,L)$ as the \emph{portion of three-dimensional space occupied by ice at time $t$}, namely:
\begin{equation*}
	D_{\textup{ice}}(t):=\{(x,z)\in\Omega\times (0,L);0< z< h(t,x)\}.
\end{equation*}

Let us denote by $(D_{\textup{ice}}(t))^{\textup{c}}$ the complementary of the set $D_{\textup{ice}}(t)$ in $\overline{\Omega}\times [0,L]$. Observe that, since $\tilde{\Theta}\equiv 0$ in $(D_{\textup{ice}}(t))^{\textup{c}}$, we can extend $U_1(0,\cdot,\cdot)$, $U_2(0,\cdot,\cdot)$ and $v_z(0,\cdot,\cdot)$ continuously in $(D_{\textup{ice}}(t))^{\textup{c}}$.

Consider a test function $\Xi:[0,T]\times\overline{\Omega}\times [0,L]\to[0,\Theta_{\textup{m}}]$ satisfying~\eqref{Theta-final-3}, \eqref{Theta-final-5}, \eqref{Theta-final-6}, and test the left-hand side of~\eqref{eq:thermal-3} at $(\Xi-\tilde{\Theta})$. On the one hand, the fact that $\tilde{\Theta}(t,\cdot,\cdot)\equiv 0$ in $(D_{\textup{ice}}(t))^{\textup{c}}$, the fact that $\frac{\partial\tilde{\Theta}}{\partial t}(t,\cdot)\ge 0$ in $(D_{\textup{ice}}(t))^{\textup{c}}$, and the fact that $\Xi(t,\cdot,\cdot)\ge 0$ in $(D_{\textup{ice}}(t))^{\textup{c}}$ modify~\eqref{eq:thermal-3} into:
\begin{align*}
	&\rho c\int_{0}^{L}\int_{\Omega}\left(\dfrac{\partial\tilde{\Theta}}{\partial t}+U_1(0,\cdot,\cdot)\dfrac{\partial\tilde{\Theta}}{\partial x_1}+U_2(0,\cdot,\cdot)\dfrac{\partial\tilde{\Theta}}{\partial x_2}+v_z(0,\cdot,\cdot)\dfrac{\partial\tilde{\Theta}}{\partial z}\right) (\Xi-\tilde{\Theta}) \dd x \dd z\\
	&\quad-\kappa\int_{0}^{L}\int_{\Omega}\Delta\tilde{\Theta} (\Xi-\tilde{\Theta}) \dd x \dd z\\
	&\ge\rho c\iint_{D_{\textup{ice}}(t)}\left(\dfrac{\partial\tilde{\Theta}}{\partial t}+U_1(0,\cdot,\cdot)\dfrac{\partial\tilde{\Theta}}{\partial x_1}+U_2(0,\cdot,\cdot)\dfrac{\partial\tilde{\Theta}}{\partial x_2}+v_z(0,\cdot,\cdot)\dfrac{\partial\tilde{\Theta}}{\partial z}\right) (\Xi-\tilde{\Theta}) \dd x \dd z\\
	&\quad-\kappa\iint_{D_{\textup{ice}}(t)}\Delta\tilde{\Theta} (\Xi-\tilde{\Theta}) \dd x \dd z\\
	&=2(\rho g C_0)^p\iint_{D_{\textup{ice}}(t)}A(\cdot,\cdot,\tilde{\Theta})(h-z)^p (\Xi-\tilde{\Theta})\dd x \dd z\\
	&=2(\rho g C_0)^p\int_{0}^{L}\int_{\Omega}A(\cdot,\cdot,\tilde{\Theta})(\{h-z\}^{+})^p (\Xi-\tilde{\Theta})\dd x \dd z.
\end{align*}

On the one hand, an application of the positiveness of $q_{\textup{geo}}^\perp$ gives that for all $t\in (0,T)$:
\begin{align*}
	&\rho c\int_{0}^{L}\int_{\Omega}\left(\dfrac{\partial\tilde{\Theta}}{\partial t}+U_1(0,\cdot,\cdot)\dfrac{\partial\tilde{\Theta}}{\partial x_1}+U_2(0,\cdot,\cdot)\dfrac{\partial\tilde{\Theta}}{\partial x_2}+v_z(0,\cdot,\cdot)\dfrac{\partial\tilde{\Theta}}{\partial z}\right) (\Xi-\tilde{\Theta}) \dd x \dd z\\
	&\quad-\kappa\int_{0}^{L}\int_{\Omega}\Delta\tilde{\Theta} (\Xi-\tilde{\Theta}) \dd x \dd z\\
	&=\rho c\iint_{D_{\textup{ice}}(t)}\left(\dfrac{\partial\tilde{\Theta}}{\partial t}+U_1(0,\cdot,\cdot)\dfrac{\partial\tilde{\Theta}}{\partial x_1}+U_2(0,\cdot,\cdot)\dfrac{\partial\tilde{\Theta}}{\partial x_2}+v_z(0,\cdot,\cdot)\dfrac{\partial\tilde{\Theta}}{\partial z}\right) (\Xi-\tilde{\Theta}) \dd x \dd z\\
	&\quad+\rho c\iint_{(D_{\textup{ice}}(t))^{\textup{c}}}\left(\dfrac{\partial\tilde{\Theta}}{\partial t}+U_1(0,\cdot,\cdot)\dfrac{\partial\tilde{\Theta}}{\partial x_1}+U_2(0,\cdot,\cdot)\dfrac{\partial\tilde{\Theta}}{\partial x_2}+v_z(0,\cdot,\cdot)\dfrac{\partial\tilde{\Theta}}{\partial z}\right) (\Xi-\tilde{\Theta}) \dd x \dd z\\
	&\quad-\kappa\iint_{D_{\textup{ice}}(t)}\Delta\tilde{\Theta} (\Xi-\tilde{\Theta}) \dd x \dd z-\kappa\iint_{(D_{\textup{ice}}(t))^{\textup{c}}}\Delta\tilde{\Theta} (\Xi-\tilde{\Theta}) \dd x \dd z\\
	&=\rho c\iint_{D_{\textup{ice}}(t)}\left(\dfrac{\partial\tilde{\Theta}}{\partial t}+U_1(0,\cdot,\cdot)\dfrac{\partial\tilde{\Theta}}{\partial x_1}+U_2(0,\cdot,\cdot)\dfrac{\partial\tilde{\Theta}}{\partial x_2}+v_z(0,\cdot,\cdot)\dfrac{\partial\tilde{\Theta}}{\partial z}\right) (\Xi-\tilde{\Theta}) \dd x \dd z\\
	&\quad+\rho c\iint_{(D_{\textup{ice}}(t))^{\textup{c}}}\left(\dfrac{\partial\tilde{\Theta}}{\partial t}+U_1(0,\cdot,\cdot)\dfrac{\partial\tilde{\Theta}}{\partial x_1}+U_2(0,\cdot,\cdot)\dfrac{\partial\tilde{\Theta}}{\partial x_2}+v_z(0,\cdot,\cdot)\dfrac{\partial\tilde{\Theta}}{\partial z}\right) (\Xi-\tilde{\Theta}) \dd x \dd z\\
	&\quad+\kappa\iint_{D_{\textup{ice}}(t)}\nabla\tilde{\Theta}\cdot\nabla(\Xi-\tilde{\Theta})\dd x \dd z+\kappa\iint_{(D_{\textup{ice}}(t))^{\textup{c}}}\underbrace{\nabla\tilde{\Theta}}_{=\bm{0} \textup{ by }\eqref{Theta-final-2}}\cdot\nabla(\Xi-\tilde{\Theta})\dd x \dd z\\
	&\quad+\int_{\Omega^{+}_t}q_{\textup{geo}}^\perp (\Xi(\cdot,\cdot,0)-\tilde{\Theta}(\cdot,\cdot,0))\dd x\\
	&\quad-\kappa\int_{\Omega^{+}_t}\underbrace{\nabla\tilde{\Theta}(\cdot,\cdot,h(\cdot,\cdot))\cdot\bm{\nu}_{\textup{ice}}}_{=0 \textup{ by }\eqref{Theta-final-7}}\left(\Xi(\cdot,\cdot,h(\cdot,\cdot))-\tilde{\Theta}(\cdot,\cdot,h(\cdot,\cdot))\right) \dd x\\
	&\quad-\kappa\int_{0}^{L}\int_{\partial\Omega}\nabla\tilde{\Theta}(\cdot,\cdot,z)\cdot\bm{\nu}_{\partial\Omega}(\underbrace{\Xi(\cdot,\cdot,z)-\tilde{\Theta}(\cdot,\cdot,z)}_{=0 \textup{ by }\eqref{Theta-final-5}})\dd\mathcal{H}^1\dd z\\
	&\quad-\kappa\int_{\Omega} \dfrac{\partial\tilde{\Theta}}{\partial z}(\cdot,\cdot, L)(\underbrace{\Xi(\cdot,\cdot, L)}_{=0 \textup{ by }\eqref{Theta-final-6}}-\underbrace{\tilde{\Theta}(\cdot,\cdot,L)}_{=0 \textup{ by }\eqref{Theta-final-6}})\dd x\\
	&\quad+\kappa\int_{\Omega^{-}_t\times\{0\}}\underbrace{\nabla\tilde{\Theta}(\cdot,\cdot,0)\cdot\bm{\nu}_b}_{=0 \textup{ by }\eqref{Theta-final-2}}\left(\Xi(\cdot,\cdot,0)-\tilde{\Theta}(\cdot,\cdot,0)\right) \dd x\\
	&\quad+\kappa\int_{\Omega^{+}_t}\underbrace{\nabla\tilde{\Theta}(\cdot,\cdot,h(\cdot,\cdot))\cdot\bm{\nu}_{\textup{ice}}}_{=0 \textup{ by }\eqref{Theta-final-7}}\left(\Xi(\cdot,\cdot,h(\cdot,\cdot))-\tilde{\Theta}(\cdot,\cdot,h(\cdot,\cdot))\right) \dd x\\
	&\le\rho c\int_{0}^{L}\int_{\Omega}\left(\dfrac{\partial\tilde{\Theta}}{\partial t}+U_1(0,\cdot,\cdot)\dfrac{\partial\tilde{\Theta}}{\partial x_1}+U_2(0,\cdot,\cdot)\dfrac{\partial\tilde{\Theta}}{\partial x_2}+v_z(0,\cdot,\cdot)\dfrac{\partial\tilde{\Theta}}{\partial z}\right) (\Xi-\tilde{\Theta}) \dd x \dd z\\
	&\quad+\kappa\int_{0}^{L}\int_{\Omega}\nabla\tilde{\Theta}\cdot\nabla(\Xi-\tilde{\Theta})\dd x \dd z+\int_{\Omega}q_{\textup{geo}}^\perp(\Xi(\cdot,\cdot,0)-\tilde{\Theta}(\cdot,\cdot,0))\dd x.
\end{align*}

In conclusion, the evolution of the ice temperature is a solution of the following variational problem: \emph{Find $\tilde{\Theta}:[0,T]\times\overline{\Omega}\times [0,L]\to[0,\Theta_{\textup{m}}]$ satisfying the variational inequalities
\begin{equation*}
	\begin{aligned}
		&\rho c\int_{0}^{L}\int_{\Omega}\left(\dfrac{\partial\tilde{\Theta}}{\partial t}+U_1(0,\cdot,\cdot)\dfrac{\partial\tilde{\Theta}}{\partial x_1}+U_2(0,\cdot,\cdot)\dfrac{\partial\tilde{\Theta}}{\partial x_2}+v_z(0,\cdot,\cdot)\dfrac{\partial\tilde{\Theta}}{\partial z}\right) (\Xi-\tilde{\Theta}) \dd x \dd z\\
		&\quad+\kappa\int_{0}^{L}\int_{\Omega}\nabla\tilde{\Theta}\cdot\nabla(\Xi-\tilde{\Theta})\dd x \dd z+\int_{\Omega}q_{\textup{geo}}^\perp(\Xi(\cdot,\cdot,0)-\tilde{\Theta}(\cdot,\cdot,0))\dd x\\
		&\ge 2(\rho g C_0)^p\int_{0}^{L}\int_{\Omega}A(\cdot,\cdot,\tilde{\Theta})\left(\left\{u^{\frac{p-1}{2p}}-z\right\}^{+}\right)^p (\Xi-\tilde{\Theta})\dd x \dd z,
	\end{aligned}
\end{equation*}
for all $\Xi:[0,T]\times\overline{\Omega}\times [0,L]\to\mathbb{R}$ such that $0\le\Xi\le\Theta_{\textup{m}}$, the boundary conditions
\begin{equation*}
	\begin{aligned}
		\tilde{\Theta}=0&,\quad\textup{ on }[0,T]\times\partial\Omega\times [0,L],\\
		\tilde{\Theta}(\cdot,\cdot,L)=0&,\quad\textup{ on }[0,T]\times\overline{\Omega},
	\end{aligned}
\end{equation*}
and the initial condition:
\begin{equation*}
	\tilde{\Theta}(0,\cdot,\cdot)=\tilde{\Theta}_0(\cdot,\cdot),\quad\textup{ in }\overline{\Omega}\times [0,L].
\end{equation*}
}

Finally, we are able to write down the formal variational formulation of the thermo-mechanical problem under consideration: \emph{Find $u:[0,T]\times\overline{\Omega}\to\mathbb{R}^{+}\cup\{0\}$ and $\tilde{\Theta}:[0,T]\times \overline{\Omega}\times [0,L]\to[0,\Theta_{\textup{m}}]$ satisfying the following variational inequalities}
\begin{gather}
	\begin{aligned}
		&\int_{0}^{T} \int_{\Omega} \left(\dfrac{\partial}{\partial t}(|u|^{\frac{3p-1}{2p}-2} \, u)\right)(v-u) \dd x \dd t\\
		&\quad -\int_{0}^{T} \int_{\Omega} \mu(x,u,\Theta_{\textup{ice}}) |\nabla u|^{p-2} \nabla u \cdot \nabla(v-u) \dd x\dd t\\
		&\ge \int_{0}^{T} \int_{\Omega} \tilde{a} (v-u) \dd x \dd t,
	\end{aligned}\label{weak-u}\\[3ex]
	\begin{aligned}
		&\rho c\int_{0}^{T}\int_{0}^{L}\int_{\Omega}\left(\dfrac{\partial\tilde{\Theta}}{\partial t}+U_1(0,\cdot,\cdot)\dfrac{\partial\tilde{\Theta}}{\partial x_1}+U_2(0,\cdot,\cdot)\dfrac{\partial\tilde{\Theta}}{\partial x_2}+v_z(0,\cdot,\cdot)\dfrac{\partial\tilde{\Theta}}{\partial z}\right) (\Xi-\tilde{\Theta}) \dd x \dd z \dd t\\
		&\quad+\kappa\int_{0}^{T}\int_{0}^{L}\int_{\Omega}\nabla\tilde{\Theta}\cdot\nabla(\Xi-\tilde{\Theta})\dd x \dd z \dd t+\int_{0}^{T}\int_{\Omega}q_{\textup{geo}}^\perp(\Xi-\tilde{\Theta})\dd x \dd t\\
		&\ge 2(\rho g C_0)^p\int_{0}^{T}\int_{0}^{L}\int_{\Omega}A(\cdot,\cdot,\tilde{\Theta})\left(\left\{u^{\frac{p-1}{2p}}-z\right\}^{+}\right)^p (\Xi-\tilde{\Theta})\dd x \dd z \dd t,
	\end{aligned}\label{vi:thermal}
\end{gather}
\emph{for all $v:[0,T]\times\overline{\Omega}\to\mathbb{R}^{+}\cup\{0\}$ and for all $\Xi:[0,T]\times\overline{\Omega}\times [0,L]\to[0,\Theta_{\textup{m}}]$, the boundary conditions}
\begin{gather}
	u=0,\quad\textup{ on } [0,T]\times\partial\Omega,\label{u:bdry}\\
	\tilde{\Theta}=0,\quad\textup{ on }[0,T]\times\partial\Omega\times [0,L],\label{theta:bdry:1}\\
	\tilde{\Theta}(\cdot,\cdot,L)=0,\quad\textup{ on }[0,T]\times\overline{\Omega},\label{theta:bdry:2}
\end{gather}
\emph{as well as the initial conditions}
\begin{gather}
	u(0,\cdot)=u_0(\cdot)=[h_0(\cdot)]^{\frac{2p}{p-1}},\quad\textup{ in }\overline{\Omega},\label{u:ic}\\
	\tilde{\Theta}(0,\cdot,\cdot)=\tilde{\Theta}_0(\cdot,\cdot),\quad\textup{ in }\overline{\Omega}\times [0,L],\label{theta:ic}
\end{gather}
\emph{where $h_0>0$ and $0\le\tilde{\Theta}_0\le\Theta_{\textup{m}}$ are given (viz. Remark~\ref{rem:2}).}

Once again, note that, at least for the time being, a rigorous concept of solution has not been defined yet. This task will be carried out in Section~\ref{Sec:2}, in which the function spaces where the solutions are going to be sought will be defined as well as the requirements a function has to meet in order to be regarded as a solution.

\section{Weak formulation of the unilateral boundary value problem}
\label{Sec:2}

Let $\Omega \subset \mathbb{R}^2$ be a Lipschitz domain. Let the parameter $p$ satisfy $2.8 \le p \le 5$, as suggested by experimental results (cf., e.g., \cite{Glen1970,GreveBlatter2009,JouvBuel2012,PT23}).
For simplicity, we adopt the following assumptions (viz. Sections~\ref{Sec:1} and~\ref{Sec:1bis}):
\begin{enumerate}[$(H1)$]
	\item The bedrock $b$ is flat in $\overline{\Omega}$;\label{H1}
	\item The function $\tilde{a}$ is independent of the ice sheet height and belongs to $W^{1,p,\infty}(0,T;\mathcal{C}^0(\overline{\Omega}))$;\footnote{Note that $(H\ref{H3})$ is a particular case of~\eqref{a:continuous}.}\label{H3}
	\item The initial ice velocity $(\bm{U}(0),v_z(0))$ is of class $\bm{\mathcal{C}}^0(\overline{\Omega}\times[0,L])$ and is non-identically zero.\label{H4}
\end{enumerate}

Note that assumption $(H\ref{H1})$ implies that $\bm{U}_b = \bm{0}$.
For brevity, define the number
\begin{equation}
	\label{alpha}
	\gls{alpha}:=\dfrac{3p-1}{2p},
\end{equation}
and note that $1 < \alpha < 2 < p+1$. By the Rellich--Kondra\v{s}ov theorem, the following chain of embeddings holds:
\begin{equation*}
	\label{emb}
	W_0^{1,p}(\Omega) \hookrightarrow\hookrightarrow \mathcal{C}^0(\overline{\Omega}) \hookrightarrow L^\alpha(\Omega).
\end{equation*}

Define the number $\gls{L}$ as the \emph{theoretical maximum height a shallow ice sheet can reach}. Several physical mechanisms impose certain bounds for the number $L$. Atmospherically, the saturation vapour pressure of water decreases exponentially with temperature via the Clausius--Clapeyron relation \cite{WH06}.
The Antarctic Ice Sheet does not exceed $5\,\text{km}$ in thickness~\cite{Fretwell2013}. Extrapolating these trends, an ice sheet growing substantially beyond $5\,\text{km}$ would encounter an atmosphere nearly devoid of moisture, making further snowfall-driven growth physically impossible.

Define the set
\begin{equation*}
	\label{Ksurf}
	\gls{Ksurf}:=\{v \in W_0^{1,p}(\Omega); v \ge 0 \textup{ in }\overline{\Omega}\},
\end{equation*}
and observe that no upper bound has been imposed, despite the number $L$ having been introduced precisely as a theoretical maximum thickness. The reason for this is twofold.

First, the ice surface elevation $h$ (and, consequently, the transformed variable $u$) is independent of the vertical coordinate $z$. Indeed, since $L$ denotes the height of the cylinder $\Omega\times (0,L)$ where the ice internal temperature is defined, it results that $L$ does not constrain the admissible values of $u$  mathematically.

Second, the \emph{physical} impossibility of an ice sheet exceeding the height $L$ means that the temperature model is not expected to produce meaningful values for $z>L$.

Imposing an upper bound on $K_\textup{surf}$ is thus unnecessary both for physical reasons and for mathematical consistency of the model.

Define the function space
\begin{equation*}
	\gls{V}:=\{\Xi\in H^1(\Omega\times(0,L));\Xi=0 \textup{ on }\partial\Omega\times (0,L) \cup \Omega\times\{L\}\},
\end{equation*}
and define the set
\begin{equation*}
	\label{Ktemp}
	\gls{Ktemp}:=\{\Xi\in V;0\le\Xi\le\Theta_{\textup{m}} \textup{ a.e. in }\Omega\times (0,L)\}.
\end{equation*}

The weak formulation of the unilateral boundary value problem and the definition of a solution are derived via a constructive proof based on the penalty method. The use of the penalty method for proving the existence of solutions to time-dependent contact problems was introduced by J.-L. Lions~\cite{Lions1969}. We now generalise a preparatory result, originally established in~\cite{PT23} for the space $L^2$.
\begin{lemma}
	\label{lem:1}
	Let $I$ be a bounded open interval in $\mathbb{R}$, and let $\mathscr{O} \subset \mathbb{R}^m$, with $m\ge 1$ an integer, be a bounded open set. Let $2\le q\le\infty$.
	Then, the following results hold:
	\begin{itemize}
		\item[$(a)$] The operator $\gls{neg-part}:L^q(\mathscr{O}) \to L^q(\mathscr{O})$ defined by
		$$
		f \in L^q(\mathscr{O}) \mapsto -\{f\}^{-}:=\min\{f,0\} \in L^q(\mathscr{O}),
		$$
		is monotone, bounded and Lipschitz continuous with Lipschitz constant equal to one;
		
		\item[$(b)$] The Nemitskii operator associated with $-\{\cdot\}^{-}$, which is defined by
		\begin{equation*}
			-\{f\}^{-}(t):=-\{f(t)\}^{-},\quad\textup{ for each }f\in \mathcal{C}^0(\overline{I};L^q(\mathscr{O})),
		\end{equation*}
		is such that
		\begin{equation*}
			-\{\cdot\}^{-}:\mathcal{C}^0(\overline{I};L^q(\mathscr{O}))\to\mathcal{C}^0(\overline{I};L^q(\mathscr{O}))
		\end{equation*}
		is continuous.
	\end{itemize}
\end{lemma}
\begin{proof}
	Observe that, for all measurable functions $f:\mathscr{O}\to\mathbb{R}$ and all $g:\mathscr{O}\to\mathbb{R}$, the following estimate holds
	\begin{equation}
		\label{neg-part-pw-1}
		\left|\left(-\{f\}^{-}(x)\right)-\left(-\{g\}^{-}(x)\right)\right|=\left|\left(-\{f(x)\}^{-}\right)-\left(-\{g(x)\}^{-}\right)\right|\le |f(x)-g(x)|,
	\end{equation}
	for a.a. $x\in\mathscr{O}$. Additionally, it results that
	\begin{equation}
		\label{neg-part-pw-2}
		\left(\left(-\{f(x)\}^{-}\right)-\left(-\{g(x)\}^{-}\right)\right)(f(x)-g(x))\ge\left(\{f(x)\}^{-}-\{g(x)\}^{-}\right)^2,
	\end{equation}
	for a.a. $x\in\mathscr{O}$.
	The rest of the proof is divided into two parts, numbered $(i)$ and $(ii)$.
	
	$(i)$ \emph{Proof of $(a)$}. To establish the boundedness of $-\{\cdot\}^{-}$, let $\mathscr{F}\subset L^q(\mathscr{O})$ be a bounded set. If $2\le q<\infty$, we observe that~\eqref{neg-part-pw-1} gives:
	\begin{equation*}
		\left(\int_{\mathscr{O}}|-\{f\}^{-}|^q\dd x\right)^{1/q}\le\|f\|_{L^q(\mathscr{O})},\quad\textup{ for all }f\in\mathscr{F}.
	\end{equation*}
	
	If $q=\infty$, we obtain that:
	\begin{equation*}
		\sup_{x\in\mathscr{O}}|-\{f(x)\}^{-}|\le\|f\|_{L^\infty(\mathscr{O})},\quad\textup{ for all }f\in\mathscr{F}.
	\end{equation*}
	
	The Lipschitz continuity of $(-\{\cdot\}^{-})$ follows by~\eqref{neg-part-pw-1}; indeed:
	\begin{equation*}
		\left\|(-\{f\}^{-})-(-\{g\}^{-})\right\|_{L^q(\mathscr{O})}\le\|f-g\|_{L^q(\mathscr{O})}, \quad\textup{ for all }f,g\in L^q(\mathscr{O}).
	\end{equation*}
	
	To establish the monotonicity of $-\{\cdot\}^{-}$, let us first observe that if $2\le q\le \infty$ then the assumed boundedness of $\mathscr{O}$ allows us to infer that
	\begin{equation*}
		-\{f\}^{-}g\in L^1(\mathscr{O}),\quad\textup{ for all }f,g\in L^q(\mathscr{O}).
	\end{equation*}
	
	Combining the latter with~\eqref{neg-part-pw-2} gives
	\begin{equation*}
		\int_{\mathscr{O}}\left(\left(-\{f\}^{-}\right)-\left(-\{g\}^{-}\right)\right)(f-g)\dd x\ge\left\|\left(-\{f\}^{-}\right)-\left(-\{g\}^{-}\right)\right\|_{L^2(\mathscr{O})}^2\ge 0,
	\end{equation*}
	for all $f,g\in L^q(\mathscr{O})$, for all $2\le q\le\infty$.
	
	$(ii)$ \emph{Proof of $(b)$}. To begin with, we check that, given $2\le q<\infty$, if $f\in\mathcal{C}^0(\overline{I};L^q(\mathscr{O}))$ then $-\{f\}^{-}\in\mathcal{C}^0(\overline{I};L^q(\mathscr{O}))$. For each $s,t\in\overline{I}$, an application of~\eqref{neg-part-pw-1} gives:
	\begin{equation*}
		\begin{aligned}
			\left(\int_{\mathscr{O}}|(-\{f(t)\}^{-})-(-\{f(s)\}^{-})|^q\dd x\right)\le\left(\int_{\mathscr{O}}|f(t)-f(s)|^q\dd x\right)\to 0,\quad\textup{ as }s\to t.
		\end{aligned}
	\end{equation*}
	
	To check the continuity of the operator $-\{\cdot\}^{-}$ from $\mathcal{C}^0(\overline{I};L^q(\mathscr{O}))$ to itself, let $\{f_n\}_{n=1}^\infty$ be a sequence that converges to $f$ in $\mathcal{C}^0(\overline{I};L^q(\mathscr{O}))$ and observe that an application of~\eqref{neg-part-pw-1} gives:
	\begin{equation*}
		\max_{t\in\overline{I}}\left(\int_{\mathscr{O}}|(-\{f_n(t)\}^{-})-(-\{f(t)\}^{-})|^q\dd x\right)^{1/q}
		\le\max_{t\in\overline{I}}\left(\int_{\mathscr{O}}|f_n(t)-f(t)|^q\dd x\right)^{1/q}\to 0,\quad\textup{ as } n\to\infty.
	\end{equation*}
	
	The case $q=\infty$ follows by resorting to a similar strategy. To show that $f\in\mathcal{C}^0(\overline{I};L^\infty(\mathscr{O}))$ implies $-\{f\}^{-}\in\mathcal{C}^0(\overline{I};L^\infty(\mathscr{O}))$, it suffices to observe that for each $s,t\in\overline{I}$, an application of~\eqref{neg-part-pw-1} gives:
	\begin{equation*}
		\begin{aligned}
			\sup_{x\in\mathscr{O}}|(-\{f(t)(x)\}^{-})-(-\{f(s)(x)\}^{-})|\le \sup_{x\in\mathscr{O}}|f(t)(x)-f(s)(x)|\to 0,\quad\textup{ as }s\to t.
		\end{aligned}
	\end{equation*}
	
	To check the continuity of the operator $-\{\cdot\}^{-}$ from $\mathcal{C}^0(\overline{I};L^\infty(\mathscr{O}))$ to itself, let $\{f_n\}_{n=1}^\infty$ be a sequence that converges to $f$ in $\mathcal{C}^0(\overline{I};L^\infty(\mathscr{O}))$ and observe that an application of~\eqref{neg-part-pw-1} gives
	\begin{equation*}
		\max_{t\in\overline{I}}\left(\sup_{x\in\mathscr{O}}|(-\{f_n(t)(x)\}^{-})-(-\{f(t)(x)\}^{-})|\right)
		\le\max_{t\in\overline{I}}\left(\sup_{x\in\mathscr{O}}|f_n(t)-f(t)|\right)\to 0,\quad\textup{ as } n\to\infty,
	\end{equation*}
	and the proof is complete.
\end{proof}

\begin{remark}
	\label{rem:neg}
	The properties established in Lemma~\ref{lem:1} can be established verbatim for $\gls{pos-part}$.
	\bqed
\end{remark}

The next lemma shows that the mapping $t\in\mathbb{R} \mapsto t^\beta$, with $0<\beta<1$, is sub-additive.
\begin{lemma}[Lemma~3.2 in~\cite{PT23}]
	\label{lem:2-0}
	Let $0<\beta \le 1$. Then
	$$
	\left||x|^\beta - |y|^\beta\right| \le |x-y|^\beta,\quad\textup{ for all } x, y \in \mathbb{R}.
	$$
	\qed
\end{lemma}

The next lemma is a complement of Lemma~\ref{lem:2-0}.
\begin{lemma}[Lemma~3.3 in~\cite{PT23}]
	\label{lem:2}
	Let $\sigma\ge 1$. Then:
	$$
	|x-y|^\sigma \le \left|x^\sigma - y^\sigma\right|, \quad\textup{ for all } x,y \ge 0.
	$$
	\qed
\end{lemma}

Thanks to Lemma~\ref{lem:2}, it is possible to establish the following result, which will allow us to recover the validity of the initial conditions for~\eqref{weak-u}.
\begin{lemma}[Lemma~3.4 in~\cite{PT23}]
	\label{lem:3}
	Let $\Omega$ be a Lipschitz domain in $\mathbb{R}^2$, let $0<T<\infty$ be given, let $1< \sigma <2$, and let $u \in L^\infty(0,T;L^\sigma(\Omega))$ be such that:
	$$
	\dfrac{\dd }{\dd t}\left(|u|^\frac{\sigma-2}{2}u\right) \in L^2(0,T;L^2(\Omega)).
	$$
	
	Then, we have that $\left(|u|^\frac{\sigma-2}{2} u\right) \in \mathcal{C}^0([0,T];L^2(\Omega))$, $|u| \in \mathcal{C}^0([0,T];L^\sigma(\Omega))$, and $u \in\mathcal{C}^0([0,T];L^\sigma(\Omega))$.
	\qed
\end{lemma}

The following lemma leverages Lemma~\ref{lem:3} for establishing finer convergence properties.
\begin{lemma}[Lemma~3.6 in~\cite{PT23}]
	\label{lem:5}
	Let $\Omega$ be a Lipschitz domain in $\mathbb{R}^2$, let $0<T<\infty$ be given, let $1< \sigma <2$.
	Let $\{u_k\}_{k=1}^\infty$ be such that:
	$$
	\left\{|u_k|^\frac{\sigma-2}{2} u_k\right\}_{k=1}^\infty \textup{ strongly converges in }\mathcal{C}^0([0,T];L^2(\Omega)).
	$$
	
	Then $\{|u_k|\}_{k=1}^\infty$ strongly converges in $\mathcal{C}^0([0,T];L^\sigma(\Omega))$.
	\qed
\end{lemma}

The proof establishing the existence of solutions relies on a compactness result proved by Dubinskii~\cite{Dubinski1965} (see also~\cite{BarretSuli2011} for improvements and corrections), as well as on other results presented by P.-A. Raviart~\cite{Raviart1970}, for which we give a different proof below.
\begin{lemma}
	\label{lem:7}
	Let $1< r <\infty$ and let $r'$ denote the H\"older conjugate exponent of $r$. The following inequalities hold
	\begin{align*}
		(|\xi|^{r-2}\xi-|\eta|^{r-2}\eta)\xi &\ge \dfrac{|\xi|^r-|\eta|^r}{r'},\\
		(|\xi|^{r-2}\xi-|\eta|^{r-2}\eta)(\xi-\eta)&\ge C \left(|\xi|^{(r-2)/2}\xi-|\eta|^{(r-2)/2}\eta\right)^2, \textup{ for some }C=C(r)>0,
	\end{align*}
	for all $\xi, \eta \in \mathbb{R}$.
\end{lemma}
\begin{proof}
	We begin with the first inequality. Expanding the left-hand side gives
	\begin{equation*}
		(|\xi|^{r-2}\xi-|\eta|^{r-2}\eta)\xi = |\xi|^r - |\eta|^{r-2}\eta\xi.
	\end{equation*}
	
	Hence the desired inequality is equivalent to
	\begin{equation*}
		|\xi|^r - |\eta|^{r-2}\eta\xi \ge \frac{|\xi|^r-|\eta|^r}{r'}
		\iff \frac{|\xi|^r}{r}+\frac{|\eta|^r}{r'} \ge |\eta|^{r-2}\eta\xi,
	\end{equation*}
	where we have used the identity $1-\frac{1}{r'}=\frac{1}{r}$. Now $|\eta|^{r-2}\eta\xi \le |\eta|^{r-1}|\xi|$, and by Young's inequality~\cite{Young1912} with exponents $r$ and $r'$,
	\begin{equation*}
		\frac{|\xi|^r}{r}+\frac{(|\eta|^{r-1})^{r'}}{r'} \ge |\xi||\eta|^{r-1}.
	\end{equation*}
	
	Since $(r-1)r'=r$, we have $(|\eta|^{r-1})^{r'}=|\eta|^r$, and the first inequality follows.
	
	The second inequality is due to Tartar\footnote{It appears that this inequality was originally established by Luc Tartar. See~\cite{Simon1978} for another proof.}. We now give a different proof for the second inequality.
	For $\xi\neq\eta$, define
	\begin{equation*}
		F(\xi,\eta)=\dfrac{\left(|\xi|^{r-2}\xi-|\eta|^{r-2}\eta\right)(\xi-\eta)}{\left(|\xi|^{\frac{r-2}{2}}\xi-|\eta|^{\frac{r-2}{2}}\eta\right)^{2}},
	\end{equation*}
	and observe that the denominator vanishes only when $\xi=\eta$ since the map $t\mapsto|t|^{(r-2)/2}t$ is strictly monotone. Hence $F$ is well defined and continuous on $\mathbb{R}^{2}\setminus\{(\xi,\eta):\xi=\eta\}$. The rest of the proof is divided into four steps, numbered $(i)$--$(iv)$.
	
	$(i)$ \emph{The mapping $F$ is homogeneous of degree zero}. For any $\lambda>0$, it results
	\begin{equation*}
		F(\lambda\xi,\lambda\eta)=\dfrac{\left(|\lambda\xi|^{r-2}\lambda\xi-|\lambda\eta|^{r-2}\lambda\eta\right)(\lambda\xi-\lambda\eta)}{\left(|\lambda\xi|^{\frac{r-2}{2}}\lambda\xi-|\lambda\eta|^{\frac{r-2}{2}}\lambda\eta\right)^{2}}=\dfrac{\lambda^{r}\left(|\xi|^{r-2}\xi-|\eta|^{r-2}\eta\right)(\xi-\eta)}{\lambda^{r}\left(|\xi|^{\frac{r-2}{2}}\xi-|\eta|^{\frac{r-2}{2}}\eta\right)^{2}}.
	\end{equation*}

	$(ii)$ \emph{Reduction of the problem to a compact set}.
	Thanks to the homogeneity of $F$ established in item $(i)$, the values of $F$ on $\mathbb{R}^{2}\setminus\{(0,0)\}$ are known by solely knowing the values of $F$ on the compact set
	\begin{equation*}
		S=\left\{(\xi,\eta):|\xi|^{r}+|\eta|^{r}=1\right\}.
	\end{equation*}
	
	Indeed, for $(\xi,\eta)\neq(0,0)$ let $m=(|\xi|^{r}+|\eta|^{r})^{1/r}>0$. Then, it results $(\xi/m,\eta/m)\in S$ and, thanks to item $(i)$, $F(\xi,\eta)=F(\xi/m,\eta/m)$.
	The sought inequality certainly holds for $(\xi,\eta)=(0,0)$. Therefore, it suffices to establish the inequality in $S$ and exploit the homogeneity of degree zero to conclude.
	
	$(iii)$ \emph{The mapping $F$ can be continuously extended to $S$}. On $S$ the denominator of $F$ is zero exactly at the two points where $\xi=\eta$. On $S$ the condition $\xi=\eta$ implies $2|\xi|^{r}=1$, i.e. $\xi=\eta=\pm a$ with $a=2^{-1/r}>0$.
	
	We now show that $F$ can be continuously extended to these two points, and that the extension attains positive values there.
	Because of the evenness $F(-\xi,-\eta)=F(\xi,\eta)$, it is enough to study the limit as $(\xi,\eta)\to(a,a)$ along $S$.
	
	By homogeneity, for any $\eta>0$ we have $F(\xi,\eta)=F(\xi/\eta,1)$. For points $(\xi,\eta)\in S$ sufficiently close to $(a,a)$, both coordinates are positive, so we may reduce the limit to the behaviour of
	\begin{equation*}
		H(\xi)=F(\xi,1)=\frac{(\xi^{r-1}-1)(\xi-1)}{(\xi^{r/2}-1)^{2}},\quad \textup{ for all }\xi>0 \textup{ such that } \xi\neq1.
	\end{equation*}
	
	As $(\xi,\eta)\to(a,a)$ along $S$, the ratio $\xi/\eta\to 1$, and the limit of $F$ on $S$ coincides with $\lim_{\xi\to 1}H(\xi)$.
	The Taylor expansion about $\xi_0=1$ of the first factor in the numerator of $H$ is given by:
	\begin{equation*}
		\xi^{r-1}-1=(r-1)(\xi-1)+O((\xi-1)^2).
	\end{equation*}
	
	The Taylor expansion about $\xi_0=1$ of the denominator of $H$ is given by:
	\begin{equation*}
		\xi^{r/2}-1=\dfrac{r}{2}(\xi-1)+O((\xi-1)^2).
	\end{equation*}
	
	Plugging these expressions into the expression for $H$ gives:
	\begin{equation*}
		H(\xi)=\dfrac{(r-1)(\xi-1)^2+O((\xi-1)^3)}{\frac{r^2}{4}(\xi-1)^2+O((\xi-1)^3)}\to\dfrac{4(r-1)}{r^2}>0\quad \textup{ as }\xi\to 1.
	\end{equation*}
	
	$(iv)$ \emph{Completion of the proof}. Since both maps $t\mapsto|t|^{r-2}t$ and $t\mapsto|t|^{(r-2)/2}t$ are strictly increasing for $r>1$, the numerator and denominator of $F$ share the sign of the term $(\xi-\eta)$. Consequently $F>0$ whenever $\xi\neq\eta$. The continuous extension of $F$ to the diagonal takes the value $\frac{4(r-1)}{r^2}>0$. Thus $F$ is everywhere strictly positive on the compact set $S$, and by continuity of the extension and compactness of $S$, the minimum of $F$ on $S$ exists and is strictly positive. We denote this minimum value by $C(r)$, and the proof is complete.
\end{proof}

We now establish two lemmas of general interest.
\begin{lemma}
	\label{lem:9}
	The following estimate holds for all $1\le r<\infty$:
	\begin{equation*}
		\label{est-1}
		|x^r-y^r|\le r(x+y)^{r-1}|x-y|,\quad\textup{ for all }0\le x,y<\infty.
	\end{equation*}
\end{lemma}
\begin{proof}
	The case $x=y$ is trivial, so we focus our attention on the case where $x\neq y$.
	Since the mapping $\xi\in\mathbb{R}^+\cup\{0\}\to\xi^r$ is continuous and differentiable, an application of the Mean Value Theorem (cf., e.g., Theorem~9.3-1 in~\cite{Ciarlet2025}) gives that there exists a number $\min\{x,y\}<\eta<\max\{x,y\}$ such that:
	\begin{equation*}
		x^r-y^r=r\eta^{r-1}(x-y).
	\end{equation*}
	
	Since $\eta<x+y$, the sought inequality immediately follows.
\end{proof}

\begin{lemma}
	\label{lem:10}
	Let $m\ge 1$ be an integer, let $\mathscr{O}$ be an open subset of $\mathbb{R}^m$, let $1\le r<\infty$, and let $\{f_n\}_{n=1}^\infty$ be a sequence of non-negative functions $f_n:\mathscr{O}\to\mathbb{R}^{+}\cup\{0\}$. Then, then following results hold:
	\begin{itemize}
		\item[$(a)$] If the sequence $\{f_n\}_{n=1}^\infty$ converges strongly in $L^r(\mathscr{O})$ as $n\to\infty$, then the sequence $\{f_n^r\}_{n=1}^\infty$ converges strongly in $L^1(\mathscr{O})$ as $n\to\infty$;
		\item[$(b)$] If the sequence $\{f_n\}_{n=1}^\infty$ converges strongly in $L^p(\mathscr{O})$ as $n\to\infty$, with $1\le p<\infty$, then the sequence $\{f_n^r\}_{n=1}^\infty$ converges strongly in $L^{p/r}(\mathscr{O})$ as $n\to\infty$ provided that $1\le r\le p$.
	\end{itemize}
\end{lemma}
\begin{proof}
	The proof is divided into two parts, numbered $(i)$ and $(ii)$.
	
	$(i)$ \emph{Proof of $(a)$}. Let us denote the limit of the sequence $\{f_n\}_{n=1}^\infty$ in $L^r(\mathscr{O})$ by $f$, and observe that $f\ge 0$ a.e. in $\mathscr{O}$. If $r=1$ the conclusion is trivial. If $r>1$, an application of Lemma~\ref{lem:9} gives:
	\begin{equation*}
		\int_{\mathscr{O}}|f_n^r-f^r|\le r\int_{\mathscr{O}}(f_n+f)^{r-1}|f_n-f|\dd x.
	\end{equation*}
	
	It is straightforward to observe that the sequence $\left\{(f_n+f)^{r-1}\right\}_{n=1}^\infty$ is bounded in $L^{r'}(\mathscr{O})$, where $r'$ is the H\"{o}lder conjugate exponent, independently of $n$. An application of H\"older's inequality gives:
	\begin{equation*}
		\int_{\mathscr{O}}(f_n+f)^{r-1}|f_n-f|\dd x\le\left(\int_{\mathscr{O}}(f_n+f)^r\dd x\right)^\frac{r-1}{r}\left(\int_{\mathscr{O}}|f_n-f|^r\dd x\right)^\frac{1}{r}.
	\end{equation*}
	
	Observing that the right-hand side of the latter is the product between a factor that is bounded independently of $n$ and a factor that tends to zero as $n\to\infty$ the conclusion follows.
	
	$(ii)$ \emph{Proof of $(b)$}. Let us denote the limit of the sequence $\{f_n\}_{n=1}^\infty$ in $L^p(\mathscr{O})$ by $f$, and observe that $f\ge 0$ a.e. in $\mathscr{O}$. Since $p/r\ge 1$, an application of Lemma~\ref{lem:2} and item~$(a)$ gives
	\begin{equation*}
		\int_{\mathscr{O}}|f_n^r-f^r|^{p/r}\dd x\le\int_{\mathscr{O}}|f_n^p-f^p|\dd x\to 0,\quad\textup{ as }n\to\infty,
	\end{equation*}
	and the proof is complete.
\end{proof}

We now generalise the conclusion in Lemma~\ref{lem:5}.
\begin{lemma}
	\label{lem:5:improved}
	Let $0<T<\infty$, let $\Omega$ be a Lipschitz domain in $\mathbb{R}^2$, let $2.8\le p\le 5$ and let $\alpha=\frac{3p-1}{2p}$. Let $\{u_n\}_{n=1}^\infty$ be a sequence of functions $u_n:(0,T)\to W_0^{1,p}(\Omega)$ be such that $\left\{|u_n|^\frac{\alpha-2}{2}u_n\right\}_{n=1}^\infty$ converges strongly in $L^r(0,T;L^{2p/\alpha}(\Omega))$,
	for all $1\le r<\infty$. Then, the sequence $\{|u_n|\}_{n=1}^\infty$ converges strongly in $L^p(0,T;L^p(\Omega))$.
\end{lemma}
\begin{proof}
	We show that the sequence $\{|u_n|\}_{n=1}^\infty$ is a Cauchy sequence in $L^p(0,T;L^p(\Omega))$. To this aim, observe that an application of Lemma~\ref{lem:2} gives:
	\begin{equation*}
		\begin{aligned}
			&\int_{0}^{T}\int_{\Omega}\left||u_n|-|u_m|\right|^p\dd x\dd t
			=\int_{0}^{T}\int_{\Omega}\left|\left||u_n|^ \frac{\alpha-2}{2}u_n\right|^\frac{2}{\alpha}-\left||u_m|^\frac{\alpha-2}{2}u_m\right|^\frac{2}{\alpha}\right|^p\dd x\dd t\\
			&\le\int_{0}^{T}\int_{\Omega}\left|\left||u_n|^\frac{\alpha-2}{2}u_n\right|^\frac{2p}{\alpha}-\left||u_m|^\frac{\alpha-2}{2}u_m\right|^\frac{2p}{\alpha}\right|\dd x\dd t.
		\end{aligned}
	\end{equation*}
	
	Since $\left\{|u_n|^\frac{\alpha-2}{2}u_n\right\}_{n=1}^\infty$ strongly converges in $L^r(0,T;L^{2p/\alpha}(\Omega))$ for each $1\le r<\infty$, an application of Lemma~\ref{lem:10}$(a)$ with $r=2p/\alpha>4$ gives that the sequence $\left\{\left||u_n|^\frac{\alpha-2}{2}u_n\right|^\frac{2p}{\alpha}\right\}_{n=1}^\infty$ strongly converges in $L^1((0,T)\times\Omega)$ and so in $L^1(0,T;L^1(\Omega))$ (Theorem~8.28 of~\cite{Leoni2017}). Since, once again thanks to Theorem~8.28 of~\cite{Leoni2017}, it is possible to identify $L^p(0,T;L^p(\Omega))$ with $L^p((0,T)\times\Omega)$, the conclusion follows.
\end{proof}

Let us now study the convergence of the sequence $\left\{|u_n|^\frac{p-1}{2p}\right\}_{n=1}^\infty$, for $2.8\le p\le 5$.
\begin{lemma}
	\label{lem:11}
	Let $0<T<\infty$, let $\Omega$ be a Lipschitz domain in $\mathbb{R}^2$, let $2.8\le p\le 5$. Let $\{u_n\}_{n=1}^\infty$ be a sequence of functions $u_n:(0,T)\to W_0^{1,p}(\Omega)$ be such that $\{|u_n|\}_{n=1}^\infty$ converges strongly in $L^p(0,T;L^p(\Omega))$ to a function $u$. Then, the sequence $\left\{|u_n|^\frac{p-1}{2p}\right\}_{n=1}^\infty$ converges strongly to $|u|^\frac{p-1}{2p}$ in $L^{2p^2/(p-1)}(0,T;L^{2p^2/(p-1)}(\Omega))$.
\end{lemma}
\begin{proof}
	We adopt the same strategy as Lemma~\ref{lem:5:improved}.
	Since $2p^2/(p-1)>8$, an application of Lemma~\ref{lem:2} gives:
	\begin{equation*}
		\int_{0}^{T}\int_{\Omega}\left||u_n|^\frac{p-1}{2p}-|u|^\frac{p-1}{2p}\right|^\frac{2p^2}{p-1}\dd x\dd t
		\le\int_{0}^{T}\int_{\Omega}\left||u_n|^p-|u|^p\right|\dd x\dd t.
	\end{equation*}
	
	The conclusion follows as a result of an application of Lemma~\ref{lem:10}$(a)$ and Theorem~8.28 of~\cite{Leoni2017}.
\end{proof}

Thanks to the continuity of the positive part $\{\cdot\}^{+}$, an application of Lemma~\ref{lem:1}$(a)$ (which clearly holds for $\{\cdot\}^{+}$ as well in light of Remark~\ref{rem:neg}) gives that if $\{|u_n|\}_{n=1}^\infty$ converges strongly in $L^p(0,T;L^p(\Omega))$ then:
\begin{equation*}
	\left\{\left\{|u_n|^\frac{p-1}{2p}-z\right\}^+\right\}_{n=1}^\infty \textup{ converges strongly in }L^\frac{2p^2}{p-1}((0,T)\times\Omega\times (0,L)).
\end{equation*}

Therefore, an application of Lemma~\ref{lem:10}$(b)$ (observe that $\frac{5}{2}\le\frac{2p}{p-1}\le \frac{28}{9}$ if $2.8\le p\le 5$) gives that if $\{|u_n|\}_{n=1}^\infty$ converges strongly in $L^p(0,T;L^p(\Omega))$ then:
\begin{equation}
	\label{conv-rhs-temp}
	\left\{\left(\left\{|u_n|^\frac{p-1}{2p}-z\right\}^+\right)^p\right\}_{n=1}^\infty \textup{ converges strongly in }L^\frac{2p}{p-1}((0,T)\times\Omega\times (0,L)).
\end{equation}

For the purpose of studying the right-hand side in~\eqref{vi:thermal}, we establish the following results of general interest.
\begin{lemma}
	\label{lem:V'}
	Let $\mathscr{O}$ be a bounded open set in $\mathbb{R}^m$ for some integer $m\ge 1$, let $2.8\le p\le 5$ and let $\frac{10}{3}\le q\le\frac{28}{5}$. Let $\sigma\in\mathbb{R}$ be defined by:
	\begin{equation*}
		\dfrac{1}{\sigma}=\dfrac{1}{q}+\dfrac{p-1}{2p}.
	\end{equation*}
	
	Then, it results $q'\le\sigma\le 2$ and $2\le\sigma'\le q$ and the following chain of embeddings holds:
	\begin{equation*}
		L^q(\mathscr{O})\hookrightarrow L^{\sigma'}(\mathscr{O})\hookrightarrow L^2(\mathscr{O})\hookrightarrow L^\sigma(\mathscr{O})\hookrightarrow L^{q'}(\mathscr{O}).
	\end{equation*}
\end{lemma}
\begin{proof}
	Observe that the definition of $\sigma$ gives that:
	\begin{equation*}
		\sigma\ge\dfrac{5q}{2q+5}.
	\end{equation*}
	
	A simple calculation gives:
	\begin{equation*}
		\dfrac{5q}{2q+5}\ge\dfrac{q}{q-1} \quad\textup{ if and only if }q\ge\dfrac{10}{3}.
	\end{equation*}
	
	Again from the definition of $\sigma$, we obtain that:
	\begin{equation*}
		\sigma\le\dfrac{28q}{9q+28}.
	\end{equation*}
	
	Another simple calculation gives:
	\begin{equation*}
		\dfrac{28q}{9q+28}\le 2 \quad\textup{ if and only if }q\le\dfrac{28}{5}.
	\end{equation*}
	
	In conclusion, we have shown that $q'\le\sigma\le 2$. The remaining inequalities for $\sigma'$ immediately follow from the fact that the mapping $\sigma>1\mapsto\frac{\sigma}{\sigma-1}$ is decreasing. The claimed chain of embeddings follows by H\"older's inequality.
\end{proof}

We establish some preliminary results of independent interest for the exponential mapping defined over Bochner spaces, that will be used to study the function $A$ introduced in~\eqref{arrhenius}.
\begin{lemma}
	\label{lemma:exp:1}
	Let $m\ge1$ be an integer, let $\mathscr{O}\subset\mathbb{R}^m$ be open and bounded, and let $1\le q\le\infty$. For each function $f\in L^1(\mathscr{O})$ such that $f(x)\le 0$ for a.a. $x\in\mathscr{O}$, we define the \emph{exponential at} $f$ by:
	\begin{equation*}
		\gls{expf}(x):=e^{f(x)},\quad\textup{ for a.a. }x\in\mathscr{O}.
	\end{equation*}
	
	Consider a sequence $\{f_n\}_{n=1}^\infty$ of non-positive real valued functions $f_n:\mathscr{O}\to\mathbb{R}^{-}\cup\{0\}$ such that $f_n\to f$ in $L^q(\mathscr{O})$ as $n\to\infty$.
	Then, it results:
	\begin{equation*}
		\exp(f_n)\to\exp(f),\quad\textup{ in }L^q(\mathscr{O}) \textup{ as }n\to\infty.
	\end{equation*}
\end{lemma}
\begin{proof}
	Let us consider the case where $1\le q<\infty$. The fact that $f_n\le 0$ a.e. in $\mathscr{O}$ for all integers $n\ge 1$ implies that:
	\begin{equation*}
		0\le\exp(f_n)\le 1,\quad\textup{ a.e. in }\mathscr{O}\textup{ for all integers }n\ge 1.
	\end{equation*}
	
	An application of the Mean Value Theorem (cf., e.g., Theorem~9.3-1 in~\cite{Ciarlet2025}) gives
	\begin{equation}
		\label{lip:exp}
		\left(\int_{\mathscr{O}}|\exp(f_n)-\exp(f)|^q \dd x\right)^{1/q}\le\left(\int_{\mathscr{O}}|f_n-f|^q \dd x\right)^{1/q},
	\end{equation}
	and the right-hand side of the latter converges to zero as $n\to\infty$ in light of the assumed convergence.
	
	For the case where $q=\infty$, it suffices to observe that an application of the Mean Value Theorem (cf., e.g., Theorem~9.3-1 in~\cite{Ciarlet2025}) gives
	\begin{equation*}
		\|\exp(f_n)-\exp(f)\|_{L^\infty(\mathscr{O})}\le\|f_n-f\|_{L^\infty(\mathscr{O})} \to 0,
	\end{equation*}
	as $n\to\infty$ in light of the assumed convergence. The proof is complete.
\end{proof}

For the purpose of studying the function $\mu$ in~\eqref{mu-def:2}, we establish the following results of general interest.
The next result complements Lemma~\ref{lemma:exp:1}.
\begin{lemma}
	\label{lemma:exp-2}
	Let $I$ be a bounded open interval in $\mathbb{R}$, let $\mathscr{O}$ be a bounded open set in $\mathbb{R}^m$, for some integer $m\ge 1$. Assume that $1\le q\le\infty$. Let $\{f_n\}_{n=1}^\infty$ be a sequence in $\mathcal{C}^0(\overline{I};L^q(\mathscr{O}))$ such that:
	\begin{itemize}
		\item[$(a)$] For each $t\in\overline{I}$ it results $f_n(t)\le 0$ a.e. in $\mathscr{O}$,
		\item[$(b)$] The sequence $\{f_n\}_{n=1}^\infty$ converges in $\mathcal{C}^0(\overline{I};L^q(\mathscr{O}))$ to a function $f\in\mathcal{C}^0(\overline{I};L^q(\mathscr{O}))$.
	\end{itemize}
	
	Then, we obtain that $\{\exp(f_n)\}_{n=1}^\infty$ converges in $\mathcal{C}^0(\overline{I};L^q(\mathscr{O}))$ to $\exp(f)\in\mathcal{C}^0(\overline{I};L^q(\mathscr{O}))$.
\end{lemma}
\begin{proof}
	Let us consider the case where $1\le q<\infty$; the case $q=\infty$ follows by the same argument with $\|\cdot\|_{L^\infty(\mathscr{O})}$ in place of $\|\cdot\|_{L^q(\mathscr{O})}$. The proof is divided into four parts, numbered $(i)$--$(iv)$.
	
	$(i)$ \emph{Non-positiveness of the limit.} We first show that $f(t)\le 0$ a.e. in $\mathscr{O}$ for each $t\in\overline{I}$. By Lemma~\ref{lem:1}$(a)$, the operator $-\{\cdot\}^{-}$ is Lipschitz continuous on $L^q(\mathscr{O})$. Since $f_n(t)\le 0$ a.e. in $\mathscr{O}$ for each $t\in\overline{I}$ by assumption $(a)$, we have $f_n(t)=-\{f_n(t)\}^{-}$ a.e. in $\mathscr{O}$ for each $t\in\overline{I}$. Consequently,
	\begin{equation*}
		\begin{aligned}
			\max_{t\in\overline{I}}&\left(\int_{\mathscr{O}}|f_n(t)-(-\{f(t)\}^{-})|^q\dd x\right)^{1/q}
			=\max_{t\in\overline{I}}\left(\int_{\mathscr{O}}|(-\{f_n(t)\}^{-})-(-\{f(t)\}^{-})|^q\dd x\right)^{1/q}\\
			&\le\max_{t\in\overline{I}}\left(\int_{\mathscr{O}}|f_n(t)-f(t)|^q\dd x\right)^{1/q}\to 0,
		\end{aligned}
	\end{equation*}
	as $n\to\infty$. Hence, for each $t\in\overline{I}$, the uniqueness of the limit in $L^q(\mathscr{O})$ gives $f(t)=-\{f(t)\}^{-}=\min\{f(t),0\}\le 0$ a.e. in $\mathscr{O}$.
	
	$(ii)$ \emph{A pointwise estimate.} Since $f_n(t),f(t)\le 0$ a.e.\ in $\mathscr{O}$ for all $t\in\overline{I}$, the exponential function is evaluated on non-positive arguments. Applying the Mean Value Theorem (cf., e.g., Theorem~9.3-1 in~\cite{Ciarlet2025}) to $\xi\mapsto\exp(\xi)$ on the interval between $f_n(t)(x)$ and $f(t)(x)$ yields, for all $t\in\overline{I}$ and for a.a.\ $x\in\mathscr{O}$
	\begin{equation*}
		|\exp(f_n(t)(x))-\exp(f(t)(x))| = e^{\xi_{t,x}}\,|f_n(t)(x)-f(t)(x)| \le |f_n(t)(x)-f(t)(x)|,
	\end{equation*}
	where $\xi_{t,x}\le 0$ and therefore $e^{\xi_{t,x}}\le 1$.
	
	Applying the same argument as in the estimate above gives, for any $s,t\in\overline{I}$
	\begin{equation*}
		\left\|\exp(f_n(t))-\exp(f_n(s))\right\|_{L^q(\mathscr{O})}\le\|f_n(t)-f_n(s)\|_{L^q(\mathscr{O})}\to 0,
	\end{equation*}
	as $s\to t$, thus showing that $\{\exp(f_n)\}_{n=1}^\infty\subset\mathcal{C}^0(\overline{I};L^q(\mathscr{O}))$. The same argument applies to $\exp(f)$.
	
	$(iii)$ \emph{From pointwise to $L^q(\mathscr{O})$ norm.} Using the monotonicity of $\xi\mapsto\xi^q$ on $[0,\infty)$, raising the pointwise estimate to the power $q$ gives:
	\begin{equation*}
		|\exp(f_n(t)(x))-\exp(f(t)(x))|^q \le |f_n(t)(x)-f(t)(x)|^q,\quad\textup{ for all }t\in\overline{I}\textup{ and for a.a. }x\in\mathscr{O}.
	\end{equation*}
	
	Integrating over $\mathscr{O}$ and using the monotonicity of $\xi\mapsto\xi^{1/q}$ gives:
	\begin{equation*}
		\left(\int_{\mathscr{O}}|\exp(f_n(t))-\exp(f(t))|^q\dd x\right)^{1/q}
		\le\left(\int_{\mathscr{O}}|f_n(t)-f(t)|^q\dd x\right)^{1/q}.
	\end{equation*}
	
	$(iv)$ \emph{Convergence in $\mathcal{C}^0(\overline{I};L^q(\mathscr{O}))$.} Taking the maximum over $t\in\overline{I}$ and using assumption $(b)$ gives
	\begin{equation*}
		\max_{t\in\overline{I}}\left(\int_{\mathscr{O}}|\exp(f_n(t))-\exp(f(t))|^q\dd x\right)^{1/q}
		\le\max_{t\in\overline{I}}\left(\int_{\mathscr{O}}|f_n(t)-f(t)|^q\dd x\right)^{1/q}\to 0,
	\end{equation*}
	as $n\to\infty$. Finally, the fact that $\exp(f)\in\mathcal{C}^0(\overline{I};L^q(\mathscr{O}))$ follows from the boundedness of $\mathscr{O}$ and the estimate $0\le\exp(f(t))\le 1$ a.e. in $\mathscr{O}$ for all $t\in\overline{I}$, which guarantees $\exp(f(t))\in L^q(\mathscr{O})$ for each $t\in\overline{I}$, together with the same Lipschitz estimate applied to $f$ itself to establish continuity at $t\in\overline{I}$.
	
	The proof for $q=\infty$ is analogous, replacing the $L^q(\mathscr{O})$ norm with the essential supremum.
\end{proof}

The next results establishes useful properties for averages in Bochner spaces.
\begin{lemma}
	\label{lem:avg}
	Let $I$ be a bounded open interval in $\mathbb{R}$, let $\mathscr{O}_1$ be a bounded open set in $\mathbb{R}^{m_1}$ and let $\mathscr{O}_2$ be an open set in $\mathbb{R}^{m_2}$, for some integer $m_1,m_2\ge 1$. Assume that $1\le q\le\infty$.
	Let $\{f_n\}_{n=1}^\infty$ be a sequence in $L^q(I;L^q(\mathscr{O}_1\times\mathscr{O}_2))$ such that $\{f_n\}_{n=1}^\infty$ converges in $L^q(I;L^q(\mathscr{O}_1\times\mathscr{O}_2))$.
	Then, it results that the sequence
	\begin{equation*}
		\left\{\int_{\mathscr{O}_1}f_n\dd x\right\}_{n=1}^\infty
	\end{equation*}
	converges in $L^q(I;L^q(\mathscr{O}_2))$.
\end{lemma}
\begin{proof}
	Let us consider the case where $q$ is finite. We first verify that if $f\in L^q(I;L^q(\mathscr{O}_1\times\mathscr{O}_2))$ then $\int_{\mathscr{O}_1}f\dd x\in L^q(I;L^q(\mathscr{O}_2))$. To this aim, we observe that an application of H\"older's inequality gives
	\begin{equation*}
		\int_{I}\int_{\mathscr{O}_2} \left|\int_{\mathscr{O}_1} f \dd x\right|^q \dd z \dd t\le |\mathscr{O}_1|^{q-1}\int_{I}\int_{\mathscr{O}_2}\int_{\mathscr{O}_1}|f|^q\dd x\dd z\dd t,
	\end{equation*}
	and the conclusion follows directly from the Fubini--Tonelli theorem (cf., e.g., \cite{Ciarlet2025}).
	
	To complete the proof, consider a sequence $\{f_n\}_{n=1}^\infty$ that converges to a function $f$ in $L^q(I;L^q(\mathscr{O}_1\times\mathscr{O}_2))$, and observe that
	\begin{equation*}
		\begin{aligned}
			&\int_{I}\int_{\mathscr{O}_2}\left|\int_{\mathscr{O}_1}f_n \dd x-\int_{\mathscr{O}_1} f\dd x\right|^q \dd z \dd t
			\le|\mathscr{O}_1|^{q-1}\int_{I}\int_{\mathscr{O}_2}\int_{\mathscr{O}_1}|f_n-f|^q \dd x\dd z\dd t\\
			&=|\mathscr{O}_1|^{q-1}\int_{I}\int_{\mathscr{O}_1\times\mathscr{O}_2}|f_n-f|^q \dd x\dd z\dd t,
		\end{aligned}
	\end{equation*}
	as $n\to\infty$, where the estimates holds once again thanks to H\"older's inequality and Theorem~8.28 in~\cite{Leoni2017}.
	
	If $q=\infty$, for a.a. $t\in I$ and for a.a. $z\in\mathscr{O}_2$, another application of the H\"older inequality and Theorem~8.28 in~\cite{Leoni2017} give:
	\begin{equation*}
		\left|\int_{\mathscr{O}_1}(f_n(t)(x,z)-f(t)(x,z))\dd x\right|\le |\mathscr{O}_1|\|f_n(t)-f(t)\|_{L^\infty(\mathscr{O}_1\times\mathscr{O}_2)}.
	\end{equation*}
	
	Passing to the essential supremum over $I$ in the estimate above gives:
	\begin{equation*}
		\sup_{t\in I}\left\|\int_{\mathscr{O}_1}(f_n(t)(x,z)-f(t)(x,z))\dd x\right\|_{L^\infty(\mathscr{O}_2)}\le |\mathscr{O}_1|\sup_{t\in I}\|f_n(t)-f(t)\|_{L^\infty(\mathscr{O}_1\times\mathscr{O}_2)}.
	\end{equation*}
	
	The assumed convergence for the sequence $\{f_n\}_{n=1}^\infty$ in turn gives that the right-hand side of the estimate above tends to zero as $n\to\infty$, completing the proof.
\end{proof}

\begin{lemma}
	\label{lem:avg-2}
	Let $I$ be a bounded open interval in $\mathbb{R}$, let $\mathscr{O}_1$ be a bounded open set in $\mathbb{R}^{m_1}$ and let $\mathscr{O}_2$ be an open set in $\mathbb{R}^{m_2}$, for some integer $m_1,m_2\ge 1$. Assume that $1\le q\le\infty$. Let $\{f_n\}_{n=1}^\infty$ be a sequence in $\mathcal{C}^0(\overline{I};L^q(\mathscr{O}_1\times\mathscr{O}_2))$ such that $\{f_n\}_{n=1}^\infty$ converges in $\mathcal{C}^0(\overline{I};L^q(\mathscr{O}_1\times\mathscr{O}_2))$. Then, it results that the sequence
	\begin{equation*}
		\left\{\int_{\mathscr{O}_1}f_n\dd x\right\}_{n=1}^\infty
	\end{equation*}
	converges in $\mathcal{C}^0(\overline{I};L^q(\mathscr{O}_2))$.
\end{lemma}
\begin{proof}
	Let us first consider the case where $q$ is finite. We first verify that if $f\in\mathcal{C}^0(\overline{I};L^q(\mathscr{O}_1\times\mathscr{O}_2))$ then $\int_{\mathscr{O}_1}f_n\dd x\in\mathcal{C}^0(\overline{I};L^q(\mathscr{O}_2))$. To this aim, we observe that for each $s,t\in\overline{I}$, it results that an application of the H\"older inequality gives
	\begin{equation*}
		\int_{\mathscr{O}_2}\left|\int_{\mathscr{O}_1}f(s)\dd x-\int_{\mathscr{O}_1}f(t)\dd x\right|^q\dd z\le|\mathscr{O}_1|^{q-1}\int_{\mathscr{O}_2}\int_{\mathscr{O}_1}|f(s)-f(t)|^q\dd x\dd z\to 0,
	\end{equation*}
	when $s\to t$. To complete the proof, consider a Cauchy sequence $\{f_n\}_{n=1}^\infty$ that converges to $f$ in $\mathcal{C}^0(\overline{I};L^q(\mathscr{O}_1\times\mathscr{O}_2))$, and evaluate
	\begin{equation*}
		\max_{t\in\overline{I}}\left\|\int_{\mathscr{O}_1}f_n(t)\dd x-\int_{\mathscr{O}_1}f(t)\dd x\right\|_{L^q(\mathscr{O}_2)}\le|\mathscr{O}_1|^{1-\frac{1}{q}}\max_{t\in\overline{I}}\left(\int_{\mathscr{O}_2}\int_{\mathscr{O}_1}|f_n(t)-f(t)|^q\dd x\dd z\right)^{1/q}\to 0,
	\end{equation*}
	as $n\to\infty$, where the estimates holds once again thanks to H\"older's inequality.
	
	If $q=\infty$, for all $t\in\overline{I}$ and for a.a. $z\in\mathscr{O}_2$, another application of the H\"older inequality and Theorem~8.28 in~\cite{Leoni2017} give:
	\begin{equation*}
		\left|\int_{\mathscr{O}_1}(f_n(t)(x,z)-f(t)(x,z))\dd x\right|\le |\mathscr{O}_1|\|f_n(t)-f(t)\|_{L^\infty(\mathscr{O}_1\times\mathscr{O}_2)}.
	\end{equation*}
	
	taking the maximum over $\overline{I}$ in the estimate above gives:
	\begin{equation*}
		\max_{t\in\overline{I}}\left\|\int_{\mathscr{O}_1}(f_n(t)(x,z)-f(t)(x,z))\dd x\right\|_{L^\infty(\mathscr{O}_2)}\le |\mathscr{O}_1|\max_{t\in\overline{I}}\|f_n(t)-f(t)\|_{L^\infty(\mathscr{O}_1\times\mathscr{O}_2)}.
	\end{equation*}
	
	The assumed convergence for the sequence $\{f_n\}_{n=1}^\infty$ in turn gives that the right-hand side of the estimate above tends to zero as $n\to\infty$, completing the proof.
\end{proof}

\begin{remark}
	\label{rem:lem:avg}
	Following a strategy similar to the one presented for establishing Lemma~\ref{lem:avg}, one can show that if $\{f_n\}_{n=1}^\infty$ converges in $L^q(I;L^q(\mathscr{O}_1))$ then the sequence $\left\{\int_{\mathscr{O}_1}f_n\dd x\right\}_{n=1}^\infty$ converges in $L^q(I)$.
	
	Likewise, if $\{f_n\}_{n=1}^\infty$ converges in $\mathcal{C}^0(\overline{I};L^q(\mathscr{O}_1))$ then the sequence $\left\{\int_{\mathscr{O}_1}f_n\dd x\right\}_{n=1}^\infty$ converges in $\mathcal{C}^0(\overline{I})$.
	\bqed
\end{remark}

Let us now recall a useful result to analyse the spatial gradient of functions defined over Sobolev--Bochner spaces.
\begin{lemma}
	\label{th:gradBochner}
	Let $N\ge 1$ be an integer, let $\mathscr{O}$ be an open set in $\mathbb{R}^N$, let $0<T\le\infty$ and let $1<p<\infty$.
	Let $\vec{v} \in L^p(0,T;W^{1,p}(\mathscr{O}))$. For each integer $1\le i\le N$, it is possible to define the mapping $\widetilde{\partial_i \vec{v}}:(0,T) \to L^p(\mathscr{O})$, which is given by $\widetilde{\partial_i \vec{v}}(t):=\partial_i(\vec{v}(t))$, for a.a. $t\in(0,T)$.
	
	In particular, if $\{\vec{w}_n\}_{n=1}^\infty$ is a sequence such that $\vec{w}_n\rightharpoonup\vec{w}$ in $L^p(0,T;W^{1,p}(\mathscr{O}))$ as $n\to\infty$, it results:
	\begin{equation*}
		\widetilde{\partial_i\vec{w}_n}\rightharpoonup\widetilde{\partial_i\vec{w}},\quad\textup{ in }L^p(0,T;L^p(\mathscr{O})) \textup{ as }n\to\infty.
	\end{equation*}
\end{lemma}
\begin{proof}
	To begin with, let us observe that for each $\vec{v}\in L^p(0,T;W^{1,p}(\mathscr{O}))$ it results:
	\begin{equation}
		\label{grad:bound-1}
		\int_0^T \int_{\mathscr{O}} \left|\widetilde{\partial_i \vec{v}}(t)\right|^p\dd x \dd t=\int_0^T \int_{\mathscr{O}} \left|\partial_i (\vec{v}(t))\right|^p\dd x \dd t\le \|\vec{v}\|_{L^p(0,T;W^{1,p}(\mathscr{O}))}^p.
	\end{equation}
	
	Let $\{\vec{w}_n\}_{n=1}^\infty$ be a sequence such that $\vec{w}_n\rightharpoonup\vec{w}$ in $L^p(0,T;W^{1,p}(\mathscr{O}))$ as $n\to\infty$. In light of~\eqref{grad:bound-1}, it thus results that the sequence $\{\widetilde{\partial_i\vec{w}_n}\}_{n=1}^\infty$ is bounded in $L^p(0,T;L^p(\mathscr{O}))$ for each $1\le i\le N$ independently of $n$.
	
	Let $\vec{\varphi}\in\mathcal{D}(0,T;\mathcal{D}(\mathscr{O}))$ be given, and recall that $L^{p'}(0,T;L^{p'}(\mathscr{O}))\cong L^{p'}((0,T)\times\mathscr{O})$ (Theorem~8.28 in~\cite{Leoni2017}). Let $\phi\in L^{p'}((0,T)\times\mathscr{O})$ be such that $\phi(t,x)=\vec{\varphi}(t)(x)$. An application of the definition of weak derivative in Bochner spaces gives
	\begin{equation*}
		\int_{0}^{T}\int_{\mathscr{O}}\partial_i(\vec{w}_n(t))\vec{\varphi}(t)\dd x\dd t=-\int_{0}^{T}\int_{\mathscr{O}}\vec{w}_n(t)\partial_i\phi(t,\cdot)\dd x\dd t,
	\end{equation*}
	and we observe that the right-hand side verifies the convergence
	\begin{equation*}
		-\int_{0}^{T}\int_{\mathscr{O}}\vec{w}_n(t)\partial_i\phi(t,\cdot)\dd x\dd t\to-\int_{0}^{T}\int_{\mathscr{O}}\vec{w}(t)\partial_i(\vec{\varphi}(t))\dd x\dd t,
	\end{equation*}
	as $n\to\infty$. The proof is complete upon observing that $\mathcal{D}(0,T;\mathcal{D}(\mathscr{O}))$ is dense in $L^{p'}(0,T;L^{p'}(\mathscr{O}))$.
\end{proof}

The next preliminary result we recall, is a generalised integration-by-parts formula, whose proof hinges on the Lebesgue theorem (cf., e.g., Theorem~2.11-3 of~\cite{PGCLCSHA}).
\begin{lemma}[Lemma~1.2 of~\cite{Raviart1970}]
	\label{lem:8}
	Let $\Omega$ be a Lipschitz domain in $\mathbb{R}^2$, and let $T>0$.
	Let $\alpha$ and $p$ be two real numbers greater than $1$, and let $v$ be a function such that
	\begin{align*}
		v \in L^p(0,T;W_0^{1,p}(\Omega)) \cap \mathcal{C}^0([0,T];L^\alpha(\Omega)),\\
		\dfrac{\dd }{\dd t}(|v|^{\alpha-2} v) \in L^{p'}(0,T;W^{-1,p'}(\Omega)).
	\end{align*}
	
	Then, the following formula holds
	$$
	\int_{0}^{T}\left\langle \dfrac{\dd }{\dd t}(|v|^{\alpha-2} v), v \right\rangle \dd t = \dfrac{\|v(T)\|_{L^\alpha(\Omega)}^\alpha}{\alpha'} - \dfrac{\|v(0)\|_{L^\alpha(\Omega)}^\alpha}{\alpha'},
	$$
	where $\langle \cdot, \cdot\rangle$ denotes the duality between $W^{-1,p'}(\Omega)$ and $W_0^{1,p}(\Omega)$.
	\qed
\end{lemma}

We now state the well-known Dubinskii's lemma (cf., e.g., Theorem~1 of~\cite{Dubinski1965}, Theorem~2.1 of~\cite{BarretSuli2011}, Theorem~3 in~\cite{CL12} or Theorem~1 in~\cite{CJL14}).
\begin{lemma}
	\label{Dub}
	Let $A_0$ and $A_1$ be Banach spaces such that $A_0 \hookrightarrow A_1$. Let $S \subset A_0$ be such that $\lambda S \subset S$, for all $\lambda \in \mathbb{R}$.
	
	Assume that the set $S$ is endowed with the definite homogeneous gauge $M:S\to \mathbb{R}$, that satisfies the following properties:
	\begin{itemize}
		\item $M(v) \ge 0$, for all $v \in S$,
		\item $M(\lambda v)=|\lambda| M(v)$, for all $v \in S$ and all $\lambda \in \mathbb{R}$.
	\end{itemize}
	
	Assume that the set $\mathscr{M}:=\{v \in S;M(v)\le 1\}$ is relatively compact in $A_0$. Consider the semi-normed set
	$$
	Y:=\left\{u \in L^1_{\textup{loc}}(0,T;A_1); \int_{0}^{T} [M(u(t))]^{q_0} \dd t < \infty \textup{ and } \int_{0}^{T} \left\|\dfrac{\dd u}{\dd t}(t)\right\|_{A_1}^{q_1}\dd t < \infty\right\},
	$$
	with $1 \le q_0 \le \infty$ and $1 \le q_1 \le \infty$, and the pairs $(q_0,q_1)=(1,\infty)$ and $(q_0,q_1)=(\infty,1)$ cannot be attained.
	Then, it results that:
	\begin{itemize}
		\item[$(a)$] $Y$ is relatively compact in $L^{q_0}(0,T;A_0)$ when $1\le q_0<\infty$;
		\item[$(b)$] $Y$ is relatively compact in $\mathcal{C}^0([0,T];A_0)$ when $q_0=\infty$.
	\end{itemize}
	\qed
\end{lemma}

Finally, we establish a sharp generalised Poincar\'{e}--Friedrichs inequality for the space $V$, giving an explicit construction for the constant.
\begin{lemma}
	\label{PFV}
	The following generalised Poincar\'{e}--Friedrichs inequality holds:
	\begin{equation*}
		\|\Xi\|_{L^2(\Omega\times (0,L))}\le\sqrt{2}L\|\nabla\Xi\|_{\bm{L}^2(\Omega\times (0,L))},\quad\textup{ for all }\Xi\in V.
	\end{equation*}
\end{lemma}
\begin{proof}
	We reflect the the Lipschitz domain $\Omega\times (0,L)$ across the plane $\{z=0\}$ by even extension. As a result, with each $\Xi\in V$, we can associate the function $\hat{\Xi}\in H^1_0(\Omega\times (-L,L))$ defined by:
	\begin{equation*}
		\hat{\Xi}(x,z):=
		\begin{cases}
			\Xi(x,z)&,\textup{ if }(x,z)\in\Omega\times (0,L),\\
			\Xi(x,-z)&,\textup{ if }(x,z)\in\Omega\times (-L,0).
		\end{cases}
	\end{equation*}
	
	Recall that, in the special case where $\overline{\Omega}$ can be regarded as a surface in $\mathbb{R}^3$ with zero torsion, the best constant for Theorem~9.7 in~\cite{Brez11} is given by $\sqrt{2}$. Therefore, an application of the Poincar\'{e}--Friedrichs inequality (cf., e.g., Theorem~8.3-3$(a)$ in~\cite{Ciarlet2025}) gives
	\begin{equation*}
		\sqrt{2}\|\Xi\|_{L^2(\Omega\times (0,L))}=\|\hat{\Xi}\|_{L^2(\Omega\times (-L,L))}\le\sqrt{2}L\|\nabla\hat{\Xi}\|_{\bm{L}^2(\Omega\times (-L,L))}=2L\|\nabla\Xi\|_{\bm{L}^2(\Omega\times (0,L))},
	\end{equation*}
	from which the completion of the proof immediately follows.
\end{proof}

The coupling term $\mu$ defined in~\eqref{mu-def} takes the following form in the context of the framework we are adopting,
\begin{equation}
	\label{mu-def:2}
	\begin{aligned}
		&\mu(\tilde{\Theta}):=2\left(\rho g \dfrac{p-1}{2p}\right)^{p-1}\\
		&\quad\times\int_{0}^{1}\left(\max\left\{A_0 \exp\left(-\dfrac{\beta\sqrt{L}}{|\Omega|^{1/2}}\int_{\Omega}\left\{\tilde{\Theta}(t,x,s'L)-\Theta_{\textup{m}}\right\}^{-}\dd x\right),A_{\textup{min}}\right\}^{+}\right) (1-s')^p \dd s',
	\end{aligned}
\end{equation}
and the bounds~\eqref{mu0} continue to hold. The simplification proposed in~\eqref{mu-def:2} will be critical to establish the validity of the energy estimates in Theorem~\ref{thm:2}.

We now state the penalised problem, which is suggested by the formal model~\eqref{weak-u}--\eqref{theta:ic}. Observe that the validity of the initial conditions makes sense thanks to Lemma~\ref{lem:3}. In all what follows, we denote by $\gls{elle}$ a \emph{positive penalty parameter} meant to approach zero.

\begin{customprob}{$\mathcal{P}_\ell$}
	\label{Pkappa}
	Find function $\gls{uell}$ and $\gls{Tell}$ satisfying
	\begin{align*}
		u_\ell &\in L^\infty(0,T;W_0^{1,p}(\Omega)),\\
		\dfrac{\dd }{\dd t}\left(|u_\ell|^{\frac{\alpha-2}{2}}u_\ell\right) &\in L^2(0,T;L^2(\Omega)),\\
		\dfrac{\dd }{\dd t}\left(|u_\ell|^{\alpha-2}u_\ell\right) &\in L^\infty(0,T;W^{-1,p'}(\Omega)),\\
		\tilde{\Theta}_\ell&\in L^\infty(0,T;V),\\
		\dfrac{\dd \tilde{\Theta}_\ell}{\dd t} &\in L^2(0,T;L^2(\Omega\times (0,L))),
	\end{align*}
	satisfying the variational equations
	\begin{equation}
		\label{penalty-u}
		\int_{\Omega} \dfrac{\dd }{\dd t}(|u_\ell|^{\alpha-2} u_\ell) v \dd x
		+\mu(\tilde{\Theta}_\ell)\int_{\Omega}|\nabla u_\ell|^{p-2}\nabla u_\ell \cdot \nabla v \dd x -\dfrac{\mu(\tilde{\Theta}_\ell)}{\ell}\int_{\Omega} \{u_\ell\}^{-} v\dd x
		= \int_{\Omega} \tilde{a}(t,x) v \dd x,
	\end{equation}
	in the sense of distributions on $(0,T)$ for all $v \in W_0^{1,p}(\Omega)$, the variational equations
	\begin{equation}
		\label{penalty-thetatilde}
		\begin{aligned}
			&\rho c\int_{0}^{L}\int_{\Omega}\left(\dfrac{\dd \tilde{\Theta}_\ell}{\dd t}+U_1(0)\dfrac{\partial\tilde{\Theta}_\ell}{\partial x_1}+U_2(0)\dfrac{\partial\tilde{\Theta}_\ell}{\partial x_2}+v_z(0)\dfrac{\partial\tilde{\Theta}_\ell}{\partial z}\right) \Xi \dd x \dd z\\
			&\quad+\kappa\int_{0}^{L}\int_{\Omega}\nabla\tilde{\Theta}_\ell\cdot\nabla\Xi\dd x \dd z-\dfrac{1}{\ell}\int_{0}^{L}\int_{\Omega}\{\tilde{\Theta}_\ell\}^{-}\Xi \dd x \dd z+\dfrac{1}{\ell}\int_{0}^{L}\int_{\Omega}\{\tilde{\Theta}_\ell-\Theta_{\textup{m}}\}^{+}\Xi \dd x \dd z\\
			&= 2(\rho g C_0)^p\int_{0}^{L}\int_{\Omega}A(\cdot,\cdot,\tilde{\Theta}_\ell)\left(\left\{|u_\ell|^{\frac{p-1}{2p}}-z\right\}^{+}\right)^p \Xi\dd x \dd z-\int_{\Omega}q_{\textup{geo}}^\perp\Xi\dd x,
		\end{aligned}
	\end{equation}
	in the sense of distributions on $(0,T)$ for all $\Xi\in V$, as well as the following initial conditions
	\begin{equation*}
		\begin{aligned}
			u_\ell(0)=u_0&,\quad\textup{ a.e. in }\Omega,\\
			\tilde{\Theta}_\ell(0)=\tilde{\Theta}_0&,\quad\textup{ a.e. in }\Omega\times (0,L),
		\end{aligned}
	\end{equation*}
	for some prescribed non-zero $u_0 \in K_{\textup{surf}} \cap \mathcal{C}^1(\overline{\Omega})$ and $\tilde{\Theta}_0\in K_{\textup{temp}}$.
	\bqed
\end{customprob}

We first establish the existence of solutions for Problem~\ref{Pkappa} by means of a semi-discrete scheme. Define the operator $-\Delta_b:V\to V'$ by:
\begin{equation*}
	\langle-\Delta_b\Xi,\Theta\rangle_{V',V}:=\int_{\Omega\times (0,L)}\nabla\Xi\cdot\nabla\Theta\dd x\dd z+\int_{\Omega}\dfrac{q_{\textup{geo}^\perp}}{\kappa}\Theta\dd x,
\end{equation*}
for all $\Xi,\Theta\in V$. To establish the monotonicity of $-\Delta_b$, observe that for each $\Xi,\Theta\in V$
\begin{equation*}
	\langle-\Delta_b\Xi-(-\Delta_b\Theta),\Xi-\Theta\rangle_{V',V}=\int_{0}^{L}\int_{\Omega}|\nabla(\Xi-\Theta)|^2\dd x\dd z\ge \dfrac{\|\Xi-\Theta\|_{V}^2}{1+2L^2},
\end{equation*}
where the last estimate is due to Lemma~\ref{PFV}. The same strategy can be employed to establish the coerciveness of the operator $-\Delta_b$.

To begin with, we divide the interval $[0,T]$ into $\gls{Nodes}$ sub-intervals of length $\gls{k}:=T/N$. This gives rise to $N+1$ equally-spaced nodes in the interval $[0,T]$.
For each $0 \le n \le N-1$, consider the following semi-discretisation of Problem~\ref{Pkappa}.
\begin{customprob}{$\mathcal{P}_\ell^{n+1}$}
	\label{Pkappa:seq}
	Given $u_{\ell,k}^{n} \in W_0^{1,p}(\Omega)$ and $\tilde{\Theta}_{\ell,k}^{n}\in V$, find a function $\gls{un+1}\in W_0^{1,p}(\Omega)$ and a function $\gls{Tn+1}\in V$ that satisfy the following variational equations:
	\begin{gather}
		\begin{aligned}
			&\dfrac{1}{k}\left\{|u_{\ell,k}^{n+1}|^{\alpha-2} u_{\ell,k}^{n+1} - |u_{\ell,k}^{n}|^{\alpha-2}u_{\ell,k}^{n}\right\}-\mu(\tilde{\Theta}_{\ell,k}^{n+1})\nabla \cdot \left(|\nabla u_{\ell,k}^{n+1}|^{p-2} \nabla u_{\ell,k}^{n+1}\right)
			-\dfrac{\mu(\tilde{\Theta}_{\ell,k}^{n+1})}{\ell}\{u_{\ell,k}^{n+1}\}^{-}\\
			&=\dfrac{1}{k} \int_{nk}^{(n+1)k} \tilde{a}(t) \dd t,\quad \textup{ in }W^{-1,p'}(\Omega),
		\end{aligned}\label{penalty:seq}\\[2ex]
		\begin{aligned}
			&\rho c\dfrac{\tilde{\Theta}_{\ell,k}^{n+1}-\tilde{\Theta}_{\ell,k}^{n}}{k}+\rho c (\bm{U}(0),v_z(0))\cdot\nabla\tilde{\Theta}_{\ell,k}^{n+1}+\kappa(-\Delta_b)\tilde{\Theta}_{\ell,k}^{n+1}
			-\dfrac{\{\tilde{\Theta}_{\ell,k}^{n+1}\}^{-}}{\ell}+\dfrac{\{\tilde{\Theta}_{\ell,k}^{n+1}-\Theta_{\textup{m}}\}^{+}}{\ell}\\
			&=2(\rho g C_0)^p A(\tilde{\Theta}_{\ell,k}^{n})\left(\left\{|u_{\ell,k}^{n}|^{\frac{p-1}{2p}}-z\right\}^{+}\right)^p,\quad \textup{ in }V',
		\end{aligned}\label{FD:temp}
	\end{gather}
	where $u_{\ell,k}^{0}:=u_0 \in K_{\textup{surf}}$ and $\tilde{\Theta}_{\ell,k}^{0}:=\tilde{\Theta}_0\in K_{\textup{temp}}$ are the prescribed element appearing in the initial conditions for Problem~\ref{Pkappa}.
	\bqed
\end{customprob}

The next lemma establishes a regularity result for the right-hand side of~\eqref{FD:temp}.
\begin{lemma}
	\label{lem:hlkn}
	For each integer $n=0,\dots,N-1$, define:
	\begin{equation*}
		h_{\ell,k}^n:=2(\rho g C_0)^p A(\tilde{\Theta}_{\ell,k}^{n})\left(\left\{|u_{\ell,k}^{n}|^{\frac{p-1}{2p}}-z\right\}^{+}\right)^p.
	\end{equation*}
	
	Then, it results that $h_{\ell,k}^n\in L^\infty(\Omega\times (0,L))$ and, \emph{a fortiori}, in $V'$.
\end{lemma}
\begin{proof}
	Since $u_{\ell,k}^{n}\in W^{1,p}_0(\Omega)$, $\Omega$ is a Lipschitz domain in $\mathbb{R}^2$ and $p>2$, the continuity of the positive part (Lemma~\ref{lem:1} and Remark~\ref{rem:neg}) and the continuity of the mappings $t\in\mathbb{R}^+\cup\{0\}\mapsto t^\frac{p-1}{2p}$ and $t\in\mathbb{R}^+\cup\{0\}\mapsto t^p$ give:
	\begin{equation*}
		\left(\left\{|u_{\ell,k}^{n}|^\frac{p-1}{2p}-z\right\}^{+}\right)^p\in\mathcal{C}^0(\overline{\Omega}\times[0,L]).
	\end{equation*}
	
	The definition of $A$ in~\eqref{arrhenius} in turn gives that $A(\tilde{\Theta}_{\ell,k}^{n})\in L^\infty(\Omega\times (0,L))$, thus verifying that $h_{\ell,k}^{n}\in L^\infty(\Omega\times (0,L))$. Thanks to the Sobolev embedding theorem (cf., e.g., Theorem~8.4-1 in~\cite{Ciarlet2025}) and to the density of $V$ in $L^6(\Omega\times (0,L))$, we obtain that $L^{6/5}(\Omega\times (0,L))\hookrightarrow V'$. The fact that $h_{\ell,k}^{n}\in L^\infty(\Omega\times (0,L))$ is sufficient to deduce that $h_{\ell,k}^{n}\in V'$.
\end{proof}

The following existence-and-uniqueness result can be established.
\begin{theorem}
	\label{thm:1}
	Let $T>0$, $\Omega \subset \mathbb{R}^2$ and $p$ be as in Section~\ref{Sec:2} and let $\alpha$ be as in~\eqref{alpha}. Let $\ell>0$ be given, let $N \ge 1$ be an integer, and define $k:=T/N$. Assume that~$(H\ref{H1})$--$(H\ref{H4})$ hold and that the positive constant $L$ satisfies the following upper bound:
	\begin{equation}
		\label{L}
		L\le\dfrac{\sqrt{2}\kappa}{4\rho c\|(\bm{U}(0),v_z(0))\|_{\bm{\mathcal{C}}^0(\overline{\Omega}\times[0,L])}}.
	\end{equation}
	
	Then, for each integer $0 \le n \le N-1$, it results that Problem~\ref{Pkappa:seq} admits a unique solution $(u_{\ell,k}^{n+1},\tilde{\Theta}_{\ell,k}^{n+1}) \in W_0^{1,p}(\Omega)\times V$.
\end{theorem}
\begin{proof}
	The proof is divided into two parts, numbered $(i)$ and $(ii)$.
	
	$(i)$ \emph{The variational equations~\eqref{FD:temp} admit a unique solution $\tilde{\Theta}_{\ell,k}^{n+1}\in V$.}
	Consider the operator $\tilde{T}_\ell:V\to V'$ defined by:
	\begin{equation*}
		\tilde{T}_\ell\Xi:=\dfrac{\rho c}{k} \Xi +\rho c (\bm{U}(0),v_z(0))\cdot\nabla\Xi+\kappa(-\Delta_b)\Xi-\dfrac{\{\Xi\}^{-}}{\ell}+\dfrac{\{\Xi-\Theta_{\textup{m}}\}^{+}}{\ell}.
	\end{equation*}
	
	The operator $\tilde{T}_\ell$ is hemi-continuous, as each of its terms is hemi-continuous; note in particular that the advection term is linear, the penalty terms are Lipschitz continuous by Lemma~\ref{lem:1} and Remark~\ref{rem:neg}.
	
	Since the advection term is not, in general, monotone, we discuss the strict monotonicity of the operator $\tilde{T}_\ell$ separately. Given $\Theta,\Xi\in V$, an application of Lemma~\ref{lem:1}, Remark~\ref{rem:neg}, Lemma~\ref{PFV} and~\eqref{L} give:
	\begin{equation*}
		\begin{aligned}
			&\langle\tilde{T}_\ell\Theta-\tilde{T}_\ell\Xi,\Theta-\Xi\rangle_{V',V}=\dfrac{\rho c}{k}\int_{\Omega\times (0,L)}|\Theta-\Xi|^2\dd z\dd x\\
			&\quad+\rho c\int_{\Omega\times(0,L)}((\bm{U}(0),v_z(0))\cdot\nabla(\Theta-\Xi))(\Theta-\Xi)\dd z\dd x\\
			&\quad+\kappa\int_{\Omega\times (0,L)}|\nabla(\Theta-\Xi)|^2\dd z\dd x+\dfrac{1}{\ell}\int_{\Omega\times (0,L)}\left((-\{\Theta\}^{-})-(-\{\Xi\}^{-})\right)(\Theta-\Xi)\dd z\dd x\\
			&\quad+\dfrac{1}{\ell}\int_{\Omega\times (0,L)}\left(\{\Theta-\Theta_{\textup{m}}\}^{+}-\{\Xi-\Theta_{\textup{m}}\}^{+}\right)(\Theta-\Xi)\dd z\dd x\\
			&\ge\dfrac{\rho c}{k}\|\Theta-\Xi\|_{L^2(\Omega\times (0,L))}^2+\left(\kappa-\sqrt{2}\rho c L\|(\bm{U}(0),v_z(0))\|_{\bm{\mathcal{C}}^{0}(\overline{\Omega})}\right)\|\nabla(\Theta-\Xi)\|_{\bm{L}^2(\Omega\times(0,L))}^2\\
			&\ge\min\left\{\dfrac{\rho c}{k},\dfrac{\kappa}{2}\right\}\|\Theta-\Xi\|_{V}^2.
		\end{aligned}
	\end{equation*}
	
	Let us show that the operator $\tilde{T}_\ell$ is coercive. An application of Young's inequality~\cite{Young1912} and the positiveness of $\Theta_{\textup{m}}$ give:
	\begin{equation}
		\label{ex:un-step-1}
		\begin{aligned}
			&\langle \tilde{T}_\ell \Xi, \Xi \rangle_{V',V}\ge \dfrac{\rho c}{k}\|\Xi\|_{L^2(\Omega\times (0,L))}^2+ \rho c\int_{\Omega\times (0,L)}((\bm{U}(0),v_z(0))\cdot\nabla\Xi)\Xi\dd x\dd z\\
			&\quad+\kappa\int_{\Omega\times (0,L)}|\nabla\Xi|^2\dd x \dd z+\int_{\Omega}q_{\textup{geo}}^\perp\Xi\dd x+\dfrac{\|\{\Xi\}^{-}\|_{L^2(\Omega\times (0,L))}^2}{\ell}+\dfrac{\|\{\Xi-\Theta_{\textup{m}}\}^{+}\|_{L^2(\Omega\times (0,L))}^2}{\ell}\\
			&\ge \dfrac{\rho c}{k} \|\Xi\|_{L^2(\Omega\times (0,L))}^2+\rho c \int_{\Omega\times (0,L)}((\bm{U}(0),v_z(0))\cdot\nabla\Xi)\Xi\dd x\dd z\\
			&\quad+\kappa\int_{\Omega\times (0,L)}|\nabla\Xi|^2\dd x \dd z+\dfrac{\|\{\Xi\}^{-}\|_{L^2(\Omega\times (0,L))}^2}{\ell}+\dfrac{\|\{\Xi-\Theta_{\textup{m}}\}^{+}\|_{L^2(\Omega\times (0,L))}^2}{\ell}\\
			&\quad-\dfrac{\hat{C}_{\textup{tr}}^2}{2\min\left\{\dfrac{\rho c}{k},\dfrac{\kappa}{2}\right\}}\int_{\Omega}|q_{\textup{geo}}^\perp|^2\dd x-\dfrac{\min\left\{\dfrac{\rho c}{k},\dfrac{\kappa}{2}\right\}}{2}\|\Xi\|_{V}^2,
		\end{aligned}
	\end{equation}
	where $\gls{Ctr}>0$ is the \emph{continuity constant of the trace operator $\textup{tr}:V\to\Omega\times\{0\}$} (cf., e.g., Chapter~18 in~\cite{Leoni2017}).
	
	An application of the Cauchy--Schwarz inequality, Lemma~\ref{PFV} and~\eqref{L} gives:
	\begin{equation}
		\label{ex:un-step-2}
		\begin{aligned}
			&\rho c\int_{\Omega\times (0,L)}((\bm{U}(0),v_z(0))\cdot\nabla\Xi)\Xi\dd x\dd z\\
			&\ge-\sqrt{2}L\rho c\|(\bm{U}(0),v_z(0))\|_{\bm{\mathcal{C}}^0(\overline{\Omega}\times[0,L])}\|\nabla\Xi\|_{\bm{L}^2(\Omega\times (0,L))}^2
			\ge-\dfrac{\kappa}{2}\|\nabla\Xi\|_{\bm{L}^2(\Omega\times (0,L))}^2.
		\end{aligned}
	\end{equation}
	
	Combining~\eqref{ex:un-step-1} and~\eqref{ex:un-step-2} gives
	\begin{equation*}
		\langle\tilde{T}_\ell \Xi, \Xi \rangle_{V',V}\ge\dfrac{1}{2}\min\left\{\dfrac{\rho c}{k},\dfrac{\kappa}{2}\right\}\|\Xi\|_{V}^2-\dfrac{\hat{C}_{\textup{tr}}^2}{2\min\left\{\dfrac{\rho c}{k},\dfrac{\kappa}{2}\right\}}\int_{\Omega}|q_{\textup{geo}}^\perp|^2\dd x,
	\end{equation*}
	so that $\langle \tilde{T}_\ell \Xi, \Xi \rangle_{V',V}/\|\Xi\|_{V}\to\infty$ as $\|\Xi\|_{V}\to\infty$ and the sought coerciveness is established. The latter conclusion and the fact that $h_{\ell,k}^{n}\in V'$ (Lemma~\ref{lem:hlkn}) put us in position to apply the Minty--Browder theorem (cf., e.g., Theorem~12.5-3 in~\cite{Ciarlet2025}) so as to infer the existence of a unique element $\tilde{\Theta}_{\ell,k}^{n+1}\in V$ satisfying~\eqref{FD:temp}.
	
	$(ii)$ \emph{The variational equations~\eqref{penalty:seq} admit a unique solution $u_{\ell,k}^{n+1}\in W^{1,p}_0(\Omega)$.} Observe that, thanks to item~$(i)$, the existence and uniqueness for $\tilde{\Theta}_{\ell,k}^{n+1}$ has already been established.
	This puts us in position to consider the operator $A_\ell :W_0^{1,p}(\Omega) \to W^{-1,p'}(\Omega)$ defined by:
	\begin{equation*}
		\label{Akappa}
		A_\ell(v):=\dfrac{|v|^{\alpha-2} v}{k} -\mu(\tilde{\Theta}_{\ell,k}^{n+1})\nabla \cdot \left(|\nabla v|^{p-2} \nabla v\right)-\dfrac{\mu(\tilde{\Theta}_{\ell,k}^{n+1})}{\ell}\{v\}^{-},\quad\textup{ for all }v\in W_0^{1,p}(\Omega).
	\end{equation*}
	
	The operator $A_\ell$ is hemi-continuous, as each of its terms is hemi-continuous. By Lemma~\ref{lem:1} and Lemma~\ref{lem:7}, the operator $A_\ell$ is strictly monotone.
	Finally, an application of the Poincar\'e--Friedrichs inequality (cf., e.g., Theorem~8.3-3$(a)$ in~\cite{Ciarlet2025}) and the monotonicity of the negative part operator (Lemma~\ref{lem:1}) give:
	\begin{equation*}
		\dfrac{\langle A_\ell v, v \rangle_{W^{-1,p'}(\Omega), W_0^{1,p}(\Omega)}}{\|v\|_{W^{1,p}(\Omega)}}\ge\dfrac{k^{-1}\|v\|_{L^\alpha(\Omega)}^\alpha + \mu_1 \|\nabla v\|_{\bm{L}^p(\Omega)}^p}{\|v\|_{W^{1,p}(\Omega)}}\ge \dfrac{\mu_1}{\gls{cP}^p} \|v\|_{W^{1,p}(\Omega)}^{p-1}=\dfrac{\mu_1 p}{(\textup{diam }\Omega)^p}\|v\|_{W^{1,p}(\Omega)}^{p-1}.
	\end{equation*}
	
	Observing that the term on the right diverges as $\|v\|_{W^{1,p}(\Omega)} \to \infty$, we obtain that the operator $A_\ell$ is coercive.
	An application of the Minty--Browder theorem (cf., e.g., Theorem~12.5-3 in~\cite{Ciarlet2025}) ensures that the numerical scheme in~\eqref{penalty:seq} admits a unique solution $u_{\ell,k}^{n+1} \in W_0^{1,p}(\Omega)$ for~\eqref{penalty:seq}.
	
	The proof is complete.
\end{proof}

Given an integer $N\in\mathbb{N}$, and recalling that $k=T/N$, define the mapping $\gls{right-shift-u}:(0,T) \to W_0^{1,p}(\Omega)$ by
\begin{equation}
	\label{Pil}
	\Pi_k \bm{u}_{\ell,k}(t):=u_{\ell,k}^{n+1}, \quad\textup{ if } n k < t \le (n+1) k,
\end{equation}
and define the mapping $\gls{right-shift-T}:(0,T) \to V$ by:
\begin{equation}
	\label{PilT}
	\Pi_k \tilde{\bm{\Theta}}_{\ell,k}(t):=\tilde{\Theta}_{\ell,k}^{n+1}, \quad\textup{ if } n k < t \le (n+1) k.
\end{equation}

It is also helpful to consider the left shift by $k$ of the mapping $\Pi_k\bm{u}_{\ell,k}$, that we denote and define by
\begin{equation}
	\label{PilL}
	\gls{left-shift-u}(t):=\Pi_k\bm{u}_{\ell,k}(t-k), \quad\textup{ if } n k < t \le (n+1) k,
\end{equation}
where we let $\Pi_k\bm{u}_{\ell,k}(t)=u_0$ if $t\le 0$.

In the same fashion, we denote and define the \emph{left shift by $k$ of the mapping $\Pi_k\tilde{\bm{\Theta}}_{\ell,k}$} by
\begin{equation}
	\label{PilTL}
	\gls{left-shift-T}(t):=\Pi_k\tilde{\bm{\Theta}}_{\ell,k}(t-k), \quad\textup{ if } n k < t \le (n+1) k,
\end{equation}
where we let $\Pi_k\tilde{\bm{\Theta}}_{\ell,k}(t)=\tilde{\Theta}_0$ if $t\le 0$.

Given $\bm{v}_{\ell,k}=\{v_{\ell,k}^{n}\}_{n=0}^{N}$, the function $\gls{Dk}(\Pi_k \bm{v}_{\ell,k}):(0,T) \to W_0^{1,p}(\Omega)$ is defined by
\begin{equation*}
	\label{finite-difference}
	D_k(\Pi_k \bm{v}_{\ell,k})(t):=\dfrac{v_{\ell,k}^{n+1}-v_{\ell,k}^{n}}{k},\quad\textup{ for all } n k < t \le (n+1)k, \quad 0 \le n \le N-1.
\end{equation*}

We now recall a variant of the Aubin--Lions--Simon compactness theorem, established by Dreher \& J\"{u}ngel~\cite{DJ12}, which is helpful for studying finite-difference equations.
\begin{lemma}[Theorem~1 in~\cite{DJ12}]
	\label{ALS:dis}
	Let $X_0$, $X_1$ and $X_2$ be Banach spaces such that $X_0\hookrightarrow\hookrightarrow X_1\hookrightarrow X_2$, let $0<T<\infty$ and let $N\ge 1$ be an arbitrary integer. Let either $1\le p<\infty$ and $r=1$, or $p=\infty$ and $r>1$. Let $k:=T/N$ and consider the vector $\bm{v}_k:=\{v_k^{n}\}_{n=0}^{N}$ of elements of $X_0$, and define the function $\Pi_k\bm{v}_k:(0,T)\to X_0$ by:
	\begin{equation*}
		(\Pi_k\bm{v}_k)(t):=v_k^{n+1},\quad\textup{ for a.a. }nk<t\le(n+1)k.
	\end{equation*}
	
	Assume that there exists a constant $C>0$ independent of $k$ such that:
	\begin{equation*}
		k\sum_{n=0}^{N-1}\left\|\dfrac{v_k^{n+1}-v_k^{n}}{k}\right\|_{X_2}^r+k\sum_{n=0}^{N}\|v_k^{n}\|_{X_0}^p\le C.
	\end{equation*}
	
	Then, it results that:
	\begin{itemize}
		\item[$(a)$] If $1\le p<\infty$, the sequence $\{\Pi_k\bm{v}_k\}_{N=1}^\infty$ is relatively compact in $L^p(0,T;X_1)$;
		\item[$(b)$] If $p=\infty$, the sequence $\{\Pi_k\bm{v}_k\}_{N=1}^\infty$ is relatively compact in $L^q(0,T;X_1)$ for each $1\le q<\infty$ and the limit belongs to $\mathcal{C}^0([0,T];X_1)$.
	\end{itemize}
	\qed
\end{lemma}

We also provide the statement for the discrete version of Lemma~\ref{Dub}, originally established by P.-A. Raviart in~\cite{Raviart1967} and then refined by Chen, J\"{u}ngel \& Liu in~\cite{CJL14}.
\begin{lemma}[Lemma~1.4 of~\cite{Raviart1970} and Theorem~2 in~\cite{CJL14}]
	\label{Dub:dis}
	Let $A_0$ and $A_1$ be Banach spaces such that $A_0 \hookrightarrow A_1$, let $0<T<\infty$ and let $N\ge 1$ be an arbitrary integer.
	Let $S \subset A_0$ be such that $\lambda S \subset S$, for all $\lambda \in \mathbb{R}$.
	Assume that the set $S$ is endowed with the definite homogeneous gauge $M:S\to \mathbb{R}$, having the following properties:
	\begin{itemize}
		\item $M(v) \ge 0$, for all $v \in S$,
		\item $M(\lambda v)=|\lambda| M(v)$, for all $v \in S$ and all $\lambda \in \mathbb{R}$.
	\end{itemize}
	
	Assume that the set $\mathscr{M}:=\{v \in S;M(v)\le 1\}$ is relatively compact in $A_0$.
	Let $k:=T/N$, and let $\bm{v}_k:=\{v_k^{n}\}_{n=0}^{N}$ be a vector of elements in $S$. Assume that there exists a constant $C>0$ independent of $N$ such that
	$$
	k\sum_{n=0}^{N} [M(v_k^{n})]^{q_0} \le C,
	$$
	and
	$$
	k\sum_{n=0}^{N-1}\left\|\dfrac{v_k^{n+1}-v_k^{n}}{k}\right\|_{A_1}^{q_1} \le C,
	$$
	where $1 \le q_0 \le \infty$ and $1 \le q_1 \le \infty$, and the pairs $(q_0,q_1)=(1,\infty)$ and $(q_0,q_1)=(\infty,1)$ cannot be attained. Define the function $\Pi_k\bm{v}_k:(0,T)\to A_0$ by:
	\begin{equation*}
		(\Pi_k\bm{v}_k)(t):=v_k^{n+1},\quad\textup{ for a.a. }nk<t\le(n+1)k.
	\end{equation*}
	
	Then, it results that:
	\begin{itemize}
		\item[$(a)$] If $1\le q_0<\infty$, the sequence $\{\Pi_k\bm{v}_k\}_{N=1}^\infty$ is relatively compact in $L^{q_0}(0,T;A_0)$;
		\item[$(b)$] If $q_0=\infty$, the sequence $\{\Pi_k\bm{v}_k\}_{N=1}^\infty$ is relatively compact in $L^q(0,T;A_0)$ for each $1\le q<\infty$ and the limit belongs to $\mathcal{C}^0([0,T];A_0)$.
	\end{itemize}
	\qed
\end{lemma}

\begin{remark}
	\label{rem:left-shift}
	The conclusion of Lemma~\ref{ALS:dis} and Lemma~\ref{Dub:dis} also holds for the left shift of the piecewise constant-in-time functions $\Pi_k \bm{v}_k$, upon considering a suitable extension (viz., e.g., \eqref{PilL} and~\eqref{PilTL}) (cf., e.g., Theorem~2 in~\cite{CJL14}).
	\bqed
\end{remark}

\section{A priori estimates}
\label{Sec:3}

In this section we discuss the convergence of the sequences $\{\Pi_k \bm{u}_\ell\}_{N=1}^\infty$ and $\{\Pi_k \tilde{\bm{\Theta}}_\ell\}_{N=1}^\infty$ as $N \to\infty$. To this aim, we need to establish some \emph{a priori} estimates. While the recovery of the energy estimates for~\eqref{penalty:seq} follows the strategy of Theorem~4.2 in~\cite{PT23}, the presence of the temperature in~\eqref{penalty:seq} and the form of the equations~\eqref{FD:temp} require additional care.
\begin{theorem}
	\label{thm:2}
	Let $T>0$, $\Omega \subset \mathbb{R}^2$ and $p$ be as in Section~\ref{Sec:2} and let $\alpha$ be as in~\eqref{alpha}. Let $\ell>0$ be given and intended to go to zero, let $N \ge 1$ be an integer, and define $k:=T/N$.
	Assume that~($H$\ref{H1})--($H$\ref{H4}) hold and that the positive constant $L$ satisfies the upper bound~\eqref{L}.
	
	Then, it results:
	\begin{gather}
		\left\{\max_{0\le n\le N}\|u_{\ell,k}^{n}\|_{L^\alpha(\Omega)}\right\}_{N=1}^\infty \textup{ is bounded independently of }N\textup{ and }\ell,\label{est:1}\\
		\left\{k \sum_{n=0}^{N} \|u_{\ell,k}^{n}\|_{W^{1,p}(\Omega)}^p\right\}_{N=1}^\infty \textup{ is bounded independently of }N\textup{ and }\ell,\label{est:2}\\
		\left\{\max_{0\le n\le N} \left\||u_{\ell,k}^{n}|^\frac{\alpha-2}{2} u_{\ell,k}^{n}\right\|_{L^2(\Omega)}\right\}_{N=1}^\infty \textup{ is bounded independently of }N\textup{ and }\ell,\label{est:3:1}\\
		\begin{aligned}
			&\left\{k \sum_{n=0}^{N-1} \left\|\dfrac{|u_{\ell,k}^{n+1}|^{\frac{\alpha-2}{2}}u_{\ell,k}^{n+1} -|u_{\ell,k}^{n}|^{\frac{\alpha-2}{2}}u_{\ell,k}^{n}}{k}\right\|_{L^2(\Omega)}^2\right\}_{N=1}^\infty\\
			&\qquad\textup{ is bounded independently of }N\textup{ and }\ell,
		\end{aligned}\label{est:3}\\
		\left\{\max_{0\le n\le N} \|u_{\ell,k}^{n}\|_{W^{1,p}(\Omega)}\right\}_{N=1}^\infty \textup{ is bounded independently of }N\textup{ and }\ell,\label{est:4}\\
		\begin{aligned}
			&\left\{\max_{0\le n\le N-1} \left\|\dfrac{|u_{\ell,k}^{n+1}|^{\alpha-2}u_{\ell,k}^{n+1} - |u_{\ell,k}^{n}|^{\alpha-2}u_{\ell,k}^{n}}{k}\right\|_{W^{-1,p'}(\Omega)}\right\}_{N=1}^\infty\\
			&\qquad\textup{ is bounded independently of }N,
		\end{aligned}\label{est:5}\\
		\left\{\max_{0\le n\le N} \left\||u_{\ell,k}^{n}|^{\alpha-2} u_{\ell,k}^{n}\right\|_{L^{\alpha'}(\Omega)}\right\}_{N=1}^\infty \textup{ is bounded independently of }N\textup{ and }\ell,\label{est:5:1}\\
		\left\{\max_{0\le n\le N}\|\tilde{\Theta}_{\ell,k}^{n}\|_{L^2(\Omega\times (0,L))}\right\}_{N=1}^\infty \textup{ is bounded independently of }N\textup{ and }\ell,\label{est:6}\\
		\left\{k \sum_{n=0}^{N} \|\tilde{\Theta}_{\ell,k}^{n}\|_{V}^2\right\}_{N=1}^\infty \textup{ is bounded independently of }N\textup{ and }\ell,\label{est:7}\\
		\left\{\max_{0\le n\le N}\|\tilde{\Theta}_{\ell,k}^{n}\|_{V}\right\}_{N=1}^\infty \textup{ is bounded independently of }N\textup{ and }\ell,\label{est:8}\\
		\left\{k\sum_{n=0}^{N-1}\left\|\dfrac{\tilde{\Theta}_{\ell,k}^{n+1}-\tilde{\Theta}_{\ell,k}^{n}}{k}\right\|_{L^2(\Omega\times (0,L))}^2\right\}_{N=1}^\infty \textup{ is bounded independently of }N\textup{ and }\ell.\label{est:9}
	\end{gather}
\end{theorem}
\begin{proof}
	The proof is divided into seven steps, numbered $(i)$--$(vii)$.
	
	$(i)$ \emph{Proof of~\eqref{est:1} and~\eqref{est:2}.}
	Multiply~\eqref{penalty:seq} by $u_{\ell,k}^{n+1}$ in the sense of the duality between $W^{-1,p'}(\Omega)$ and $W_0^{1,p}(\Omega)$, obtaining:
	\begin{equation*}
		\begin{aligned}
			&\dfrac{1}{k} \int_{\Omega}\left(|u_{\ell,k}^{n+1}|^{\alpha-2} u_{\ell,k}^{n+1} -|u_{\ell,k}^{n}|^{\alpha-2} u_{\ell,k}^{n}\right) u_{\ell,k}^{n+1} \dd x
			+\mu(\tilde{\Theta}_{\ell,k}^{n+1})\int_{\Omega}|\nabla u_{\ell,k}^{n+1}|^{p-2} \nabla u_{\ell,k}^{n+1} \cdot \nabla u_{\ell,k}^{n+1}\dd x\\
			&\quad-\dfrac{\mu(\tilde{\Theta}_{\ell,k}^{n+1})}{\ell} \int_{\Omega} \{u_{\ell,k}^{n+1}\}^{-} u_{\ell,k}^{n+1} \dd x =\dfrac{1}{k}\int_{\Omega} \left(\int_{nk}^{(n+1)k} \tilde{a}(t) \dd t\right) u_{\ell,k}^{n+1} \dd x.
		\end{aligned}
	\end{equation*}
	
	Using the first estimate in Lemma~\ref{lem:7} with $r=\alpha$, $\xi=u_{\ell,k}^{n+1}$ and $\eta=u_{\ell,k}^{n}$, the Poincar\'e--Friedrichs inequality (cf., e.g., Theorem~8.3-3$(a)$ in~\cite{Ciarlet2025}), Bochner's theorem (cf., e.g., Theorem~8.9 of~\cite{Leoni2017}), H\"{o}lder's inequality, Young's inequality (cf., e.g., \cite{Young1912} and page~92 in~\cite{Brez11}), the fact that $p>2$, and~\eqref{mu0} we obtain:
	\begin{equation*}
		\begin{aligned}
			&\dfrac{1}{\alpha' k}(\|u_{\ell,k}^{n+1}\|_{L^\alpha(\Omega)}^\alpha - \|u_{\ell,k}^{n}\|_{L^\alpha(\Omega)}^\alpha)+\dfrac{\mu_1 p}{p+(\textup{diam }\Omega)^p}\|u_{\ell,k}^{n+1}\|_{W^{1,p}(\Omega)}^p+\dfrac{\mu_1}{\ell}\|\{u_{\ell,k}^{n+1}\}^{-}\|_{L^2(\Omega)}^2\\
			&\le\left(\dfrac{\mu_1 p}{2(p+(\textup{diam }\Omega)^p)}\right)^{-\frac{1}{p-1}}k^{p'/p}\left(\dfrac{1}{k^{p'}}\int_{nk}^{(n+1)k} \|\tilde{a}(t)\|_{W^{-1,p'}(\Omega)}^{p'}\dd t\right)+\dfrac{\mu_1 p}{2(p+(\textup{diam }\Omega)^p)}\|u_{\ell,k}^{n+1}\|_{W^{1,p}(\Omega)}^p.
		\end{aligned}
	\end{equation*}
	
	Multiplying both sides of the latter estimate by $(k \alpha')$, and summing over $0 \le n \le s-1$, where $1 \le s \le N$, gives
	\begin{equation*}
		\begin{aligned}
			&\sum_{n=0}^{s-1} \left\{\|u_{\ell,k}^{n+1}\|_{L^\alpha(\Omega)}^\alpha-\|u_{\ell,k}^{n}\|_{L^\alpha(\Omega)}^\alpha+\dfrac{\alpha' k \mu_1 p}{2(p+(\textup{diam }\Omega)^p)}\|u_{\ell,k}^{n+1}\|_{W^{1,p}(\Omega)}^p+\dfrac{\alpha' k \mu_1}{\ell}\|\{u_{\ell,k}^{n+1}\}^{-}\|_{L^2(\Omega)}^2\right\}\\
			&\le\alpha'\left(\dfrac{\mu_1 p}{2(p+(\textup{diam }\Omega)^p)}\right)^{-\frac{1}{p-1}}\sum_{n=0}^{s-1} \left(\int_{nk}^{(n+1)k} \|\tilde{a}(t)\|_{W^{-1,p'}(\Omega)}^{p'} \dd t\right),
		\end{aligned}
	\end{equation*}
	and we finally obtain:
	\begin{equation*}
		\begin{aligned}
			&\|u_{\ell,k}^{s}\|_{L^\alpha(\Omega)}^\alpha +\dfrac{\alpha' k \mu_1 p}{2(p+(\textup{diam }\Omega)^p)} \sum_{n=0}^{s-1} \|u_{\ell,k}^{n+1}\|_{W^{1,p}(\Omega)}^p
			+\dfrac{\alpha' k \mu_1}{\ell} \sum_{n=0}^{s-1} \|\{u_{\ell,k}^{n+1}\}^{-}\|_{L^2(\Omega)}^2\\
			&\le\alpha'\left(\dfrac{\mu_1 p}{2(p+(\textup{diam }\Omega)^p)}\right)^{-\frac{1}{p-1}}\sum_{n=0}^{s-1} \left(\int_{nk}^{(n+1)k} \|\tilde{a}(t)\|_{W^{-1,p'}(\Omega)}^{p'} \dd t\right)+\|u_0\|_{L^\alpha(\Omega)}^\alpha,
		\end{aligned}
	\end{equation*}
	that establishes~\eqref{est:1} and~\eqref{est:2}. In particular, the estimates~\eqref{est:1} and~\eqref{est:2} and the fact that $u_0\in K_{\textup{surf}}$ imply that there exists a constant independent of $N$ and $\ell$ such that:
	\begin{equation}
		\label{penalty-est-1}
		\begin{aligned}
			&\dfrac{k}{\ell}\sum_{n=0}^N\|\{u_{\ell,k}^{n}\}^{-}\|_{L^2(\Omega)}^2\le\dfrac{1}{\mu_1}\left(\dfrac{\mu_1 p}{2(p+(\textup{diam }\Omega)^p)}\right)^{-\frac{1}{p-1}}\sum_{n=0}^{s-1} \left(\int_{nk}^{(n+1)k} \|\tilde{a}(t)\|_{W^{-1,p'}(\Omega)}^{p'} \dd t\right)\\
			&\quad+\dfrac{\|u_0\|_{L^\alpha(\Omega)}^\alpha}{\mu_1 \alpha'}.
		\end{aligned}
	\end{equation}
	
	For the sake of brevity, we denote by $C_1(k)$ the following constant:
	\begin{equation}
		\label{C1}
		C_1(k):=\dfrac{1}{\mu_1}\left(\dfrac{\mu_1 p}{2(p+(\textup{diam }\Omega)^p)}\right)^{-\frac{1}{p-1}}\sum_{n=0}^{s-1} \left(\int_{nk}^{(n+1)k} \|\tilde{a}(t)\|_{W^{-1,p'}(\Omega)}^{p'} \dd t\right)+\dfrac{\|u_0\|_{L^\alpha(\Omega)}^\alpha}{\mu_1\alpha'}.
	\end{equation}
	
	$(ii)$\emph{Proof of~\eqref{est:3:1}}.
	By~\eqref{est:1}, for each $0 \le n \le N$, we have that:
	\begin{equation*}
		\|u_{\ell,k}^{n}\|_{L^\alpha(\Omega)}^\alpha =\int_{\Omega} \left||u_{\ell,k}^{n}|^\frac{\alpha-2}{2} u_{\ell,k}^{n}\right|^2 \dd x =
		\left\||u_{\ell,k}^{n}|^\frac{\alpha-2}{2} u_{\ell,k}^{n}\right\|_{L^2(\Omega)}^2.
	\end{equation*}
	
	Since the left-hand side of the previous identity is bounded independently of $N$ and $\ell$, we immediately infer that the sequence
	$$
	\left\{\max_{0\le n\le N} \left\||u_{\ell,k}^{n}|^\frac{\alpha-2}{2} u_{\ell,k}^{n}\right\|_{L^2(\Omega)}\right\}_{N=1}^\infty,
	$$
	is bounded by $C_1$ defined in~\eqref{C1}, thus establishing the estimate~\eqref{est:3:1}.
	
	$(iii)$\emph{Proof of~\eqref{est:6} and~\eqref{est:7}}. Multiply~\eqref{FD:temp} by $\tilde{\Theta}_{\ell,k}^{n+1}$ in the sense of the duality between $V'$ and $V$, obtaining:
	\begin{equation*}
		\begin{aligned}
			&\dfrac{\rho c}{k}\|\tilde{\Theta}_{\ell,k}^{n+1}\|_{L^2(\Omega\times (0,L))}^2-\dfrac{\rho c}{k}\int_{\Omega\times (0,L)}\tilde{\Theta}_{\ell,k}^{n+1}\tilde{\Theta}_{\ell,k}^{n} \dd x\dd z+\rho c\int_{\Omega\times (0,L)}\left((\bm{U}(0),v_z(0))\cdot\nabla\tilde{\Theta}_{\ell,k}^{n+1}\right)\tilde{\Theta}_{\ell,k}^{n+1} \dd x\dd z\\
			&\quad+\kappa\int_{\Omega\times (0,L)}|\nabla\tilde{\Theta}_{\ell,k}^{n+1}|^2\dd x\dd z+\dfrac{\|\{\tilde{\Theta}_{\ell,k}^{n+1}\}^{-}\|_{L^2(\Omega\times(0,L))}^2}{\ell}+\dfrac{1}{\ell}\int_{\Omega\times (0,L)}\{\tilde{\Theta}_{\ell,k}^{n+1}-\Theta_{\textup{m}}\}^{+} \tilde{\Theta}_{\ell,k}^{n+1}\dd x\dd z\\
			&=2(\rho g C_0)^p\int_{\Omega\times (0,L)}A(\tilde{\Theta}_{\ell,k}^{n})\left(\left\{|u_{\ell,k}^{n}|^\frac{p-1}{2p}-z\right\}^{+}\right)^p\tilde{\Theta}_{\ell,k}^{n+1}\dd x \dd z-\int_{\Omega}q_{\textup{geo}}^\perp \tilde{\Theta}_{\ell,k}^{n+1}\dd x.
		\end{aligned}
	\end{equation*}
	
	An application of the Young inequality~\cite{Young1912}, Lemma~\ref{PFV}, the Cauchy--Schwarz inequality~\cite{Brez11}, \eqref{arrhenius}, \eqref{L}, assumption~$(H\ref{H4})$, the positiveness of $\Theta_{\textup{m}}$ the Sobolev embedding theorem (cf., e.g., Theorem~8.4-1 in~\cite{Ciarlet2025}), Lemma~\ref{lem:2}, estimate~\eqref{est:5} to the equation above gives:
	\begin{equation}
		\label{est-theta-0}
		\begin{aligned}
			&\dfrac{\rho c}{2k}\|\tilde{\Theta}_{\ell,k}^{n+1}\|_{L^2(\Omega\times(0,L))}^2-\dfrac{\rho c}{2k}\|\tilde{\Theta}_{\ell,k}^{n}\|_{L^2(\Omega\times(0,L))}^2+\dfrac{\kappa}{2}\|\nabla\tilde{\Theta}_{\ell,k}^{n+1}\|_{\bm{L}^2(\Omega\times(0,L))}^2\\
			&\le\dfrac{\rho c}{2k}\|\tilde{\Theta}_{\ell,k}^{n+1}\|_{L^2(\Omega\times(0,L))}^2-\dfrac{\rho c}{2k}\|\tilde{\Theta}_{\ell,k}^{n}\|_{L^2(\Omega\times(0,L))}^2\\
			&\quad+\left(\kappa-\rho c\sqrt{2}L\|(\bm{U}(0),v_z(0))\|_{\bm{\mathcal{C}}^0(\overline{\Omega}\times[0,L])}\right)\|\nabla\tilde{\Theta}_{\ell,k}^{n+1}\|_{\bm{L}^2(\Omega\times (0,L))}^2\\
			&\quad+\dfrac{\|\{\tilde{\Theta}_{\ell,k}^{n+1}\}^{-}\|_{L^2(\Omega\times (0,L))}^2}{\ell}+\dfrac{\|\{\tilde{\Theta}_{\ell,k}^{n+1}-\Theta_{\textup{m}}\}^{+}\|_{L^2(\Omega\times (0,L))}^2}{\ell}\\
			&\quad+\underbrace{\dfrac{1}{\ell}\int_{\Omega\times (0,L)}\{\tilde{\Theta}_{\ell,k}^{n+1}-\Theta_{\textup{m}}\}^{+}\Theta_{\textup{m}}\dd x\dd z}_{\ge 0}\\
			&\le\dfrac{16 L^3 A_0^2(\rho g C_0)^{2p}}{\kappa}\|u_{\ell,k}^{n}\|_{L^{p-1}(\Omega)}^{p-1}+\dfrac{\kappa}{4}\|\nabla\tilde{\Theta}_{\ell,k}^{n+1}\|_{\bm{L}^2(\Omega\times (0,L))}^2+\dfrac{2(1+2L^2)}{\kappa}\hat{C}_{\textup{tr}}^2\|q_{\textup{geo}}^\perp\|_{L^2(\Omega)}^2.
		\end{aligned}
	\end{equation}
	
	As a result, we obtain the estimate:
	\begin{equation}
		\label{est-theta-1}
		\begin{aligned}
			&\dfrac{\rho c}{2k}\|\tilde{\Theta}_{\ell,k}^{n+1}\|_{L^2(\Omega\times(0,L))}^2-\dfrac{\rho c}{2k}\|\tilde{\Theta}_{\ell,k}^{n}\|_{L^2(\Omega\times(0,L))}^2+\dfrac{\kappa}{4}\|\nabla\tilde{\Theta}_{\ell,k}^{n+1}\|_{\bm{L}^2(\Omega\times(0,L))}^2\\
			&\le\dfrac{16 L^3 A_0^2(\rho g C_0)^{2p}}{\kappa}\|u_{\ell,k}^{n}\|_{L^{p-1}(\Omega)}^{p-1}+\dfrac{2(1+2L^2)}{\kappa}\hat{C}_{\textup{tr}}^2\|q_{\textup{geo}}^\perp\|_{L^2(\Omega)}^2.
		\end{aligned}
	\end{equation}
	
	Multiply~\eqref{est-theta-1} by $k$ and sum over $n=0,\dots, s-1$, for all $1\le s\le N$. The fact that $k:=T/N$ and an application of Young's inequality~\cite{Young1912}, H\"{o}lder's inequality and~\eqref{est:2} give
	\begin{equation*}
		\begin{aligned}
			&\dfrac{\rho c}{2}\left(\|\tilde{\Theta}_{\ell,k}^{s}\|_{L^2(\Omega\times(0,L))}^2-\|\tilde{\Theta}_0\|_{L^2(\Omega\times(0,L))}^2\right)+\dfrac{\kappa}{4}\sum_{n=1}^{s}k\|\nabla\tilde{\Theta}_{\ell,k}^{n}\|_{\bm{L}^2(\Omega\times(0,L))}^2\\
			&\le\dfrac{8L c_{P,L}^2A_0^2(\rho g C_0)^{2p}}{\kappa}\sum_{n=0}^{s-1}k\|u_{\ell,k}^{n}\|_{L^{p-1}(\Omega)}^{p-1}+\dfrac{2(1+2L^2)T}{\kappa}\hat{C}_{\textup{tr}}^2\|q_{\textup{geo}}^\perp\|_{L^2(\Omega)}^2\\
			&\le \dfrac{16 L^3 |\Omega|^{1/p} A_0^2(\rho g C_0)^{2p}}{\kappa p'}\sum_{n=0}^{s-1}k\|u_{\ell,k}^{n}\|_{W^{1,p}(\Omega)}^p+\dfrac{16 L^3 |\Omega|^{1/p} A_0^2(\rho g C_0)^{2p} T}{\kappa p}\\
			&\quad+\dfrac{2(1+2L^2)T}{\kappa}\hat{C}_{\textup{tr}}^2\|q_{\textup{geo}}^\perp\|_{L^2(\Omega)}^2,
		\end{aligned}
	\end{equation*}
	and the estimate~\eqref{est:6} is thus established. Estimate~\eqref{est:7} immediately follows by Lemma~\ref{PFV}. In particular, from~\eqref{est-theta-0} and the estimates~\eqref{est:6} and~\eqref{est:7}, we infer that there exists a constant $C>0$ independent of $N$ and $\ell$ such that:
	\begin{gather}
		\dfrac{k}{\ell}\sum_{n=0}^N\|\{\tilde{\Theta}_{\ell,k}^{n}\}^{-}\|_{L^2(\Omega\times (0,L))}^2\le C,\label{penalty-est-2}\\
		\dfrac{k}{\ell}\sum_{n=0}^N\|\{\tilde{\Theta}_{\ell,k}^{n+1}-\Theta_{\textup{m}}\}^{+}\|_{L^2(\Omega\times (0,L))}^2\le C.\label{penalty-est-3}
	\end{gather}
	
	$(iv)$\emph{Proof of~\eqref{est:8} and~\eqref{est:9}}. Multiply~\eqref{FD:temp} by $(\tilde{\Theta}_{\ell,k}^{n+1}-\tilde{\Theta}_{\ell,k}^{n})$ in the sense of the duality between $V'$ and $V$, obtaining:
	\begin{equation*}
		\begin{aligned}
			&\dfrac{\rho c}{k}\|\tilde{\Theta}_{\ell,k}^{n+1}-\tilde{\Theta}_{\ell,k}^{n}\|_{L^2(\Omega\times(0,L))}^2+\kappa\int_{\Omega\times (0,L)}\nabla\tilde{\Theta}_{\ell,k}^{n+1}\cdot\nabla(\tilde{\Theta}_{\ell,k}^{n+1}-\tilde{\Theta}_{\ell,k}^{n})\dd x\dd z\\
			&\quad+\rho c\int_{\Omega\times (0,L)}\left((\bm{U}(0),v_z(0))\cdot\nabla\tilde{\Theta}_{\ell,k}^{n+1}\right)(\tilde{\Theta}_{\ell,k}^{n+1}-\tilde{\Theta}_{\ell,k}^{n}) \dd x\dd z+\int_{\Omega}q_{\textup{geo}}^\perp (\tilde{\Theta}_{\ell,k}^{n+1}-\tilde{\Theta}_{\ell,k}^{n})\dd x\\
			&\quad-\dfrac{1}{\ell}\int_{\Omega\times (0,L)}\{\tilde{\Theta}_{\ell,k}^{n+1}\}^{-}(\tilde{\Theta}_{\ell,k}^{n+1}-\tilde{\Theta}_{\ell,k}^{n})\dd x\dd z+\dfrac{1}{\ell}\int_{\Omega\times (0,L)}\{\tilde{\Theta}_{\ell,k}^{n+1}-\Theta_{\textup{m}}\}^{+}(\tilde{\Theta}_{\ell,k}^{n+1}-\tilde{\Theta}_{\ell,k}^{n})\dd x\dd z\\
			&=2(\rho g C_0)^p\int_{\Omega\times (0,L)}A(\tilde{\Theta}_{\ell,k}^{n})\left(\left\{|u_{\ell,k}^{n}|^\frac{p-1}{2p}-z\right\}^{+}\right)^p(\tilde{\Theta}_{\ell,k}^{n+1}-\tilde{\Theta}_{\ell,k}^{n})\dd x \dd z.
		\end{aligned}
	\end{equation*}
	
	An application of Young's inequality~\cite{Young1912} allows us to estimate the previous identity as follows:
	\begin{equation}
		\label{est-theta-2}
		\begin{aligned}
			&\rho ck\left\|\dfrac{\tilde{\Theta}_{\ell,k}^{n+1}-\tilde{\Theta}_{\ell,k}^{n}}{k}\right\|_{L^2(\Omega\times(0,L))}^2+\dfrac{\kappa}{2}\|\nabla\tilde{\Theta}_{\ell,k}^{n+1}\|_{\bm{L}^2(\Omega\times (0,L))}^2-\dfrac{\kappa}{2}\|\nabla\tilde{\Theta}_{\ell,k}^{n}\|_{\bm{L}^2(\Omega\times (0,L))}^2\\
			&\quad-\rho c \|(\bm{U}(0),v_z(0))\|_{\bm{\mathcal{C}}^0(\overline{\Omega}\times[0,L])}^2\left(k\|\nabla\tilde{\Theta}_{\ell,k}^{n+1}\|_{\bm{L}^2(\Omega\times (0,L))}^2\right)-\dfrac{\rho ck}{4}\left\|\dfrac{\tilde{\Theta}_{\ell,k}^{n+1}-\tilde{\Theta}_{\ell,k}^{n}}{k}\right\|_{L^2(\Omega\times (0,L))}^2\\
			&\quad-\dfrac{1}{\ell}\int_{\Omega\times (0,L)}\{\tilde{\Theta}_{\ell,k}^{n+1}\}^{-}(\tilde{\Theta}_{\ell,k}^{n+1}-\tilde{\Theta}_{\ell,k}^{n})\dd x\dd z\\
			&\quad+\dfrac{1}{\ell}\int_{\Omega\times (0,L)}\{\tilde{\Theta}_{\ell,k}^{n+1}-\Theta_{\textup{m}}\}^{+}(\tilde{\Theta}_{\ell,k}^{n+1}-\tilde{\Theta}_{\ell,k}^{n})\dd x\dd z\\
			&\le \dfrac{4(\rho g C_0)^{2p}A_0^2 L k}{\rho c} \|u_{\ell,k}^{n}\|_{L^{p-1}(\Omega)}^{p-1}+\dfrac{\rho ck}{4}\left\|\dfrac{\tilde{\Theta}_{\ell,k}^{n+1}-\tilde{\Theta}_{\ell,k}^{n}}{k}\right\|_{L^2(\Omega\times (0,L))}^2\\
			&\quad-\int_{\Omega}q_{\textup{geo}}^\perp (\tilde{\Theta}_{\ell,k}^{n+1}-\tilde{\Theta}_{\ell,k}^{n})\dd x.
		\end{aligned}
	\end{equation}
	
	Let us now estimate the penalty terms in~\eqref{est-theta-2}. To begin with, we observe that an application of Young's inequality~\cite{Young1912} gives:
	\begin{equation}
		\label{penalty-theta-1}
		\begin{aligned}
			&-\int_{\Omega\times (0,L)}\{\tilde{\Theta}_{\ell,k}^{n+1}\}^{-}(\tilde{\Theta}_{\ell,k}^{n+1}-\tilde{\Theta}_{\ell,k}^{n})\dd x\dd z\\
			&\ge\int_{\Omega\times (0,L)}\left|\{\tilde{\Theta}_{\ell,k}^{n+1}\}^{-}\right|^2\dd x\dd z-\int_{\Omega\times (0,L)}\{\tilde{\Theta}_{\ell,k}^{n+1}\}^{-}\{\tilde{\Theta}_{\ell,k}^{n}\}^{-}\dd x\dd z\\
			&\ge\dfrac{1}{2}\|\{\tilde{\Theta}_{\ell,k}^{n+1}\}^{-}\|_{L^2(\Omega\times (0,L))}^2-\dfrac{1}{2}\|\{\tilde{\Theta}_{\ell,k}^{n}\}^{-}\|_{L^2(\Omega\times (0,L))}^2.
		\end{aligned}
	\end{equation}
	
	Summing~\eqref{penalty-theta-1} over $n=0,\dots,s-1$, for all $1\le s\le N$ gives
	\begin{equation}
		\label{penalty-theta-2}
		\begin{aligned}
			&-\sum_{n=0}^{s-1}\int_{\Omega\times (0,L)}\{\tilde{\Theta}_{\ell,k}^{n+1}\}^{-}(\tilde{\Theta}_{\ell,k}^{n+1}-\tilde{\Theta}_{\ell,k}^{n})\dd x\dd z\\
			&\ge\dfrac{1}{2}\|\{\tilde{\Theta}_{\ell,k}^{s}\}^{-}\|_{L^2(\Omega\times (0,L))}^2-\dfrac{1}{2}\|\{\tilde{\Theta}_0\}^{-}\|_{L^2(\Omega\times (0,L))}^2=\dfrac{1}{2}\|\{\tilde{\Theta}_{\ell,k}^{s}\}^{-}\|_{L^2(\Omega\times (0,L))}^2\ge 0,
		\end{aligned}
	\end{equation}
	being $\tilde{\Theta}_0\in K_{\textup{temp}}$. For the second, and last, penalty term we observe that another application of the Young inequality~\cite{Young1912} gives:
	\begin{equation}
		\label{penalty-theta-3}
		\begin{aligned}
			&\int_{\Omega\times (0,L)}\{\tilde{\Theta}_{\ell,k}^{n+1}-\Theta_{\textup{m}}\}^{+}(\tilde{\Theta}_{\ell,k}^{n+1}-\tilde{\Theta}_{\ell,k}^{n})\dd x\dd z\\
			&=\int_{\Omega\times (0,L)}\{\tilde{\Theta}_{\ell,k}^{n+1}-\Theta_{\textup{m}}\}^{+}\left((\tilde{\Theta}_{\ell,k}^{n+1}-\Theta_{\textup{m}})-(\tilde{\Theta}_{\ell,k}^{n}-\Theta_{\textup{m}})\right)\dd x\dd z\\
			&\ge\|\{\tilde{\Theta}_{\ell,k}^{n+1}-\Theta_{\textup{m}}\}^{+}\|_{L^2(\Omega\times (0,L))}^2-\int_{\Omega\times (0,L)}\{\tilde{\Theta}_{\ell,k}^{n+1}-\Theta_{\textup{m}}\}^{+}\{\tilde{\Theta}_{\ell,k}^{n}-\Theta_{\textup{m}}\}^{+}\dd x\dd z\\
			&\ge\dfrac{1}{2}\|\{\tilde{\Theta}_{\ell,k}^{n+1}-\Theta_{\textup{m}}\}^{+}\|_{L^2(\Omega\times (0,L))}^2-\dfrac{1}{2}\|\{\tilde{\Theta}_{\ell,k}^{n}-\Theta_{\textup{m}}\}^{+}\|_{L^2(\Omega\times (0,L))}^2.
		\end{aligned}
	\end{equation}
	
	Summing~\eqref{penalty-theta-3} over $n=0,\dots, s-1$, for all $1\le s\le N$, and recalling that $\tilde{\Theta}_0\in K_{\textup{temp}}$ gives:
	\begin{equation}
		\label{penalty-theta-4}
		\begin{aligned}
			&\sum_{n=0}^{s-1}\int_{\Omega\times (0,L)}\{\tilde{\Theta}_{\ell,k}^{n+1}-\Theta_{\textup{m}}\}^{+}(\tilde{\Theta}_{\ell,k}^{n+1}-\tilde{\Theta}_{\ell,k}^{n})\dd x\dd z\\
			&\ge\dfrac{1}{2}\|\{\tilde{\Theta}_{\ell,k}^{s}-\Theta_{\textup{m}}\}^{+}\|_{L^2(\Omega\times (0,L))}^2-\dfrac{1}{2}\|\{\tilde{\Theta}_0-\Theta_{\textup{m}}\}^{+}\|_{L^2(\Omega\times (0,L))}^2\\
			&=\dfrac{1}{2}\|\{\tilde{\Theta}_{\ell,k}^{s}-\Theta_{\textup{m}}\}^{+}\|_{L^2(\Omega\times (0,L))}^2\ge0.
		\end{aligned}
	\end{equation}
	
	Summing~\eqref{est-theta-2} over $n=0,\dots,s-1$, for all $1\le s\le N$, and applying the estimates~\eqref{penalty-theta-2} and~\eqref{penalty-theta-4} gives:
	\begin{equation*}
		\begin{aligned}
			&\dfrac{\rho c}{2}\sum_{n=0}^{s-1}k\left\|\dfrac{\tilde{\Theta}_{\ell,k}^{n+1}-\tilde{\Theta}_{\ell,k}^{n}}{k}\right\|_{L^2(\Omega\times(0,L))}^2+\dfrac{\kappa}{2}\|\nabla\tilde{\Theta}_{\ell,k}^{s}\|_{\bm{L}^2(\Omega\times (0,L))}^2-\dfrac{\kappa}{2}\|\nabla\tilde{\Theta}_0\|_{\bm{L}^2(\Omega\times (0,L))}^2\\
			&\le \rho c \|(\bm{U}(0),v_z(0))\|_{\bm{\mathcal{C}}^0(\overline{\Omega}\times[0,L])}^2\sum_{n=0}^{s-1}k\|\nabla\tilde{\Theta}_{\ell,k}^{n+1}\|_{\bm{L}^2(\Omega\times (0,L))}^2+\dfrac{4(\rho g C_0)^{2p} A_0^2 L}{\rho c}\sum_{n=0}^{s-1}k\|u_{\ell,k}^{n}\|_{L^{p-1}(\Omega)}^{p-1}\\
			&\quad-\int_{\Omega}q_{\textup{geo}}^\perp (\tilde{\Theta}_{\ell,k}^{s}-\tilde{\Theta}_0)\dd x\\
			&\le \rho c \|(\bm{U}(0),v_z(0))\|_{\bm{\mathcal{C}}^0(\overline{\Omega}\times[0,L])}^2\sum_{n=0}^{s-1}k\|\nabla\tilde{\Theta}_{\ell,k}^{n+1}\|_{\bm{L}^2(\Omega\times (0,L))}^2\\
			&\quad+4(\rho g C_0)^{2p}A_0^2 L \sum_{n=0}^{s-1}k \|u_{\ell,k}^{n}\|_{L^{p-1}(\Omega)}^{p-1}+\dfrac{2\hat{C}_{\textup{tr}}^2(1+2L^2)}{\kappa}\|q_{\textup{geo}}^\perp\|_{L^2(\Omega)}^2\\
			&\quad+\dfrac{\kappa}{4}\|\nabla\tilde{\Theta}_{\ell,k}^s\|_{\bm{L}^2(\Omega\times (0,L))}^2+\dfrac{\kappa}{4}\|\nabla\tilde{\Theta}_0\|_{\bm{L}^2(\Omega\times (0,L))}^2\\
			&\le \rho c \|(\bm{U}(0),v_z(0))\|_{\bm{\mathcal{C}}^0(\overline{\Omega}\times[0,L])}^2\sum_{n=0}^{s-1}k\|\nabla\tilde{\Theta}_{\ell,k}^{n+1}\|_{\bm{L}^2(\Omega\times (0,L))}^2\\
			&\quad+\dfrac{4(\rho g C_0)^{2p}A_0^2 L |\Omega|^{1/p}}{p'} \sum_{n=0}^{s-1}k \|u_{\ell,k}^{n}\|_{W^{1,p}(\Omega)}^p+\dfrac{4(\rho g C_0)^{2p}A_0^2 L |\Omega|^{1/p} T}{p}\\
			&\quad+\dfrac{2\hat{C}_{\textup{tr}}^2(1+2L^2)}{\kappa}\|q_{\textup{geo}}^\perp\|_{L^2(\Omega)}^2+\dfrac{\kappa}{4}\|\nabla\tilde{\Theta}_{\ell,k}^s\|_{\bm{L}^2(\Omega\times (0,L))}^2+\dfrac{\kappa}{4}\|\nabla\tilde{\Theta}_0\|_{\bm{L}^2(\Omega\times (0,L))}^2.
		\end{aligned}
	\end{equation*}
	
	A rearrangement in the chain of estimates above gives:
	\begin{equation}
		\label{est-theta-3}
		\begin{aligned}
			&\dfrac{\rho c}{2}\sum_{n=0}^{s-1}k\left\|\dfrac{\tilde{\Theta}_{\ell,k}^{n+1}-\tilde{\Theta}_{\ell,k}^{n}}{k}\right\|_{L^2(\Omega\times(0,L))}^2+\dfrac{\kappa}{4}\|\nabla\tilde{\Theta}_{\ell,k}^{s}\|_{\bm{L}^2(\Omega\times (0,L))}^2\\
			&\le \rho c \|(\bm{U}(0),v_z(0))\|_{\bm{\mathcal{C}}^0(\overline{\Omega}\times[0,L])}^2\sum_{n=0}^{s-1}k\|\nabla\tilde{\Theta}_{\ell,k}^{n+1}\|_{\bm{L}^2(\Omega\times (0,L))}^2\\
			&\quad+\dfrac{4(\rho g C_0)^{2p}A_0^2 L |\Omega|^{1/p}}{p'} \sum_{n=0}^{s-1}k \|u_{\ell,k}^{n}\|_{W^{1,p}(\Omega)}^p+\dfrac{4(\rho g C_0)^{2p}A_0^2 L |\Omega|^{1/p} T}{p}\\
			&\quad+\dfrac{2\hat{C}_{\textup{tr}}^2(1+2L^2)}{\kappa}\|q_{\textup{geo}}^\perp\|_{L^2(\Omega)}^2+\dfrac{3\kappa}{4}\|\nabla\tilde{\Theta}_0\|_{\bm{L}^2(\Omega\times (0,L))}^2.
		\end{aligned}
	\end{equation}
	
	Since, thanks to~\eqref{est:2} and~\eqref{est:7}, the right-hand side of~\eqref{est-theta-3} is bounded independently of $N$, we infer that the estimates~\eqref{est:8} and~\eqref{est:9} hold upon applying Lemma~\ref{PFV}.
	
	$(v)$\emph{Proof of~\eqref{est:3} and~\eqref{est:4}}. Multiply~\eqref{penalty:seq} by $(u_{\ell,k}^{n+1}-u_{\ell,k}^{n})$ in the sense of the duality between $W^{-1,p'}(\Omega)$ and $W_0^{1,p}(\Omega)$ and divide by $\mu(\tilde{\Theta}_{\ell,k}^{n+1})>0$, obtaining:
	\begin{equation}
		\label{prelim:1}
		\begin{aligned}
			&\dfrac{1}{k\mu(\tilde{\Theta}_{\ell,k}^{n+1})}\int_{\Omega}\left(|u_{\ell,k}^{n+1}|^{\alpha-2} u_{\ell,k}^{n+1} -|u_{\ell,k}^{n}|^{\alpha-2} u_{\ell,k}^{n}\right) (u_{\ell,k}^{n+1} - u_{\ell,k}^{n})\dd x\\
			&\quad+\int_{\Omega}|\nabla u_{\ell,k}^{n+1}|^{p-2} \nabla u_{\ell,k}^{n+1} \cdot \nabla (u_{\ell,k}^{n+1}-u_{\ell,k}^{n})\dd x\\
			&\quad-\dfrac{1}{\ell} \int_{\Omega} \{u_{\ell,k}^{n+1}\}^{-}(u_{\ell,k}^{n+1}-u_{\ell,k}^{n})\dd x =\dfrac{1}{\mu(\tilde{\Theta}_{\ell,k}^{n+1})}\int_{\Omega}\left(\dfrac{1}{k}\int_{nk}^{(n+1)k}\tilde{a}(t)\dd t\right) (u_{\ell,k}^{n+1}-u_{\ell,k}^{n}) \dd x.
		\end{aligned}
	\end{equation}
	
	For the sake of brevity, for each $0\le n\le N-1$ define:
	\begin{equation*}
		\label{atilde}
		\tilde{a}_k^{n} := \dfrac{1}{k} \int_{nk}^{(n+1)k} \tilde{a}(t) \dd t.
	\end{equation*}
	
	For each $1 \le s \le N$, the following identity holds:
	\begin{equation}
		\label{prelim:2}
		\begin{aligned}
			&\sum_{n=0}^{s-1} \int_{\Omega}\dfrac{\tilde{a}_k^{n}}{\mu(\tilde{\Theta}_{\ell,k}^{n+1})} (u_{\ell,k}^{n+1}-u_{\ell,k}^{n}) \dd x\\
			&=-\sum_{n=0}^{s-2}\int_{\Omega}\left(\dfrac{\tilde{a}_k^{n+1}}{\mu(\tilde{\Theta}_{\ell,k}^{n+2})}-\dfrac{\tilde{a}_k^{n}}{\mu(\tilde{\Theta}_{\ell,k}^{n+1})}\right)u_{\ell,k}^{n+1} \dd x+\int_{\Omega}\dfrac{\tilde{a}_k^{s-1}}{\mu(\tilde{\Theta}_{\ell,k}^{s})} u_{\ell,k}^{s} \dd x -\int_{\Omega}\dfrac{\tilde{a}_k^0}{\mu(\tilde{\Theta}_{\ell,k}^{1})} u_0 \dd x\\
			&=-\sum_{n=0}^{s-2}\dfrac{1}{\mu(\tilde{\Theta}_{\ell,k}^{n+2})}\int_{\Omega}\left(\tilde{a}_k^{n+1}-\tilde{a}_k^{n}\right)u_{\ell,k}^{n+1} \dd x+\int_{\Omega}\dfrac{\tilde{a}_k^{s-1}}{\mu(\tilde{\Theta}_{\ell,k}^{s})} u_{\ell,k}^{s} \dd x -\int_{\Omega}\dfrac{\tilde{a}_k^0}{\mu(\tilde{\Theta}_{\ell,k}^{1})} u_0 \dd x\\
			&\quad-\sum_{n=0}^{s-2}\int_{\Omega}\left(\dfrac{\tilde{a}_k^{n}}{\mu(\tilde{\Theta}_{\ell,k}^{n+2})}-\dfrac{\tilde{a}_k^{n}}{\mu(\tilde{\Theta}_{\ell,k}^{n+1})}\right)u_{\ell,k}^{n+1} \dd x.
		\end{aligned}
	\end{equation}
	
	For what concerns the second term on the right-hand side of~\eqref{prelim:2}, the following estimate clearly holds thanks to H\"{o}lder's inequality, assumption $(H\ref{H3})$ and item~$(i)$ established beforehand:
	\begin{equation}
		\label{prelim:2-1}
		\left|\int_{\Omega}\dfrac{\tilde{a}_k^{s-1}}{\mu(\tilde{\Theta}_{\ell,k}^{s})} u_{\ell,k}^{s} \dd x\right|\le\dfrac{\|\tilde{a}\|_{\mathcal{C}^0([0,T];L^{\alpha'}(\Omega))}}{\mu_1}\left(\max_{0\le n\le N}\|u_{\ell,k}^n\|_{L^\alpha(\Omega)}\right).
	\end{equation}
	
	For what concerns the third term on the right-hand side of~\eqref{prelim:2}, the following estimate also holds thanks to H\"{o}lder's inequality, assumption $(H\ref{H3})$ and~\eqref{est:1} established beforehand:
	\begin{equation}
		\label{prelim:2-2}
		\left|\int_{\Omega}\dfrac{\tilde{a}_k^{0}}{\mu(\tilde{\Theta}_{\ell,k}^{1})} u_0 \dd x\right|\le\dfrac{\|\tilde{a}\|_{\mathcal{C}^0([0,T];L^{\alpha'}(\Omega))}}{\mu_1}\left(\max_{0\le n\le N}\|u_{\ell,k}^n\|_{L^\alpha(\Omega)}\right).
	\end{equation}
	
	For what concerns the first term on the right-hand side of~\eqref{prelim:2}, observe that for each $n=0,\dots, s-2$ the following identity holds in $\mathcal{C}^0(\overline{\Omega})$:
	\begin{equation*}
		\begin{aligned}
			&\tilde{a}_k^{n+1}-\tilde{a}_k^{n}=\dfrac{1}{k}\int_{nk}^{(n+1)k}(\tilde{a}(t+k)-\tilde{a}(t))\dd t=\dfrac{1}{k}\int_{nk}^{(n+1)k}\int_{t}^{t+k}\dfrac{\dd \tilde{a}}{\dd s}(s)\dd s\dd t\\
			&=\dfrac{1}{k}\int_{nk}^{(n+1)k}\int_{nk}^{s}\dfrac{\dd\tilde{a}}{\dd s}(s)\dd t\dd s+\dfrac{1}{k}\int_{(n+1)k}^{(n+2)k}\int_{s-k}^{(n+1)k}\dfrac{\dd\tilde{a}}{\dd s}(s)\dd t\dd s\\
			&=\dfrac{1}{k}\int_{nk}^{(n+1)k}\dfrac{\dd\tilde{a}}{\dd s}(s) (s-nk)\dd s+\dfrac{1}{k}\int_{(n+1)k}^{(n+2)k}\dfrac{\dd\tilde{a}}{\dd s}(s)((n+2)k-s)\dd s.
		\end{aligned}
	\end{equation*}
	
	Taking the norm $\|\cdot\|_{W^{-1,p'}(\Omega)}$ in the identity above and exploiting $(H\ref{H3})$ gives:
	\begin{equation*}
		\|\tilde{a}_k^{n+1}-\tilde{a}_k^{n}\|_{W^{-1,p'}(\Omega)}\le 2k\left\|\dfrac{\dd\tilde{a}}{\dd t}\right\|_{L^\infty(0,T;W^{-1,p'}(\Omega))}.
	\end{equation*}
	
	Another application of $(H\ref{H3})$ and of the estimate above allows us to obtain the following estimate for the first term on the right-hand side of~\eqref{prelim:2}:
	\begin{equation}
		\label{prelim:2-2-bis}
		\left|\sum_{n=0}^{s-2}\dfrac{1}{\mu(\tilde{\Theta}_{\ell,k}^{n+2})}\int_{\Omega}\left(\tilde{a}_k^{n+1}-\tilde{a}_k^{n}\right)u_{\ell,k}^{n+1} \dd x\right|
		\le \dfrac{2k}{\mu_1}\left\|\dfrac{\dd\tilde{a}}{\dd t}\right\|_{L^\infty(0,T;W^{-1,p'}(\Omega))}\sum_{n=0}^{s-2}\|u_{\ell,k}^{n+1}\|_{W^{1,p}(\Omega)}.
	\end{equation}
	
	An application of Young's inequality~\cite{Young1912} to the right-hand side of the estimate above gives
	\begin{equation}
		\label{prelim:2-3}
		\begin{aligned}
			&k\sum_{n=0}^{s-2}\|u_{\ell,k}^{n+1}\|_{W^{1,p}(\Omega)}\le\sum_{n=0}^{s-2}\left\{\dfrac{k\|u_{\ell,k}^{n+1}\|_{W^{1,p}(\Omega)}^p}{p}+\dfrac{k}{p'}\right\}\\
			&\le\dfrac{k}{p}\sum_{n=0}^{s-2}\|u_{\ell,k}^{n+1}\|_{W^{1,p}(\Omega)}^p+\dfrac{T}{p'},
		\end{aligned}
	\end{equation}
	and, thanks to~\eqref{est:2} and~\eqref{prelim:2-3}, we obtain that
	\begin{equation}
		\label{prelim:2-4}
		\left\{k\sum_{n=0}^{N-2}\|u_{\ell,k}^{n+1}\|_{W^{1,p}(\Omega)}\right\}_{N=1}^\infty \textup{ is bounded independently of }N \textup{ and }\ell,
	\end{equation}
	thus showing that~\eqref{prelim:2-2-bis} is bounded independently of $N$ and $\ell$.
	
	Finally, we estimate the last term on the right-hand side of~\eqref{prelim:2}. To begin with, let us observe that for each $n=0,\dots, N-1$, an application of Lemma~\ref{lem:1}$(a)$, the Mean Value Theorem (cf., e.g., Theorem~9.3-1 in~\cite{Ciarlet2025}) and H\"{o}lder's inequality give:
	\begin{equation}
		\label{Lip:est-1}
		\begin{aligned}
			&|\mu(\tilde{\Theta}_{\ell,k}^{n+1})-\mu(\tilde{\Theta}_{\ell,k}^{n})|\\
			&\le 2A_0\left(\rho g\dfrac{p-1}{2p}\right)^{p-1}\int_{0}^{1}\bigg|\exp\left(-\dfrac{\beta\sqrt{L}}{|\Omega|^{1/2}}\int_{\Omega}\{\tilde{\Theta}_{\ell,k}^{n+1}(x,s'L)-\Theta_{\textup{m}}\}^{-}\dd x\right)\\
			&\quad-\exp\left(-\dfrac{\beta\sqrt{L}}{|\Omega|^{1/2}}\int_{\Omega}\{\tilde{\Theta}_{\ell,k}^{n}(x,s'L)-\Theta_{\textup{m}}\}^{-}\dd x\right)\bigg|\dd s'\\
			&\le 2A_0\left(\rho g\dfrac{p-1}{2p}\right)^{p-1}\dfrac{\beta}{\sqrt{L|\Omega|}}\int_{0}^{L}\int_{\Omega}\left|\{\tilde{\Theta}_{\ell,k}^{n+1}-\Theta_{\textup{m}}\}^{-}-\{\tilde{\Theta}_{\ell,k}^{n}-\Theta_{\textup{m}}\}^{-}\right|\dd x\dd z\\
			&\le 2A_0\left(\rho g\dfrac{p-1}{2p}\right)^{p-1}\dfrac{\beta}{\sqrt{L|\Omega|}}\int_{0}^{L}\int_{\Omega}|\tilde{\Theta}_{\ell,k}^{n+1}-\tilde{\Theta}_{\ell,k}^{n}|\dd x\dd z\\
			&\le 2A_0\beta\left(\rho g\dfrac{p-1}{2p}\right)^{p-1}\|\tilde{\Theta}_{\ell,k}^{n+1}-\tilde{\Theta}_{\ell,k}^{n}\|_{L^2(\Omega\times(0,L))}.
		\end{aligned}
	\end{equation}
	
	An application of Young's inequality~\cite{Young1912} and~\eqref{Lip:est-1} give:
	\begin{equation}
		\label{prelim:2-5}
		\begin{aligned}
			&\left|\sum_{n=0}^{s-2}\int_{\Omega}\left(\dfrac{\tilde{a}_k^{n}}{\mu(\tilde{\Theta}_{\ell,k}^{n+2})}-\dfrac{\tilde{a}_k^{n}}{\mu(\tilde{\Theta}_{\ell,k}^{n+1})}\right)u_{\ell,k}^{n+1} \dd x\right|\\
			&\le\sum_{n=0}^{s-2}\|\tilde{a}_k^n\|_{W^{-1,p'}(\Omega)}\|u_{\ell,k}^{n+1}\|_{W^{1,p}(\Omega)}\left|\dfrac{\mu(\tilde{\Theta}_{\ell,k}^{n+2})-\mu(\tilde{\Theta}_{\ell,k}^{n+1})}{\mu(\tilde{\Theta}_{\ell,k}^{n+2})\mu(\tilde{\Theta}_{\ell,k}^{n+1})}\right|\\
			&\le\dfrac{2A_0\beta}{\mu_1^2}\left(\rho g\dfrac{p-1}{2p}\right)^{p-1}\sum_{n=0}^{s-2}\|\tilde{a}_k^n\|_{W^{-1,p'}(\Omega)}\|u_{\ell,k}^{n+1}\|_{W^{1,p}(\Omega)}\|\tilde{\Theta}_{\ell,k}^{n+1}-\tilde{\Theta}_{\ell,k}^{n}\|_{L^2(\Omega\times(0,L))}\\
			&\le \dfrac{4\beta^2 A_0^2}{\rho c\mu_1^3} \left(\rho g\dfrac{p-1}{2p}\right)^{2p-2}\|\tilde{a}\|_{L^\infty(0,T;W^{-1,p'}(\Omega))}^2 \sum_{n=0}^{s-2}k\|u_{\ell,k}^{n+1}\|_{W^{1,p}(\Omega)}^2\\
			&\quad+\dfrac{\rho c}{4\mu_1}\sum_{n=0}^{s-2}k\left\|\dfrac{\tilde{\Theta}_{\ell,k}^{n+1}-\tilde{\Theta}_{\ell,k}^{n}}{k}\right\|_{L^2(\Omega\times (0,L))}^2\\
			&\le \dfrac{8 \beta^2 A_0^2}{p \rho c\mu_1^3} \left(\rho g\dfrac{p-1}{2p}\right)^{2p-2}\|\tilde{a}\|_{L^\infty(0,T;W^{-1,p'}(\Omega))}^2\sum_{n=0}^{s-2}k\|u_{\ell,k}^{n+1}\|_{W^{1,p}(\Omega)}^p\\
			&\quad+\dfrac{4(p-2)\beta^2 A_0^2T}{p \rho c\mu_1^3} \left(\rho g\dfrac{p-1}{2p}\right)^{2p-2}\|\tilde{a}\|_{L^\infty(0,T;W^{-1,p'}(\Omega))}^2\\
			&\quad+\dfrac{\rho c}{4\mu_1}\sum_{n=0}^{s-2}k\left\|\dfrac{\tilde{\Theta}_{\ell,k}^{n+1}-\tilde{\Theta}_{\ell,k}^{n}}{k}\right\|_{L^2(\Omega\times (0,L))}^2.
		\end{aligned}
	\end{equation}
	
	For any $0\le n \le N-1$, we have that an application of the Young's inequality~\cite{Young1912} gives:
	\begin{equation*}
		\begin{aligned}
			&\dfrac{1}{\ell}\int_{\Omega} \{u_{\ell,\kappa}^{n+1}\}^{-}(u_{\ell,\kappa}^{n+1}-u_{\ell,\kappa}^{n}) \dd x
			=\dfrac{1}{\ell}\int_{\Omega} \left(-\left|\{u_{\ell,\kappa}^{n+1}\}^{-}\right|^2-\{u_{\ell,\kappa}^{n+1}\}^{-}\left(\{u_{\ell,\kappa}^{n}\}^{+}-\{u_{\ell,\kappa}^{n}\}^{-}\right)\right) \dd x\\
			&\le -\dfrac{1}{\ell}\int_{\Omega} \left|\{u_{\ell,\kappa}^{n+1}\}^{-}\right|^2 \dd x
			+\dfrac{1}{\ell}\int_{\Omega} \{u_{\ell,\kappa}^{n+1}\}^{-} \{u_{\ell,\kappa}^{n}\}^{-} \dd x\\
			& \le -\dfrac{1}{\ell}\int_{\Omega} \left|\{u_{\ell,\kappa}^{n+1}\}^{-}\right|^2 \dd x
			+\dfrac{1}{2\ell}\int_{\Omega} \left|\{u_{\ell,\kappa}^{n+1}\}^{-}\right|^2+\left|\{u_{\ell,\kappa}^{n}\}^{-}\right|^2 \dd x\\
			&= -\dfrac{1}{2\ell}\int_{\Omega} \left|\{u_{\ell,\kappa}^{n+1}\}^{-}\right|^2 \dd x
			+\dfrac{1}{2\ell}\int_{\Omega} \left|\{u_{\ell,\kappa}^{n}\}^{-}\right|^2 \dd x.
		\end{aligned}
	\end{equation*}
	
	For any given $1 \le s \le N$, we thus have that the previous estimate gives:
	\begin{equation*}
		\begin{aligned}
			&\dfrac{1}{\ell}\sum_{n=0}^{s-1}\int_{\Omega} \{u_{\ell,\kappa}^{n+1}\}^{-}(u_{\ell,\kappa}^{n+1}-u_{\ell,\kappa}^{n}) \dd x
			\le -\dfrac{1}{2\ell} \sum_{n=0}^{s-1} \int_{\Omega} \left|\{u_{\ell,\kappa}^{n+1}\}^{-}\right|^2 - \left|\{u_{\ell,\kappa}^{n}\}^{-}\right|^2 \dd x\\
			&=-\dfrac{1}{2\ell}\int_{\Omega} \left|\{u_{\ell,\kappa}^{s}\}^{-}\right|^2 - \left|\{u_{\ell,\kappa}^{0}\}^{-}\right|^2 \dd x
			=-\dfrac{1}{2\ell}\int_{\Omega} \left|\{u_{\ell,\kappa}^{s}\}^{-}\right|^2 - \left|\{u_0\}^{-}\right|^2 \dd x\\
			&=-\dfrac{1}{2\ell}\int_{\Omega} \left|\{u_{\ell,\kappa}^{s}\}^{-}\right|^2 \dd x \le 0,
		\end{aligned}
	\end{equation*}
	where the last equality directly follows from the fact that $u_0 \in K_{\textup{surf}}$.
	To summarise, we have shown that:
	\begin{equation}
		\label{SC1}
		\dfrac{1}{\ell}\sum_{n=0}^{s-1}\int_{\Omega} \{u_{\ell,\kappa}^{n+1}\}^{-}(u_{\ell,\kappa}^{n+1}-u_{\ell,\kappa}^{n}) \dd x \le 0.
	\end{equation}
	
	Summing~\eqref{prelim:1} over $0 \le n \le s-1$, with $1\le s \le N$, applying Lemma~\ref{lem:7}, exploiting~\eqref{prelim:2-1}--\eqref{prelim:2-5},~\eqref{SC1}, Young's inequality~\cite{Young1912}, the Poincar\'e--Friedrichs inequality (cf., e.g. Theorem~8.3-3$(a)$ in~\cite{Ciarlet2025}) and~\eqref{mu0} gives:
	\begin{equation*}
		\begin{aligned}
			&\dfrac{4(\alpha-1)k}{\alpha^2 \mu_2} \sum_{n=0}^{s-1} \left\|\dfrac{|u_{\ell,k}^{n+1}|^\frac{\alpha-2}{2}u_{\ell,k}^{n+1}-|u_{\ell,k}^{n}|^\frac{\alpha-2}{2}u_{\ell,k}^{n}}{k}\right\|_{L^2(\Omega)}^2
			+\dfrac{1}{p+(\textup{diam }\Omega)^p}\left(\|u_{\ell,k}^{s}\|_{W^{1,p}(\Omega)}^p-\|u_0\|_{W^{1,p}(\Omega)}^p\right)\\
			&\le \dfrac{4(\alpha-1)k}{\alpha^2 \mu_2} \sum_{n=0}^{s-1} \left\|\dfrac{|u_{\ell,k}^{n+1}|^\frac{\alpha-2}{2}u_{\ell,k}^{n+1}-|u_{\ell,k}^{n}|^\frac{\alpha-2}{2}u_{\ell,k}^{n}}{k}\right\|_{L^2(\Omega)}^2+\dfrac{1}{p}\sum_{n=0}^{s-1} \left\{\|\nabla u_{\ell,k}^{n+1}\|_{L^p(\Omega)}^p - \|\nabla u_{\ell,k}^{n}\|_{L^p(\Omega)}^p\right\}\\
			&\le \dfrac{2\|\tilde{a}\|_{\mathcal{C}^0([0,T];L^{\alpha'}(\Omega))}}{\mu_1}\left(\max_{0\le n\le N}\|u_{\ell,k}^{n}\|_{L^\alpha(\Omega)}\right)\\
			&\quad+\dfrac{2}{\mu_1 p}\left\|\dfrac{\dd\tilde{a}}{\dd t}\right\|_{L^\infty(0,T;W^{-1,p'}(\Omega))}\left(\dfrac{pT}{p'}+\sum_{n=0}^{s-2}k\|u_{\ell,k}^{n+1}\|_{W^{1,p}(\Omega)}^p\right)\\
			&\quad+\dfrac{8\beta^2A_0^2}{p\rho c\mu_1^3}\left(\rho g \dfrac{p-1}{2p}\right)^{2p-2}\|\tilde{a}\|_{L^\infty(0,T;W^{-1,p'}(\Omega))}^2\sum_{n=0}^{s-2}k\|u_{\ell,k}^{n+1}\|_{W^{1,p}(\Omega)}^p\\
			&\quad+\dfrac{4(p-2)\beta^2A_0^2 T}{p\rho c\mu_1^3}\left(\rho g \dfrac{p-1}{2p}\right)^{2p-2}\|\tilde{a}\|_{L^\infty(0,T;W^{-1,p'}(\Omega))}^2 +\dfrac{\rho c}{4\mu_1}\sum_{n=0}^{s-2}k\left\|\dfrac{\tilde{\Theta}_{\ell,k}^{n+1}-\tilde{\Theta}_{\ell,k}^{n}}{k}\right\|_{L^2(\Omega\times (0,L))}^2.
		\end{aligned}
	\end{equation*}
	
	An application of~\eqref{est:1}, \eqref{est:2} and~\eqref{est:9} in turn implies the validity of estimates~\eqref{est:3} and ~\eqref{est:4}.
	
	$(vi)$\emph{Proof of~\eqref{est:5}}. In order to establish the estimate~\eqref{est:5}, we exploit the boundedness of the $p$-Laplace operator (cf., e.g., Section~12.6 in~\cite{{Ciarlet2025}}), the validity of the embedding $L^\alpha(\Omega)\hookrightarrow W^{-1,p'}(\Omega)$, and~\eqref{est:4} so as to be in a position to evaluate
	\begin{equation}
		\label{prelim:6}
		\begin{aligned}
			&\left\|\dfrac{|u_{\ell,k}^{n+1}|^{\alpha-2}u_{\ell,k}^{n+1}-|u_{\ell,k}^{n}|^{\alpha-2}u_{\ell,k}^{n}}{k}\right\|_{W^{-1,p'}(\Omega)}\\
			&\le \mu_2\left\|\nabla \cdot \left(|\nabla u_{\ell,k}^{n+1}|^{p-2} \nabla u_{\ell,k}^{n+1}\right)\right\|_{W^{-1,p'}(\Omega)}
			+\dfrac{\mu_2}{\ell}\|\{u_{\ell,k}^{n+1}\}^{-}\|_{W^{-1,p'}(\Omega)}+\|\tilde{a}\|_{W^{1,p}(0,T;\mathcal{C}^0(\overline{\Omega}))}\\
			&\le C\left(1+\dfrac{1}{\ell}\right), \quad\textup{ for some } C>0 \textup{ independent of } N,
		\end{aligned}
	\end{equation}
	for all $0 \le n \le N-1$, thus establishing the estimate~\eqref{est:5}.
	
	$(vii)$\emph{Proof of~\eqref{est:5:1}}. By~\eqref{est:1}, for each $0\le n \le N$, we have that:
	\begin{equation*}
		\begin{aligned}
			&\|u_{\ell,k}^{n}\|_{L^\alpha(\Omega)} =\left(\int_{\Omega} |u_{\ell,k}^{n}|^\alpha \dd x\right)^{1/\alpha}
			=\left(\int_{\Omega} \left||u_{\ell,k}^{n}|^{\alpha-2}u_{\ell,k}^{n}\right|^\frac{\alpha}{\alpha-1} \dd x\right)^{1/\alpha}\\
			&=\left(\int_{\Omega} \left||u_{\ell,k}^{n}|^{\alpha-2}u_{\ell,k}^{n}\right|^{\alpha'} \dd x\right)^\frac{1}{\alpha' (\alpha-1)}
			=\left\||u_{\ell,k}^{n}|^{\alpha-2} u_{\ell,k}^{n}\right\|_{L^{\alpha'}(\Omega)}^{\alpha'/\alpha}.
		\end{aligned}
	\end{equation*}
	
	Since the first term is bounded independently of $N$ and $\ell$, we immediately infer that the sequence
	\begin{equation}
		\label{chi-kappa}
		\left\{\max_{0\le n\le N} \left\||u_{\ell,k}^{n}|^{\alpha-2} u_{\ell,k}^{n}\right\|_{L^{\alpha'}(\Omega)}\right\}_{N=1}^\infty,
	\end{equation}
	is bounded independently of $N$ and $\ell$, thus establishing~\eqref{est:5:1}. The proof is complete.
\end{proof}

The proof of Theorem~\ref{thm:2} can be summarised by means of the following diagram.
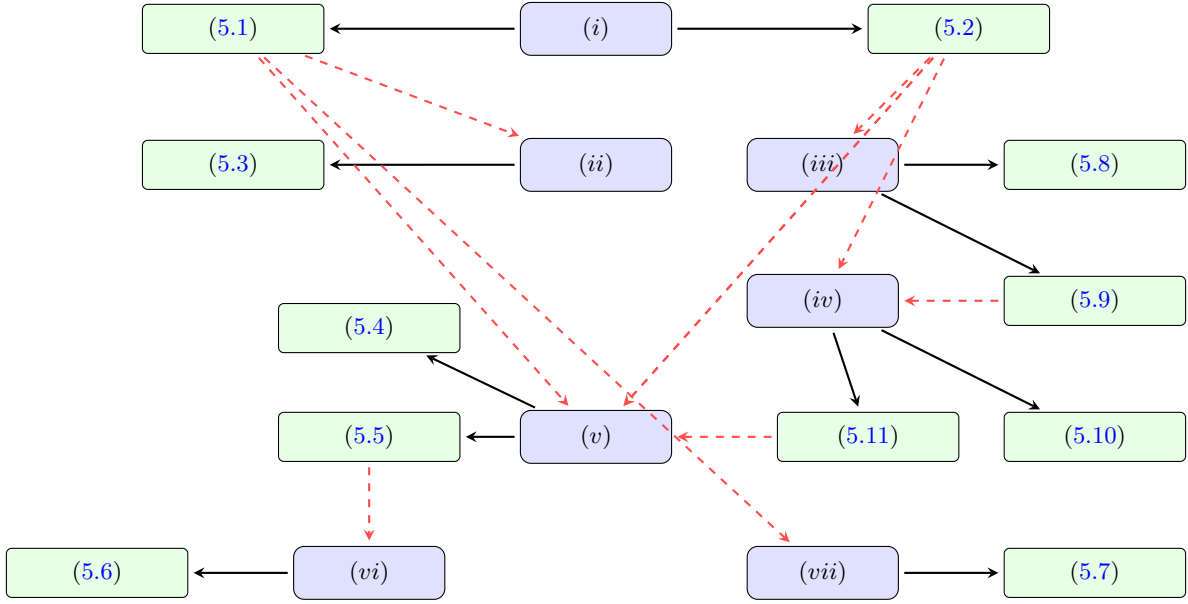
\begin{figure}[ht]
	\centering
	\begin{tikzpicture}[
		scale=1.2,
		mystep/.style={rectangle, draw=black, fill=blue!12, rounded corners=4pt,
			minimum width=2cm, minimum height=0.7cm, align=center,
			font=\footnotesize\sffamily},
		est/.style={rectangle, draw=black, fill=green!10, rounded corners=2pt,
			minimum width=2.4cm, minimum height=0.65cm, align=center,
			font=\footnotesize\ttfamily},
		arrow/.style={->, >=stealth, thick, shorten >=2pt, shorten <=2pt},
		]
		
		\node[mystep] (i)   at (0,   6)    {$(i)$};
		\node[mystep] (ii)  at (0, 4.5)  {$(ii)$};
		\node[mystep] (iii) at (2.5,  4.5)  {$(iii)$};
		\node[mystep] (iv)  at (2.5,  3)    {$(iv)$};
		\node[mystep] (v)   at (0,   1.5)   {$(v)$};
		\node[mystep] (vi)  at (-2.5, 0)    {$(vi)$};
		\node[mystep] (vii) at (2.5,  0)    {$(vii)$};
		
		\node[est] (e1)   at (-4,   6)    {\eqref{est:1}};
		\node[est] (e2)   at (4,    6)    {\eqref{est:2}};
		\node[est] (e31)  at (-4,   4.5)  {\eqref{est:3:1}};
		\node[est] (e6)   at (5.5,  4.5)  {\eqref{est:6}};
		\node[est] (e7)   at (5.5,  3)    {\eqref{est:7}};
		\node[est] (e8)   at (5.5,  1.5)  {\eqref{est:8}};
		\node[est] (e9)   at (3,    1.5)  {\eqref{est:9}};
		\node[est] (e3)  at (-2.5, 2.7)  {\eqref{est:3}};
		\node[est] (e4)  at (-2.5, 1.5)  {\eqref{est:4}};
		\node[est] (e5)   at (-5.5, 0)    {\eqref{est:5}};
		\node[est] (e51)  at (5.5,  0)    {\eqref{est:5:1}};
		
		\draw[arrow] (i)   -- (e1);
		\draw[arrow] (i)   -- (e2);
		\draw[arrow] (ii)  -- (e31);
		\draw[arrow] (iii) -- (e6);
		\draw[arrow] (iii) -- (e7);
		\draw[arrow] (iv)  -- (e8);
		\draw[arrow] (iv)  -- (e9);
		\draw[arrow] (v)   -- (e3);
		\draw[arrow] (v)   -- (e4);
		\draw[arrow] (vi)  -- (e5);
		\draw[arrow] (vii) -- (e51);
		
		\draw[arrow, dashed, color=red!70] (e1)  -- (ii);
		\draw[arrow, dashed, color=red!70] (e1)  -- (v);
		\draw[arrow, dashed, color=red!70] (e1)  -- (vii);
		\draw[arrow, dashed, color=red!70] (e2)  -- (iii);
		\draw[arrow, dashed, color=red!70] (e2)  -- (iv);
		\draw[arrow, dashed, color=red!70] (e2)  -- (v);
		\draw[arrow, dashed, color=red!70] (e2)  -- (v);
		\draw[arrow, dashed, color=red!70] (e7)  -- (iv);
		\draw[arrow, dashed, color=red!70] (e9)  -- (v);
		\draw[arrow, dashed, color=red!70] (e4) -- (vi);
		
	\end{tikzpicture}
	\caption{Dependency graph for the proof of Theorem~\ref{thm:2}. Solid arrows indicate which estimate is proved by each step; dashed arrows indicate which previously proved estimate is used as an input.}
	\label{fig:dependency}
\end{figure}

For each $\ell>0$, define the operator $B_\ell:W_0^{1,p}(\Omega)\to W^{-1,p'}(\Omega)$ by
\begin{equation}
	\label{Bkappa}
	\langle B_\ell(u),v\rangle_{W^{-1,p'}(\Omega), W_0^{1,p}(\Omega)}:=\int_{\Omega}|\nabla u|^{p-2}\nabla u\cdot\nabla v\dd x,
\end{equation}
for all $u,v\in W_0^{1,p}(\Omega)$.

The \emph{a priori} estimates~\eqref{est:1}--\eqref{est:9} can be summarised as follows:
\begin{equation}
	\label{estimates}
	\begin{aligned}
		\{\Pi_k \bm{u}_{\ell,k}\}_{N=1}^\infty &\textup{ is bounded in } L^\infty(0,T;W_0^{1,p}(\Omega)) \textup{ independently of }N \textup{ and }\ell,\\
		\{|\Pi_k \bm{u}_{\ell,k}|^\frac{\alpha-2}{2}\Pi_k \bm{u}_{\ell,k}\}_{N=1}^\infty &\textup{ is bounded in } L^\infty(0,T;L^2(\Omega))\textup{ independently of }N \textup{ and }\ell,\\
		\{D_k(|\Pi_k \bm{u}_{\ell,k}|^\frac{\alpha-2}{2}\Pi_k \bm{u}_{\ell,k})\}_{N=1}^\infty &\textup{ is bounded in } L^2(0,T;L^2(\Omega))\textup{ independently of }N \textup{ and }\ell,\\
		\{|\Pi_k \bm{u}_{\ell,k}|^{\alpha-2} \Pi_k \bm{u}_{\ell,k}\}_{N=1}^\infty &\textup{ is bounded in } L^\infty(0,T;L^{\alpha'}(\Omega))\textup{ independently of }N \textup{ and }\ell,\\
		\{D_k(|\Pi_k \bm{u}_{\ell,k}|^{\alpha-2}\Pi_k \bm{u}_{\ell,k})\}_{N=1}^\infty &\textup{ is bounded in } L^\infty(0,T;W^{-1,p'}(\Omega))\textup{ independently of }N,\\
		\{\Pi_k\tilde{\bm{\Theta}}_{\ell, k}\}_{N=1}^\infty&\textup{ is bounded in } L^\infty(0,T;V) \textup{ independently of }N \textup{ and }\ell,\\
		\{-\Delta_b(\Pi_k\tilde{\bm{\Theta}}_{\ell, k})\}_{N=1}^\infty&\textup{ is bounded in } L^\infty(0,T;V') \textup{ independently of }N \textup{ and }\ell,\\
		\{B_\ell(\Pi_k \bm{u}_{\ell,k})\}_{N=1}^\infty &\textup{ is bounded in } L^\infty(0,T;W^{-1,p'}(\Omega))\textup{ independently of }N \textup{ and }\ell,\\
		\{D_k(\Pi_k\tilde{\bm{\Theta}}_{\ell, k})\}_{N=1}^\infty &\textup{ is bounded in } L^2(0,T;L^2(\Omega\times (0,L)))\textup{ independently of }N \textup{ and }\ell.
	\end{aligned}
\end{equation}

Observe, in particular, that the third last boundedness derives from the fact that, for each $\Xi\in V$ such that $\|\Xi\|_V=1$, the following estimates hold for a.a. $t\in (0,T)$
\begin{equation*}
	\left|\langle-\Delta_b(\Pi_k\tilde{\bm{\Theta}}_{\ell, k}(t)),\Xi\rangle_{V',V}\right|\le\|\Pi_k\tilde{\bm{\Theta}}_{\ell, k}\|_{L^\infty(0,T;V)}+\hat{C}_{\textup{tr}}\|q_{\textup{geo}}^\perp\|_{L^2(\Omega)},
\end{equation*}
while the second last boundedness descends from~\eqref{prelim:6}.

\section{Other preliminary lemmas}
\label{Sec:3:bis}

Let us now show that, thanks to the discrete version of Dubinskii's lemma (Lemma~\ref{Dub:dis}) applies to the sequence $\left\{|\Lambda_k \bm{u}_{\ell,k}|^\frac{p-1}{2p}\right\}_{N=1}^\infty$. As an application of the Banach--Alaoglu--Bourbaki theorem (cf.,e.g., Theorem~3.6 of~\cite{Brez11}), up to passing to a suitable subsequence, it results $\Lambda_k\bm{u}_{\ell,k}\wsc u_\ell$ in $L^\infty(0,T;W_0^{1,p}(\Omega))$ as $N\to\infty$.
\begin{lemma}
	\label{Dub:appl-1}
	Consider the sequence $\{\Lambda_k \bm{u}_{\ell,k}\}_{N=1}^\infty$ in~\eqref{PilL}. Then, up to passing to a suitable subsequence, it results that
	\begin{equation*}
		\begin{aligned}
			\Lambda_k \bm{u}_{\ell,k}\wsc u_\ell,&\quad\textup{ in }L^\infty(0,T;W_0^{1,p}(\Omega))\text{ as }N\to\infty,\\
			|\Lambda_k \bm{u}_{\ell,k}|^\frac{\alpha-2}{2} \Lambda_k \bm{u}_{\ell,k}\to |u_\ell|^\frac{\alpha-2}{2} u_\ell,&\quad\textup{ in } L^{q_0}(0,T;L^\frac{2p}{\alpha}(\Omega)) \textup{ as } N\to\infty \textup{ for all }1\le q_0<\infty.
		\end{aligned}
	\end{equation*}
\end{lemma}
\begin{proof}
	By virtue of \eqref{estimates} and the fact that $\Lambda_k$ differs from $\Pi_k$ only by a time translation with time-step $k$, the sequence $\{\Lambda_k \bm{u}_{\ell,k}\}_{N=1}^\infty$ is bounded in $L^\infty(0,T;W_0^{1,p}(\Omega))$, and this bound is independent of $k$. Recall that the sequence $\left\{|\Lambda_k \bm{u}_{\ell,k}|^\frac{\alpha-2}{2} \Lambda_k \bm{u}_{\ell,k}\right\}_{N=1}^\infty$ is bounded in $L^\infty(0,T;L^2(\Omega))$ independently of $k$, and that the sequence $\left\{D_k\left(|\Lambda_k \bm{u}_{\ell,k}|^\frac{\alpha-2}{2} \Lambda_k \bm{u}_{\ell,k}\right)\right\}_{N=1}^\infty$ is bounded in $L^2(0,T;L^2(\Omega))$ independently of $k$. Consequently, by extracting a suitable subsequence, we obtain
	\begin{equation}
		\label{conv-Dub-1}
		\begin{aligned}
			\Lambda_k \bm{u}_{\ell,k}\wsc u_\ell,&\quad\text{ in }L^\infty(0,T;W_0^{1,p}(\Omega))\text{ as }N\to\infty,\\
			|\Lambda_k \bm{u}_{\ell,k}|^\frac{\alpha-2}{2} \Lambda_k \bm{u}_{\ell,k}\wsc v_\ell,&\quad\text{ in } L^2(0,T;L^2(\Omega))\text{ as }N\to\infty,
		\end{aligned}
	\end{equation}
	and the first claimed convergence is thus established.
	
	The rest of the proof, devoted to the identification of the limit in the second convergence of~\eqref{conv-Dub-1}, is divided into three parts, numbered $(i)$--$(iii)$.
	
	$(i)$ \emph{Construction of a semi-normed cone}. For any $u \in W_0^{1,p}(\Omega)$, we seek a real number $\beta$ such that:
	\begin{equation*}
		\left||u|^\frac{\alpha-2}{2} u\right|^\frac{\beta-2}{2} |u|^\frac{\alpha-2}{2} u =u.
	\end{equation*}
	
	We note that a sufficient condition for this is that $\beta$ satisfies
	\begin{equation*}
		\left(\dfrac{\alpha-2}{2}+1\right)\dfrac{\beta-2}{2}+\dfrac{\alpha-2}{2}=0,
	\end{equation*}
	which is equivalent to
	\begin{equation*}
		\beta= 2+2\left(\dfrac{2-\alpha}{2}\dfrac{2}{\alpha}\right)=\dfrac{4}{\alpha}.
	\end{equation*}
	
	Let $v:=|u|^\frac{\alpha-2}{2} u$ and observe that if $\beta=4/\alpha$, then $|v|^\frac{\beta-2}{2} v \in W_0^{1,p}(\Omega)$. Hence, the set
	\begin{equation*}
		\tilde{S}:=\{v;(|v|^\frac{(4/\alpha)-2}{2}v) \in W_0^{1,p}(\Omega)\}
	\end{equation*}
	is a non-empty cone. Define the function
	\begin{equation*}
		\tilde{M}(v):=\left\|\nabla(|v|^\frac{(4/\alpha)-2}{2}v)\right\|_{\bm{L}^p(\Omega)}^\frac{\alpha}{2}=\left\|\nabla(|v|^\frac{2-\alpha}{\alpha} v)\right\|_{\bm{L}^p(\Omega)}^\frac{\alpha}{2},\quad\text{ for all }v \in \tilde{S}.
	\end{equation*}
	
	We need to check that $\tilde{M}$ is a definite homogeneous gauge. Clearly, the function $\tilde{M}$ is non-negative and if $v=0$ then $\tilde{M}(v)=0$.
	If $\tilde{M}(v)=0$, an application of the Poincar\'{e}--Friedrichs inequality gives
	\begin{equation*}
		0=\tilde{M}(v)\ge c_P^{-\alpha/2}\left\||v|^{\frac{2-\alpha}{\alpha}}v\right\|_{L^p(\Omega)}^{\alpha/2},
	\end{equation*}
	and the latter in turn implies that $|v|^{\frac{2-\alpha}{\alpha}}v=0$ in $L^p(\Omega)$. The definition of $\alpha$ in turn gives that
	\begin{equation*}
		\dfrac{2-\alpha}{\alpha}=\dfrac{p+1}{3p-1}>0,
	\end{equation*}
	which implies $v=0$ as it was to be proved. Finally, let $\lambda\ge 0$ and let $v\in\tilde{S}$ and note that a straightforward computation leads to $\tilde{M}(\lambda v)=\lambda \tilde{M}(v)$.
	
	$(ii)$ \emph{A pre-compactness argument}. Define the set
	\begin{equation*}
		\tilde{\mathscr{M}}:=\{v\in \tilde{S}; \tilde{M}(v)\le 1\}.
	\end{equation*}
	
	An application of the Poincar\'e--Friedrichs inequality (cf., e.g., Theorem~8.3-3$(a)$ in~\cite{Ciarlet2025}) gives:
	\begin{equation*}
		\begin{aligned}
			1 &\ge \tilde{M}(v) \ge (1+c_P^p)^{-\frac{\alpha}{2p}} \left\||v|^\frac{2-\alpha}{\alpha} v\right\|_{W^{1,p}(\Omega)}^\frac{\alpha}{2} \ge (1+c_P^p)^{-\frac{\alpha}{2p}} \left\||v|^\frac{2-\alpha}{\alpha} v\right\|_{L^p(\Omega)}^\frac{\alpha}{2}\\
			&=(1+c_P^p)^{-\frac{\alpha}{2p}}\left(\int_{\Omega} \left||v|^\frac{2-\alpha}{\alpha} v\right|^p\dd x\right)^{\alpha/(2p)}
			=(1+c_P^p)^{-\frac{\alpha}{2p}}\left(\int_{\Omega} |v|^\frac{2p}{\alpha}\dd x\right)^{\alpha/(2p)}
			=(1+c_P^p)^{-\frac{\alpha}{2p}}\|v\|_{L^\frac{2p}{\alpha}(\Omega)}.
		\end{aligned}
	\end{equation*}
	
	Let $\{v_m\}_{m=1}^\infty$ be a sequence in $\tilde{\mathscr{M}}$. By the Rellich--Kondra\v{s}ov theorem (cf., e.g., Theorem~8.4-3 in~\cite{Ciarlet2025}), we have $W_0^{1,p}(\Omega) \hookrightarrow\hookrightarrow L^p(\Omega)$. Thus, up to passing to a subsequence, there exists $w \in L^p(\Omega)$ such that:
	\begin{equation}
		\label{conv-3}
		(|v_m|^\frac{2-\alpha}{\alpha} v_m) \to w,\quad\text{ in } L^p(\Omega)\text{ as }m\to\infty.
	\end{equation}
	
	Given that $1<\alpha<2$ and $2.8\le p \le 5$, it follows that $1 \le p' <2<p<\frac{2p}{\alpha}<2p<\infty$ and that $\{v_m\}_{m=1}^\infty$ is bounded in $L^\frac{2p}{\alpha}(\Omega)$. The reflexivity of $L^\frac{2p}{\alpha}(\Omega)$ allows us to apply the Banach--Eberlein--Smulian theorem (see, e.g., Theorem~6.17-8 of~\cite{Ciarlet2025}) and extract a subsequence, still denoted $\{v_m\}_{m=1}^\infty$, that converges weakly to an element $v$ in $L^{\frac{2p}{\alpha}}(\Omega)$ as $m\to\infty$. Consider the mapping:
	\begin{equation*}
		v \in L^{\frac{2p}{\alpha}}(\Omega) \mapsto (|v|^\frac{2-\alpha}{\alpha} v) \in L^{\left(\frac{2p}{\alpha}\right)'}(\Omega).
	\end{equation*}
	
	This mapping is hemi-continuous and monotone because the function $\xi\in \mathbb{R} \to (|\xi|^\frac{2-\alpha}{\alpha}\xi) \in \mathbb{R}$ is continuous and monotone. Therefore, by Theorem~12.5-2$(a)$ of~\cite{Ciarlet2025}, we have $w=|v|^\frac{2-\alpha}{\alpha}v \in L^{\left(\frac{2p}{\alpha}\right)'}(\Omega)$. Hence, the convergence in \eqref{conv-3} and the uniqueness of the limit gives:
	\begin{equation}
		\label{conv-3-new}
		(|v_m|^\frac{2-\alpha}{\alpha} v_m) \to w=(|v|^\frac{2-\alpha}{\alpha} v),\quad\text{ in } L^p(\Omega)\text{ as }m\to\infty.
	\end{equation}
	
	To show that $\tilde{\mathscr{M}}$ is relatively compact in $L^\frac{2p}{\alpha}(\Omega)$, we must show that every sequence $\{v_m\}_{m=1}^\infty \subset \tilde{\mathscr{M}}$ has a convergent subsequence in $L^\frac{2p}{\alpha}(\Omega)$.
	
	Since $2p/\alpha>1$, Lemma~\ref{lem:2-0} gives:
	\begin{equation*}
		\begin{aligned}
			&\left|\|v_m\|_{L^\frac{2p}{\alpha}(\Omega)}-\|v\|_{L^\frac{2p}{\alpha}(\Omega)}\right|
			=\left|\left(\int_{\Omega} |v_m|^{\frac{2p}{\alpha}} \dd x\right)^{\frac{\alpha}{2p}}-\left(\int_{\Omega} |v|^{\frac{2p}{\alpha}} \dd x\right)^{\frac{\alpha}{2p}}\right|\\
			&\le\left|\int_{\Omega} |v_m|^{\frac{2p}{\alpha}} \dd x-\int_{\Omega} |v|^{\frac{2p}{\alpha}} \dd x\right|^{\frac{\alpha}{2p}}
			=\left|\int_{\Omega} \left||v_m|^{\frac{2-\alpha}{\alpha}}v_m\right|^p \dd x-\int_{\Omega} \left||v|^{\frac{2-\alpha}{\alpha}}v\right|^p \dd x\right|^{\frac{\alpha}{2p}}\\
			&=\left|\left\||v_m|^{\frac{2-\alpha}{\alpha}}v_m\right\|_{L^p(\Omega)}^p-\left\||v|^{\frac{2-\alpha}{\alpha}}v\right\|_{L^p(\Omega)}^p\right|^{\frac{\alpha}{2p}}.
		\end{aligned}
	\end{equation*}
	
	From \eqref{conv-3-new}, the right-hand side tends to zero as $m \to \infty$, establishing that
	\begin{equation*}
		\|v_m\|_{L^\frac{2p}{\alpha}(\Omega)}\to\|v\|_{L^\frac{2p}{\alpha}(\Omega)},\quad\text{ as }m\to\infty.
	\end{equation*}
	
	Since $L^\frac{2p}{\alpha}(\Omega)$ is uniformly convex, Theorem~5.15-3 of~\cite{Ciarlet2025} implies that 
	\begin{equation*}
		v_m \to v,\quad\textup{ in } L^\frac{2p}{\alpha}(\Omega) \textup{ as }m\to\infty,
	\end{equation*}
	which establishes the desired relative compactness.
	
	$(iii)$ \emph{Application of Lemma~\ref{Dub:dis} and limit identification.} Since $\tilde{\mathscr{M}}$ is relatively compact in $L^{\frac{2p}{\alpha}}(\Omega)$, since
	\begin{equation*}
		\tilde{M}\bigl(|\Lambda_k \bm{u}_{\ell,k}|^{\frac{\alpha-2}{2}} \Lambda_k \bm{u}_{\ell,k}\bigr)=\|\nabla\Lambda_k \bm{u}_{\ell,k}\|_{\bm{L}^p(\Omega)}^{\alpha/2}
	\end{equation*}
	is bounded in $L^\infty(0,T)$ independently of $k$, and since the sequence $\bigl\{D_k\bigl(|\Lambda_k \bm{u}_{\ell,k}|^{\frac{\alpha-2}{2}} \Lambda_k \bm{u}_{\ell,k}\bigr)\bigr\}_{N=1}^\infty$ is bounded in $L^2(0,T;L^2(\Omega))$ independently of~$k$, we apply Dubinskii's lemma for finite difference quotients (Lemma~\ref{Dub:dis}) with $A_0=L^{\frac{2p}{\alpha}}(\Omega)$, $A_1=L^2(\Omega)$, $q_0\in[1,\infty)$, and $q_1=2$. Therefore, the following convergence holds:
	\begin{equation}
		\label{conv-4}
		|\Lambda_k \bm{u}_{\ell,k}|^{\frac{\alpha-2}{2}} \Lambda_k \bm{u}_{\ell,k} \to v_\ell,
		\quad\textup{ in } L^{q_0}(0,T;L^{\frac{2p}{\alpha}}(\Omega)) \textup{ as } N\to\infty.
	\end{equation}
	
	We now identify the limit~$v_\ell$. Define the Nemytskii operator $\hat{A}:L^p((0,T)\times\Omega)\to L^{p'}((0,T)\times\Omega)$ by
	\begin{equation*}
		\hat{A}(w):=|w|^{\frac{\alpha-2}{2}} w,\quad \textup{ for all }w\in L^p((0,T)\times\Omega).
	\end{equation*}
	
	For $w\in L^p((0,T)\times\Omega)$ we have $|\hat{A}(w)|^{p'}=|w|^{\alpha p'/2}$. Thanks to the definition of $\alpha$ (viz.~\eqref{alpha}), the inequality $\frac{\alpha p'}{2}\le p$ reduces to $4p^2-7p+1\ge0$, which is satisfied for $2.8\le p\le5$. Hence, it results that $\hat{A}(w)\in L^{p'}((0,T)\times\Omega)$, and $\hat{A}$ is well defined.
	
	Since $2p/\alpha\ge p'$ (equivalent, again, to $4p^2-7p+1\ge0$), since $\Omega$ is bounded, and since $0<T<\infty$, the following continuous embedding holds:
	\begin{equation*}
		L^{\frac{2p}{\alpha}}((0,T)\times\Omega)\hookrightarrow L^{p'}((0,T)\times\Omega).
	\end{equation*}
	
	Choosing $q_0=p'$ in~\eqref{conv-4}, an application of Theorem~8.28 in~\cite{Leoni2017} gives that the previous strong convergence also holds in $L^{p'}(0,T;L^{p'}(\Omega))$:
	\begin{equation*}
		\hat{A}(\Lambda_k \bm{u}_{\ell,k}) \to v_\ell,\quad\textup{ in } L^{p'}(0,T;L^{p'}(\Omega)) \textup{ as } N\to\infty.
	\end{equation*}
	
	The scalar function $\xi\mapsto|\xi|^{\frac{\alpha-2}{2}}\xi$ is continuous and strictly increasing, therefore the associated Nemytskii operator $A$ is monotone and hemi-continuous from $L^p(0,T;L^p(\Omega))$ to its dual $L^{p'}(0,T;L^{p'}(\Omega))$.  From~\eqref{conv-Dub-1} we have $\Lambda_k \bm{u}_{\ell,k}\rightharpoonup u_\ell$ in $L^p(0,T;L^p(\Omega))$ as $N\to\infty$. Applying Theorem~12.5-2$(a)$ of~\cite{Ciarlet2025} with $X=L^p(0,T;L^p(\Omega))$ and $X'=L^{p'}(0,T;L^{p'}(\Omega))$ yields
	\begin{equation*}
		v_\ell = \hat{A}(u_\ell)=|u_\ell|^{\frac{\alpha-2}{2}} u_\ell,
	\end{equation*}
	and the proof is complete.
\end{proof}

Thanks to Lemma~\ref{Dub:appl-1}, we can apply Lemma~\ref{lem:5:improved} and Lemma~\ref{lem:11}, and establish that:
\begin{equation}
	\label{RHS:thermal}
	\left\{\left(\left\{|\Lambda_k \bm{u}_{\ell,k}|^\frac{p-1}{2p}-z\right\}^+\right)^p\right\}_{N=1}^\infty \textup{ converges strongly in }L^\frac{2p}{p-1}((0,T)\times\Omega\times (0,L)).
\end{equation}

We observe that, according to~\eqref{estimates} it is possible to extract a subsequence such that:
\begin{equation*}
	\Pi_k\tilde{\bm{\Theta}}_{\ell,k}\wsc\tilde{\Theta}_\ell,\quad\textup{ in }L^\infty(0,T;V) \textup{ as } N\to\infty.
\end{equation*}

We have already established in~\eqref{estimates} that $\{\Pi_k\tilde{\bm{\Theta}}_{\ell,k}\}_{N=1}^\infty$ is bounded in $L^\infty(0,T;V)$ independently of $N$ and that $\{D_k(\Pi_k\tilde{\bm{\Theta}}_{\ell, k})\}_{N=1}^\infty $ is bounded in $L^2(0,T;L^2(\Omega\times (0,L)))$ independently of $N$ and $\ell$. Thanks to the Rellich--Kondra\v{s}ov theorem, the dense and compact embedding $V\hookrightarrow\hookrightarrow L^q(\Omega\times (0,L))$ holds for all $2\le q<6=2^\ast$. Therefore, the following chain of embeddings holds:
\begin{equation}
	\label{chain-1}
	V\hookrightarrow\hookrightarrow L^q(\Omega\times (0,L))\hookrightarrow L^2(\Omega\times (0,L))\hookrightarrow L^{q'}(\Omega\times (0,L))\hookrightarrow\hookrightarrow V'.
\end{equation}

An application of the Aubin--Lions--Simon compactness theorem for finite difference quotients in time (Lemma~\ref{ALS:dis}) gives that, up to passing to subsequences, the following convergences hold for all $2\le q<6$:
\begin{equation}
	\label{ALS-fd-lr}
	\begin{aligned}
		\Pi_k\tilde{\bm{\Theta}}_{\ell,k}\to\tilde{\Theta}_\ell&,\quad\textup{ in } L^q(0,T;L^q(\Omega\times (0,L))) \textup{ as }N\to\infty,\\
		\Lambda_k\tilde{\bm{\Theta}}_{\ell,k}\to\tilde{\Theta}_\ell&,\quad\textup{ in }L^q(0,T;L^q(\Omega\times (0,L))) \textup{ as }N\to\infty.
	\end{aligned}
\end{equation}

The Nemitskii operator associated with the operator $A$ introduced in~\eqref{arrhenius} is denoted and defined by:
\begin{equation}
	\label{arrhenius-2}
	\mathcal{A}(\Xi)(t):=A(t,x,\Xi(t)),\quad\textup{ for all }\Xi\in L^1(0,T;L^1(\Omega\times (0,L))).
\end{equation}

Let us now discuss the regularity and continuity of the operators $\mathcal{A}$ and $\mu$.
\begin{lemma}
	\label{A-mu-properties}
	Let $0<T<\infty$, let $\Omega$ be a Lipschitz domain in $\mathbb{R}^2$, and let $0<L<\infty$. Assume that $1\le q<\infty$. Then, the operator $\mathcal{A}$ introduced in~\eqref{arrhenius-2} and the operator $\mu$ defined in~\eqref{mu-def:2} satisfy the following properties:
	\begin{itemize}
		\item[$(a)$] $\mathcal{A}(\Xi)\in\mathcal{C}^0([0,T];L^q(\Omega\times(0,L)))$, for all $\Xi\in\mathcal{C}^0([0,T];L^q(\Omega\times(0,L)))$;
		\item[$(b)$] $\mathcal{A}$ is continuous from $\mathcal{C}^0([0,T];L^q(\Omega\times(0,L)))$ to $\mathcal{C}^0([0,T];L^q(\Omega\times(0,L)))$;
		\item[$(c)$] $\mathcal{A}$ is continuous from $L^q(0,T;L^q(\Omega\times(0,L)))$ to $L^q(0,T;L^q(\Omega\times(0,L)))$;
		\item[$(d)$] $\mu(\Xi)\in\mathcal{C}^0([0,T])$, for all $\Xi\in\mathcal{C}^0([0,T];L^q(\Omega\times(0,L)))$;
		\item[$(e)$] $\mu$ is continuous from $\mathcal{C}^0([0,T];L^q(\Omega\times(0,L)))$ to $\mathcal{C}^0([0,T])$;
		\item[$(f)$] $\mu$ is continuous from $L^q(0,T;L^q(\Omega\times(0,L)))$ to $L^q(0,T)$.
	\end{itemize}
\end{lemma}
\begin{proof}
	The proof is divided into four steps, numbered $(i)$--$(iv)$.
	
	$(i)$ \emph{Proof of properties $(a)$ and $(b)$}. Observe that the operator $\mathcal{A}$ is composition of the affine shift $\Xi\mapsto (\Xi-\Theta_{\textup{m}})$, the operator $(-\{\cdot\}^{-})$ (that is Lipschitz continuous up to a change of sign by Lemma~\ref{lem:1}), multiplication by $-\beta$, the exponential (which is continuous by Lemma~\ref{lemma:exp-2}, noting the input is non-positive), multiplication by $A_0$ as well as the Lipschitz continuous mapping $f\mapsto\max\{f,A_{\textup{min}}\}$.
	Since each link in the composition preserves the $\mathcal{C}^0([0,T];L^q(\Omega\times (0,L)))$ regularity, properties $(a)$ and $(b)$ immediately follow.
	
	$(ii)$ \emph{Proof of property $(c)$}. Let $\{\Xi_n\}_{n=1}^\infty$ be a sequence such that $\Xi_n\to\Xi$ in $L^q(0,T;L^q(\Omega\times (0,L)))$ as $n\to\infty$. The definition of the operator $A$ in~\eqref{arrhenius}, Lemma~\ref{lem:1}$(a)$ and~\eqref{lip:exp} in Lemma~\ref{lemma:exp:1} give:
	\begin{equation*}
		\begin{aligned}
			&\int_{0}^{T}\int_{\Omega\times (0,L)}|\mathcal{A}(\Xi_n)-\mathcal{A}(\Xi)|^q\dd z\dd x\dd t\\
			&\le A_0^q\int_{0}^{T}\int_{\Omega\times (0,L)}\left|\exp\left(-\beta\{\Xi_n-\Theta_{\textup{m}}\}^{-}\right)-\exp\left(-\beta\{\Xi-\Theta_{\textup{m}}\}^{-}\right)\right|^q\dd z\dd x\dd t\\
			&\le A_0^q \beta^q\int_{0}^{T}\int_{\Omega\times (0,L)}|\Xi_n-\Xi|^q\dd z\dd x\dd t,
		\end{aligned}
	\end{equation*}
	and we observe that the assumed convergence for the sequence $\{\Xi_n\}_{n=1}^\infty$ implies that the latter term tends to zero as $n\to\infty$.
	
	$(iii)$ \emph{Proof of properties $(d)$ and $(e)$}. Properties $(d)$ and $(e)$ immediately follow as a result of an application of Lemma~\ref{lem:avg-2} and Remark~\ref{rem:lem:avg}.
	
	$(iv)$ \emph{Proof of property $(f)$}. Let $\{\Xi_n\}_{n=1}^\infty$ be a sequence such that $\Xi_n\to\Xi$ in $L^q(0,T;L^q(\Omega\times (0,L)))$ as $n\to\infty$. The definition of $\mu$ given in~\eqref{mu-def:2}, the non-expansiveness of $\{\cdot\}^{+}$ and $\{\cdot\}^{-}$, and the fact that the best Lipschitz constant for $\exp$ with negative inputs is $1$ give:
	\begin{equation*}
		\begin{aligned}
			&\int_{0}^{T}|\mu(\Xi_n)-\mu(\Xi)|^q\dd t\le \left(2\left(\rho g \dfrac{p-1}{2p}\right)^{p-1}A_0 \beta\right)^q(L|\Omega|)^{\frac{q-2}{2}}\int_{0}^{T}\int_{\Omega\times (0,L)}|\Xi_n-\Xi|^q\dd z \dd x\dd t.
		\end{aligned}
	\end{equation*}
	
	The assumed convergence for the sequence $\{\Xi_n\}_{n=1}^\infty$ gives that the right-hand side of the estimate above tends to zero as $n\to\infty$.
\end{proof}

We are finally able to establish the strong convergence of the right-hand side of~\eqref{FD:temp} in a suitable Bochner space.
\begin{theorem}
	\label{rhs-conv}
	Let $0<T<\infty$, let $\Omega$ be a Lipschitz domain in $\mathbb{R}^2$, and let $0<L<\infty$. Assume that $\frac{10}{3}\le q\le\frac{28}{5}$, and that $2.8\le p\le 5$. Let $\Lambda_k\bm{u}_{\ell,k}$ and $\Lambda_k\tilde{\bm{\Theta}}_{\ell,k}$ be the piecewise constant functions introduced in~\eqref{PilL} and~\eqref{PilTL}, respectively. Let $\sigma>0$ be defined in a way that:
	\begin{equation*}
		\dfrac{1}{\sigma}=\dfrac{1}{q}+\dfrac{p-1}{2p}.
	\end{equation*}
	
	Then, it results that $1\le q'\le\sigma\le 2$, and that the sequence $\{\Lambda_k \bm{h}_{\ell,k}\}_{N=1}^\infty$ with elements
	\begin{equation*}
		\Lambda_k \bm{h}_{\ell,k}:=\mathcal{A}(\Lambda_k\tilde{\bm{\Theta}}_{\ell,k})\left(\left\{|\Lambda_k\bm{u}_{\ell,k}|^\frac{p-1}{2p}-z\right\}^{+}\right)^p
	\end{equation*}
	admits a subsequence that converges strongly to $h_\ell:=\mathcal{A}(\tilde{\Theta}_{\ell})\left(\left\{|u_\ell|^\frac{p-1}{2p}-z\right\}^{+}\right)^p$ in $L^\sigma(0,T;L^\sigma(\Omega\times (0,L)))$ as $N\to\infty$.
\end{theorem}
\begin{proof}
	To begin with, observe that if $2.8\le p\le 5$ and $\frac{10}{3}\le q\le\frac{28}{5}$, an application of Lemma~\ref{lem:V'} gives $q'\le\sigma\le 2$ and $2\le\sigma'\le q$.
	Additionally, again thanks to Lemma~\ref{lem:V'}, the chain of embeddings~\eqref{chain-1} can be improved as follows:
	\begin{equation}
		\label{chain-2}
		\begin{aligned}
			&V\hookrightarrow\hookrightarrow L^q(\Omega\times (0,L))\hookrightarrow L^{\sigma'}(\Omega\times (0,L)) \hookrightarrow L^2(\Omega\times (0,L))\\
			&\quad\hookrightarrow L^\sigma(\Omega\times (0,L))\hookrightarrow L^{q'}(\Omega\times (0,L))\hookrightarrow\hookrightarrow V'.
		\end{aligned}
	\end{equation}
	
	By the Aubin--Lions--Simon theorem, up to passing to suitable subsequences with common indices, we obtain that
	\begin{equation*}
		\Lambda_k\tilde{\bm{\Theta}}_{\ell,k}\to\tilde{\Theta}_\ell,\quad\textup{ in }L^q(0,T;L^q(\Omega\times(0,L))),
	\end{equation*}
	as $N\to\infty$ and that~\eqref{RHS:thermal} holds.
	
	An application of the generalised H\"{o}lder inequality (cf., e.g., Remark~2 on page~93 of~\cite{Brez11}), \eqref{RHS:thermal}, Lemma~\ref{A-mu-properties}$(c)$ and Theorem~8.28 in~\cite{Leoni2017} give
	\begin{equation*}
		\begin{aligned}
			&\left(\int_{(0,T)\times\Omega\times (0,L)}\left|\mathcal{A}(\Lambda_k\tilde{\bm{\Theta}}_{\ell,k})\left(\left\{|\Lambda_k\bm{u}_{\ell,k}|^\frac{p-1}{2p}-z\right\}^{+}\right)^p-\mathcal{A}(\tilde{\Theta}_\ell)\left(\left\{|u_\ell|^\frac{p-1}{2p}-z\right\}^{+}\right)^p\right|^\sigma \dd z\dd x \dd t\right)^{1/\sigma}\\
			&\le\|\mathcal{A}(\Lambda_k\tilde{\bm{\Theta}}_{\ell,k})\|_{L^q((0,T)\times\Omega\times (0,L))} \left\|\left(\left\{|\Lambda_k\bm{u}_{\ell,k}|^\frac{p-1}{2p}-z\right\}^{+}\right)^p-\left(\left\{|u_\ell|^\frac{p-1}{2p}-z\right\}^{+}\right)^p\right\|_{L^\frac{2p}{p-1}((0,T)\times\Omega\times (0,L))}\\
			&\quad+\|\mathcal{A}(\Lambda_k\tilde{\bm{\Theta}}_{\ell,k})-\mathcal{A}(\tilde{\Theta}_\ell)\|_{L^q((0,T)\times\Omega\times (0,L))} \left\|\left(\left\{|u_\ell|^\frac{p-1}{2p}-z\right\}^{+}\right)^p\right\|_{L^\frac{2p}{p-1}((0,T)\times\Omega\times (0,L))},
		\end{aligned}
	\end{equation*}
	and we observe that the addends on the right-hand side tend to zero as $N\to\infty$.
\end{proof}

\begin{remark}
	\label{rem:6}
	Thanks to~\eqref{chain-2} and the conclusion in Theorem~\ref{rhs-conv}, we infer that the sequence
	\begin{equation*}
		\left\{\mathcal{A}(\Lambda_k\tilde{\bm{\Theta}}_{\ell,k})\left(\left\{|\Lambda_k\bm{u}_{\ell,k}|^\frac{p-1}{2p}-z\right\}^{+}\right)^p\right\}_{N=1}^\infty
	\end{equation*}
	admits a subsequence that strongly converges to $\mathcal{A}(\tilde{\Theta}_{\ell})\left(\left\{|u_\ell|^\frac{p-1}{2p}-z\right\}^{+}\right)^p$ in $L^\sigma(0,T;V')$ as $N\to\infty$.
	\bqed
\end{remark}

\section{Existence of weak solutions for Problem~\ref{Pkappa}}\label{Sec:3:ter}

Thanks to~\eqref{estimates}, we can establish the existence of solutions for Problem~\ref{Pkappa}.

\begin{theorem}
	\label{thm:3}
	Let $T>0$, $\Omega \subset \mathbb{R}^2$ and $p$ be as in Section~\ref{Sec:2} and let $\alpha$ be as in~\eqref{alpha}. Let $\ell>0$ be given, let $N \ge 1$ be an integer, and define $k:=T/N$.
	Assume that $(H\ref{H1})$--$(H\ref{H4})$ hold.
	The a priori estimates~\eqref{estimates} imply that the following convergence process takes place (recall that $B_\ell$ has been defined in~\eqref{Bkappa})
	\begin{equation}
		\label{conv-proc}
		\begin{aligned}
			\Pi_k \bm{u}_{\ell,k} \wsc u_\ell &\textup{ in } L^\infty(0,T;W_0^{1,p}(\Omega)) \textup{ as }N\to\infty,\\
			|\Pi_k \bm{u}_{\ell,k}|^{\frac{\alpha-2}{2}} \Pi_k \bm{u}_{\ell,k}\wsc v_\ell &\textup{ in } L^\infty(0,T;L^2(\Omega))\textup{ as }N\to\infty,\\
			D_k(|\Pi_k \bm{u}_{\ell,k}|^{\frac{\alpha-2}{2}}\Pi_k\bm{u}_{\ell,k})\wsc\dfrac{\dd v_\ell}{\dd t} &\textup{ in } L^2(0,T;L^2(\Omega))\textup{ as }N\to\infty,\\
			|\Pi_k \bm{u}_{\ell,k}|^{\alpha-2} \Pi_k \bm{u}_{\ell,k}\wsc w_\ell &\textup{ in } L^\infty(0,T;L^{\alpha'}(\Omega))\textup{ as }N\to\infty,\\
			D_k(|\Pi_k \bm{u}_{\ell,k}|^{\alpha-2}\Pi_k\bm{u}_{\ell,k})\wsc\dfrac{\dd w_\ell}{\dd t} &\textup{ in } L^\infty(0,T;W^{-1,p'}(\Omega))\textup{ as }N\to\infty,\\
			|u_{\ell,k}^{N}|^{\alpha-2} u_{\ell,k}^{N}\rightharpoonup\chi_\ell&\textup{ in } L^{\alpha'}(\Omega)\textup{ as }N\to\infty,\\
			B_\ell(\Pi_k \bm{u}_{\ell,k})\wsc g_\ell&\textup{ in } L^\infty(0,T;W^{-1,p'}(\Omega))\textup{ as }N\to\infty,\\
			\Pi_k\tilde{\bm{\Theta}}_{\ell, k} \wsc \tilde{\Theta}_\ell&\textup{ in }L^\infty(0,T;V)\textup{ as }N\to\infty,\\
			-\Delta_b(\Pi_k\tilde{\bm{\Theta}}_{\ell, k})\wsc-\Delta_b(\tilde{\Theta}_\ell)&\textup{ in } L^\infty(0,T;V')\textup{ as }N\to\infty,\\
			\tilde{\Theta}_{\ell,k}^{N}\rightharpoonup\tilde{\tau}_\ell&\textup{ in } V \textup{ as }N\to\infty,\\
			D_k(\Pi_k\tilde{\bm{\Theta}}_{\ell, k}) \rightharpoonup \dfrac{\dd \tilde{\Theta}_\ell}{\dd t}&\textup{ in }L^2(0,T;L^2(\Omega\times (0,L)))\textup{ as }N\to\infty,
		\end{aligned}
	\end{equation}
	and the weak-star limits $g_\ell$, $v_\ell$, and $w_\ell$ satisfy:
	\begin{equation*}
		\begin{aligned}
			\langle g_\ell,v\rangle_{W^{-1,p'}(\Omega), W_0^{1,p}(\Omega)}&=\int_{\Omega}|\nabla u_\ell|^{p-2}\nabla u_\ell\cdot\nabla v \dd x,\\
			v_\ell&=|u_\ell|^\frac{\alpha-2}{2} u_\ell,\\
			w_\ell&=|u_\ell|^{\alpha-2} u_\ell,\\
			\chi_\ell&=|u_\ell(T)|^{\alpha-2}u_\ell(T),\\
			\tilde{\tau}_\ell&=\tilde{\Theta}_\ell(T).
		\end{aligned}
	\end{equation*}
	
	Besides, the pair $(u_\ell,\tilde{\Theta}_\ell)$ is a solution for Problem~\ref{Pkappa}.
\end{theorem}
\begin{proof}
	The proof is divided into five steps, numbered $(i)$--$(v)$. The discussion of the thermal component of the model is novel and requires substantial care.
	
	The convergences in~\eqref{conv-proc} hold by virtue of the Banach--Alaoglu--Bourbaki theorem (cf., e.g., Theorem~3.6 of~\cite{Brez11}), the estimates~\eqref{estimates} and~\eqref{chi-kappa}.
	The non-trivial part of the proof amounts to identifying the weak-star limits $v_\ell$ and $w_\ell$, as well as to show that the weak-star limits $u_\ell$ and $\tilde{\Theta}_\ell$ solve Problem~\ref{Pkappa}.
	
	$(i)$ \emph{It results that $w_\ell=|u_\ell|^{\alpha-2} u_\ell$}. Given $u \in W_0^{1,p}(\Omega)$, set $v:=|u|^{\alpha-2} u$ and observe that if $\beta=\alpha'$ then $|v|^{\beta-2} v \in W_0^{1,p}(\Omega)$ so that the set
	\begin{equation*}
		S:=\{v;(|v|^{\alpha'-2}v) \in W_0^{1,p}(\Omega)\}
	\end{equation*}
	is non-empty. Consider the definite homogeneous gauge
	\begin{equation*}
		M(v):=\left\|\nabla(|v|^{\alpha'-2}v)\right\|_{\bm{L}^p(\Omega)}^\frac{1}{\alpha'-1},\quad\textup{ for all }v \in S,
	\end{equation*}
	and define the set
	\begin{equation*}
		\mathscr{M}:=\{v\in S; M(v)\le 1\}.
	\end{equation*}
	
	An application of the Poincar\'e--Friedrichs inequality gives that there exists a constant $c_0=c_0(\Omega)>0$ such that
	\begin{equation*}
		1 \ge M(v) \ge c_P^{-\frac{1}{\alpha'-1}}\left(\int_{\Omega} |v|^{(\alpha'-1)p}\dd x\right)^{1/(p(\alpha'-1))}=c_P^{-\frac{1}{\alpha'-1}}\|v\|_{L^{(\alpha'-1)p}(\Omega)}.
	\end{equation*}
	
	It is then possible to show that $\mathscr{M}$ is relatively compact in $L^{(\alpha'-1)p}(\Omega)$ so that an application of Lemma~\ref{Dub:dis} with $A_0=L^{(\alpha'-1)p}(\Omega)$, $A_1=W^{-1,p'}(\Omega)$, $q_0=q_1=2$ gives
	\begin{equation*}
		\label{conv-2}
		|\Pi_k \bm{u}_{\ell,k}|^{\alpha-2} \Pi_k \bm{u}_{\ell,k} \to w_\ell,\quad\textup{ in } L^2(0,T;L^{(\alpha'-1)p}(\Omega)) \textup{ as } N\to\infty,
	\end{equation*}
	where, once again, the monotonicity of the mapping $\xi\in\mathbb{R} \mapsto |\xi|^{\alpha-2} \xi$, the first convergence in the process~\eqref{conv-proc} and Theorem~12.5-2$(a)$ of~\cite{Ciarlet2025} imply that:
	\begin{equation*}
		w_\ell = |u_\ell|^{\alpha-2} u_\ell.
	\end{equation*}
	
	$(ii)$ \emph{Preliminary remarks for the recovery of the weak formulation~\eqref{penalty:seq}}.
	Let $v \in \mathcal{D}(\Omega)$ and let $\psi \in \mathcal{C}^1([0,T])$. For each $0 \le n \le N-1$, multiply~\eqref{penalty:seq} by $\{v \psi(nk)\}$, getting:
	\begin{equation}
		\label{step:1}
		\begin{aligned}
			&\dfrac{\psi(nk)}{k}\int_{\Omega}\{|u_{\ell,k}^{n+1}|^{\alpha-2} u_{\ell,k}^{n+1} - |u_{\ell,k}^{n}|^{\alpha-2}u_{\ell,k}^{n}\} v \dd x\\
			&\quad+\psi(nk)\int_{\Omega}\mu(\tilde{\Theta}_{\ell,k}^{n+1}) |\nabla u_{\ell,k}^{n+1}|^{p-2} \nabla u_{\ell,k}^{n+1} \cdot \nabla v\dd x
			-\psi(nk)\dfrac{\mu(\tilde{\Theta}_{\ell,k}^{n+1})}{\ell}\int_{\Omega}\{u_{\ell,k}^{n+1}\}^{-} v \dd x\\
			&=\int_{\Omega}\left(\dfrac{1}{k} \int_{nk}^{(n+1)k} \tilde{a}(t) \dd t\right) v \psi(nk) \dd x.
		\end{aligned}
	\end{equation}
	
	Define the mapping $\psi_k:[0,T]\to\mathbb{R}$ in a way that $\psi_k(t):=\psi(nk)$, for all $nk\le t< (n+1)k$. Observe that the following estimate holds:
	\begin{equation*}
		|\psi_k(t)-\psi(t)|\le\omega(k):=\max_{\substack{t_1, t_2 \in [0,T]\\|t_1-t_2|\le k}}|\psi(t_1)-\psi(t_2)|,\quad\textup{ for all }t\in[0,T].
	\end{equation*}
	
	Since the right-hand side of the estimate above is independent of $t$, and since the continuity of $\psi$ implies that $\omega(k)\to 0$ as $N\to\infty$, we obtain that $\|\psi_k-\psi\|_{L^\infty(0,T)}\to 0$ as $N\to\infty$ and thus, by H\"{o}lder's inequality, in $L^r(0,T)$ for all $4\le r\le\infty$. By virtue of~\eqref{ALS-fd-lr}, we are in position to apply Theorem~\ref{rhs-conv} and Lemma~\ref{A-mu-properties}$(f)$ and deduce, in particular, that $\mu(\Pi_k\tilde{\bm{\Theta}}_{\ell,k})\to\mu(\tilde{\Theta}_\ell)$ in $L^q(0,T)$, for any $4\le q\le\frac{28}{5}$, as $N\to\infty$. An application of the generalised H\"{o}lder inequality (cf., e.g., Remark~2 on page~93 in~\cite{Brez11}) gives that:
	\begin{equation}
		\label{conv-mu-psi}
		\mu(\Pi_k\tilde{\bm{\Theta}}_{\ell,k})\psi_k\to\mu(\tilde{\Theta}_\ell)\psi,\quad\textup{ in } L^\upsilon(0,T),
	\end{equation}
	as $N\to\infty$ where $\upsilon$ is defined in a way that $\frac{1}{\upsilon}:=\frac{1}{r}+\frac{1}{q}$ and, moreover, it results $2\le\upsilon<6$.
	
	Multiplying~\eqref{step:1} by $k$ and summing over $0 \le n \le N-1$ gives:
	\begin{equation}
		\label{step:3}
		\begin{aligned}
			&\int_{0}^{T} \int_{\Omega} D_k(|\Pi_k \bm{u}_{\ell,k}|^{\alpha-2}\Pi_k \bm{u}_{\ell,k}) v \dd x \psi_k(t) \dd t\\
			&\quad-\int_{0}^{T}\mu(\Pi_k\tilde{\bm{\Theta}}_{\ell,k})\int_{\Omega} \nabla \cdot \left(|\nabla (\Pi_k \bm{u}_{\ell,k})|^{p-2} \nabla (\Pi_k \bm{u}_{\ell,k})\right) v\dd x \psi_k(t) \dd t\\
			&\quad-\dfrac{1}{\ell} \int_{0}^{T}\mu(\Pi_k\tilde{\bm{\Theta}}_{\ell,k})\int_{\Omega} \{\Pi_k \bm{u}_{\ell,k}\}^{-} v\dd x \psi_k(t) \dd t
			=\int_{0}^{T} \left(\int_{\Omega} \tilde{a}(t) v \dd x\right) \psi_k(t) \dd t.
		\end{aligned}
	\end{equation}
	
	Letting $N\to\infty$ in~\eqref{step:3} and exploiting the convergence process~\eqref{conv-proc}, item $(i)$, and~\eqref{conv-mu-psi} give:
	\begin{equation}
		\label{step:4}
		\begin{aligned}
			&\int_{0}^{T} \left\langle \dfrac{\dd }{\dd t}\left(|u_\ell|^{\alpha-2} u_\ell\right), v \right\rangle_{W^{-1,p'}(\Omega), W_0^{1,p}(\Omega)} \psi(t) \dd t
			+\int_{0}^{T} \langle g_\ell(t), v \rangle_{W^{-1,p'}(\Omega), W_0^{1,p}(\Omega)} \mu(\tilde{\Theta}_\ell)\psi(t) \dd t\\
			&\quad-\dfrac{1}{\ell}\int_{0}^{T}\mu(\tilde{\Theta}_\ell)\psi(t)\int_{\Omega}\{u_\ell(t)\}^{-}v\dd x\dd t=\int_{0}^{T} \int_{\Omega} \tilde{a}(t) v \dd x \psi(t) \dd t.
		\end{aligned}
	\end{equation}
	
	Since by the Sobolev embedding theorem (cf., e.g., Theorem~8.4-1 of~\cite{Ciarlet2025}) it results $L^{\alpha'}(\Omega) \hookrightarrow W^{-1,p'}(\Omega)$, an integration by parts in~\eqref{step:4}, and the fact that $u_\ell\in\mathcal{C}^0([0,T];L^\alpha(\Omega))$ (Lemma~\ref{lem:3}) give:
	\begin{equation}
		\label{step:7}
		\begin{aligned}
			&-\int_{0}^{T} \int_{\Omega} |u_\ell|^{\alpha-2} u_\ell v \dd x \dfrac{\dd \psi}{\dd t} \dd t
			+\langle|u_\ell(T)|^{\alpha-2} u_\ell(T), \psi(T) v\rangle_{W^{-1,p'}(\Omega), W_0^{1,p}(\Omega)}\\
			&\quad -\int_{\Omega} |u_\ell(0)|^{\alpha-2} u_\ell(0) \psi(0) v \dd x + \int_{0}^{T} \langle g_\ell(t), v \rangle_{W^{-1,p'}(\Omega), W_0^{1,p}(\Omega)} \mu(\tilde{\Theta}_\ell)\psi(t) \dd t\\
			&\quad-\dfrac{1}{\ell}\int_{0}^{T}\mu(\tilde{\Theta}_\ell)\psi(t)\int_{\Omega}\{u_\ell(t)\}^{-}v\dd x\dd t=\int_{0}^{T}\int_{\Omega} \tilde{a}(t) v \dd x \psi(t) \dd t.
		\end{aligned}
	\end{equation}
	
	Observe that the first term on the left-hand side of equation~\eqref{step:3} can be rearranged as follows
	\begin{equation}
		\label{step:5}
		\begin{aligned}
			&\int_{\Omega} |u_{\ell,k}^{N}|^{\alpha-2} u_{\ell,k}^{N} v \psi(T-k) \dd x-\int_{\Omega} |u_0|^{\alpha-2} u_0 v \psi(0) \dd x\\
			&\quad-\sum_{n=0}^{N-2} \int_{nk}^{(n+1)k}\int_{\Omega} |\Pi_k\bm{u}_{\ell,k}(t)|^{\alpha-2} \Pi_k\bm{u}_{\ell,k}(t) v \psi'(\theta n k+(1-\theta)(1+n)k) \dd x \dd t\\
			&\quad +\sum_{n=0}^{N-1} \int_{nk}^{(n+1)k} \int_{\Omega}|\nabla (\Pi_k\bm{u}_{\ell,k}(t))|^{p-2} \nabla (\Pi_k\bm{u}_{\ell,k}(t)) \cdot \nabla(\mu(\Pi_k\tilde{\bm{\Theta}}_{\ell,k}) \psi_k(t) v)\dd x \dd t\\
			&\quad-\dfrac{1}{\ell} \sum_{n=0}^{N-1} \int_{nk}^{(n+1)k} \int_{\Omega} \{\Pi_k\bm{u}_{\ell,k}(t)\}^{-} \left(\mu(\Pi_k\tilde{\bm{\Theta}}_{\ell,k}) \psi_k(t) v\right) \dd x \dd t\\
			&=\sum_{n=0}^{N-1} \int_{nk}^{(n+1)k} \int_{\Omega}\left(\dfrac{1}{k} \int_{nk}^{(n+1)k} \tilde{a}(\tau) \dd \tau\right) v \psi_k(t) \dd x \dd t,
		\end{aligned}
	\end{equation}
	for some $0<\theta<1$ (cf., e.g., Theorem~9.10-1$(c)$ of~\cite{Ciarlet2025}).
	Thanks to~\eqref{conv-proc}, item~$(i)$, Lemma~\ref{A-mu-properties}$(f)$, and the fact that $\psi\in\mathcal{C}^1([0,T])$, we obtain that letting $N\to\infty$ in~\eqref{step:5} gives:
	\begin{equation}
		\label{step:6}
		\begin{aligned}
			&-\int_{0}^{T} \int_{\Omega} |u_\ell|^{\alpha-2} u_\ell v \dd x \dfrac{\dd \psi}{\dd t}\dd t
			+\int_{0}^{T} \langle g_\ell(t), v \rangle_{W^{-1,p'}(\Omega), W_0^{1,p}(\Omega)} \mu(\tilde{\Theta}_\ell)\psi(t) \dd t\\
			&\quad+\int_{\Omega} \left(\chi_\ell \psi(T)-|u_0|^{\alpha-2}u_0 \psi(0)\right) v \dd x-\dfrac{1}{\ell}\int_{0}^{T}\int_{\Omega}\{u_\ell(t)\}^{-}\left(\mu(\tilde{\Theta}_\ell)\psi(t)v\right)\dd x\dd t\\
			&= \int_{0}^{T} \int_{\Omega} \tilde{a}(t) v \dd x\psi(t) \dd t.
		\end{aligned}
	\end{equation}
	
	Comparing equations~\eqref{step:6} and~\eqref{step:7} gives:
	\begin{equation}
		\label{step:8}
		\begin{aligned}
			&\langle|u_\ell(T)|^{\alpha-2} u_\ell(T), \psi(T) v\rangle_{W^{-1,p'}(\Omega), W_0^{1,p}(\Omega)} -\int_{\Omega} |u_\ell(0)|^{\alpha-2} u_\ell(0) \psi(0) v \dd x\\
			&=\int_{\Omega} \left(\chi_\ell \psi(T)-|u_0|^{\alpha-2}u_0 \psi(0)\right) v \dd x.
		\end{aligned}
	\end{equation}
	
	Since $\psi \in \mathcal{C}^1([0,T])$ is arbitrarily chosen, specialising $\psi$ in~\eqref{step:8} in a way such that $\psi(0)=0$ gives:
	\begin{equation}
		\label{step:9}
		\langle|u_\ell(T)|^{\alpha-2} u_\ell(T) - \chi_\ell, \psi(T) v\rangle_{W^{-1,p'}(\Omega), W_0^{1,p}(\Omega)}=0.
	\end{equation}
	
	Since the duality in~\eqref{step:9} is continuous with respect to $v$, since $v$ has been chosen arbitrarily in $\mathcal{D}(\Omega)$, and since $\mathcal{D}(\Omega)$ is, by definition, dense in $W_0^{1,p}(\Omega)$, and since $\chi_\ell \in L^{\alpha'}(\Omega)$, we infer that:
	\begin{equation}
		\label{step:10}
		|u_\ell(T)|^{\alpha-2} u_\ell(T) = \chi_\ell \in L^{\alpha'}(\Omega).
	\end{equation}
	
	Specialising $\psi$ in~\eqref{step:8} in a way such that $\psi(T)=0$ gives:
	\begin{equation}
		\label{step:11}
		\int_{\Omega} \left(|u_\ell(0)|^{\alpha-2} u_\ell(0) - |u_0|^{\alpha-2}u_0 \right)\psi(0) v \dd x=0.
	\end{equation}
	
	Since the integration in~\eqref{step:11} is continuous with respect to $v$, since $v$ has been chosen arbitrarily in $\mathcal{D}(\Omega)$, and since $\mathcal{D}(\Omega)$ is, by definition, dense in $W_0^{1,p}(\Omega)$, we immediately infer that
	\begin{equation*}
		|u_\ell(0)|^{\alpha-2} u_\ell(0) = |u_0|^{\alpha-2}u_0,
	\end{equation*}
	so that the injectivity of the monotone and hemi-continuous operator $\xi \mapsto |\xi|^{\alpha-2} \xi$ in turn implies that:
	\begin{equation}
		\label{step:12}
		u_\ell(0)=u_0 \in K_{\textup{surf}}.
	\end{equation}
	
	$(iii)$ \emph{Identification of the weak-star limit $g_\ell$ via Minty's trick}. For each $0 \le n \le N-1$, we multiply~\eqref{penalty:seq} by $k u_{\ell,k}^{n+1}$, sum over $n=0,\dots, N-1$, and apply Lemma~\ref{lem:7}, thus getting
	\begin{equation*}
		\begin{aligned}
			&\dfrac{1}{\alpha'} \sum_{n=0}^{N-1} \left\{\left\||u_{\ell,k}^{n+1}|^{\alpha-2}u_{\ell,k}^{n+1}\right\|_{L^{\alpha'}(\Omega)}^{\alpha'} - \left\||u_{\ell,k}^{n}|^{\alpha-2}u_{\ell,k}^{n}\right\|_{L^{\alpha'}(\Omega)}^{\alpha'}\right\}\\
			&\quad+\sum_{n=0}^{N-1} k \mu(\tilde{\Theta}_{\ell,k}^{n+1}) \left\langle B_\ell(u_{\ell,k}^{n+1}), u_{\ell,k}^{n+1} \right\rangle_{W^{-1,p'}(\Omega), W_0^{1,p}(\Omega)}\\
			&\quad+\dfrac{1}{\ell}\sum_{n=0}^{N-1}k \mu(\tilde{\Theta}_{\ell,k}^{n+1})\|\{u_{\ell,k}^{n+1}\}^{-}\|_{L^2(\Omega)}^2\\
			&\le \sum_{n=0}^{N-1} k \int_{\Omega} \left(\dfrac{1}{k}\int_{nk}^{(n+1)k} \tilde{a}(t)\dd t\right) u_{\ell,k}^{n+1} \dd x,
		\end{aligned}
	\end{equation*}
	which in turn implies:
	\begin{equation}
		\label{step:13}
		\begin{aligned}
			&\dfrac{1}{\alpha'} \left\||u_{\ell,k}^{N}|^{\alpha-2}u_{\ell,k}^{N}\right\|_{L^{\alpha'}(\Omega)}^{\alpha'}
			+\int_{0}^{T}\mu(\Pi_k\tilde{\bm{\Theta}}_{\ell,k})\left\langle B_\ell(\Pi_k \bm{u}_{\ell,k}), \Pi_k \bm{u}_{\ell,k} \right\rangle_{W^{-1,p'}(\Omega), W_0^{1,p}(\Omega)}\dd t\\
			&\quad+\dfrac{1}{\ell}\int_{0}^{T}\mu(\Pi_k\tilde{\bm{\Theta}}_{\ell,k})\|\{\Pi_k \bm{u}_{\ell,k}(t)\}^{-}\|_{L^2(\Omega)}^2\dd t\\
			&\le \int_{0}^{T} \int_{\Omega} \tilde{a}(t) \Pi_k \bm{u}_{\ell,k}(t) \dd x \dd t
			+\dfrac{1}{\alpha'} \left\||u_0|^{\alpha-2}u_0\right\|_{L^{\alpha'}(\Omega)}^{\alpha'}.
		\end{aligned}
	\end{equation}
	
	Combining the fact that $\mu(\Pi_k\tilde{\bm{\Theta}}_{\ell,k})\to\mu(\tilde{\Theta}_\ell)$ in $L^q(0,T)$, for any $4\le q\le\frac{28}{5}$, as $N\to\infty$ (itself consequence of~\eqref{ALS-fd-lr} in the particular case where $4\le q\le\frac{28}{5}$, and Lemma~\ref{A-mu-properties}$(f)$), with the fact that $\{\Pi_k\bm{u}_{\ell,k}\}_{N=1}^\infty$ is bounded in $L^\infty(0,T;W^{1,p}_0(\Omega))$ independently of $N$ and $\ell$ (viz.~\eqref{conv-proc}), the properties of the $\liminf$, and the properties of weak convergence gives:
	\begin{equation}
		\label{liminf-neg-part}
		\begin{aligned}
			&\liminf_{N\to\infty}\left(\dfrac{1}{\ell}\int_{0}^{T}\mu(\Pi_k\tilde{\bm{\Theta}}_{\ell,k})\|\{\Pi_k\bm{u}_{\ell,k}(t)\}^{-}\|_{L^2(\Omega)}^2\dd t\right)\\
			&=\liminf_{N\to\infty}\bigg(\dfrac{1}{\ell}\int_{0}^{T}\left(\mu(\Pi_k\tilde{\bm{\Theta}}_{\ell,k})-\mu(\tilde{\Theta}_\ell)\right)\|\{\Pi_k\bm{u}_{\ell,k}(t)\}^{-}\|_{L^2(\Omega)}^2\dd t\\
			&\quad+\dfrac{1}{\ell}\int_{0}^{T}\mu(\tilde{\Theta}_\ell)\|\{\Pi_k\bm{u}_{\ell,k}(t)\}^{-}\|_{L^2(\Omega)}^2\dd t\bigg)\\
			&=\liminf_{N\to\infty}\left(\dfrac{1}{\ell}\int_{0}^{T}\mu(\tilde{\Theta}_\ell)\|\{\Pi_k\bm{u}_{\ell,k}(t)\}^{-}\|_{L^2(\Omega)}^2\dd t\right)\\
			&\ge\dfrac{1}{\ell}\int_{0}^{T}\mu(\tilde{\Theta}_\ell)\|\{u_\ell(t)\}^{-}\|_{L^2(\Omega)}^2\dd t.
		\end{aligned}
	\end{equation}
	
	Passing to the $\liminf$ as $N\to\infty$ in~\eqref{step:13}, we obtain that combining the convergence process~\eqref{conv-proc} with~\eqref{step:10}, \eqref{step:12} and~\eqref{liminf-neg-part} gives:
	\begin{equation}
		\label{step:14}
		\begin{aligned}
			&\dfrac{1}{\alpha'} \|u_{\ell}(T)\|_{L^{\alpha}(\Omega)}^{\alpha}
			+\liminf_{N\to\infty}\int_{0}^{T}\mu(\Pi_k\tilde{\bm{\Theta}}_{\ell,k}) \left\langle B_\ell(\Pi_k \bm{u}_{\ell,k}), \Pi_k \bm{u}_{\ell,k} \right\rangle_{W^{-1,p'}(\Omega), W_0^{1,p}(\Omega)} \dd t\\
			&\le \int_{0}^{T} \int_{\Omega} \tilde{a}(t) u_\ell \dd x \dd t+
			\dfrac{1}{\alpha'} \left\||u_0|^{\alpha-2}u_0\right\|_{L^{\alpha'}(\Omega)}^{\alpha'}-\dfrac{1}{\ell}\int_{0}^{T}\mu(\tilde{\Theta}_\ell)\|\{u_\ell(t)\}^{-}\|_{L^2(\Omega)}^2\dd t\\
			&=\int_{0}^{T} \int_{\Omega} \tilde{a}(t) u_\ell \dd x \dd t+\dfrac{1}{\alpha'}\|u_0\|_{L^\alpha(\Omega)}^\alpha-\dfrac{1}{\ell}\int_{0}^{T}\mu(\tilde{\Theta}_\ell)\|\{u_\ell(t)\}^{-}\|_{L^2(\Omega)}^2\dd t.
		\end{aligned}
	\end{equation}
	
	Observe that $u_\ell$ regarded as an element of $L^r(0,T;W^{1,p}_0(\Omega))$ with $4\le r<\infty$ can be approximated, in the sense of the strong convergence in $L^r(0,T;W^{1,p}_0(\Omega))$, by a function of the form $(\psi(t)v(x))$, where $\psi\in\mathcal{C}^1[0,T]$ and $v\in\mathcal{D}(\Omega)$ (cf., e.g., Theorem~8.21 in~\cite{Leoni2017}).
	An application of the triangle inequality thus gives that there exists a sequence $\{\psi_{k,\ell}\}_{N=1}^\infty\subset\mathcal{C}^1([0,T])$ and a sequence $\{v_{N,\ell}\}_{N=1}^\infty\subset\mathcal{D}(\Omega)$ such that
	\begin{equation}
		\label{step:14-2}
		\psi_{k,\ell} v_{N,\ell} \to u_\ell,\quad\textup{ in }L^r(0,T;W^{1,p}_0(\Omega))\textup{ as }N\to\infty,
	\end{equation}
	for all $4\le r<\infty$. In light of~\eqref{ALS-fd-lr} and Lemma~\ref{A-mu-properties}$(f)$ it results, in particular, that $\mu(\Pi_k\tilde{\bm{\Theta}}_{\ell,k})\to\mu(\tilde{\Theta}_\ell)$ in $L^q(0,T)$ as $N\to\infty$, for any $4\le q\le\frac{28}{5}$. Define the real number $s$ by $\frac{1}{s}:=\frac{1}{q}+\frac{1}{r}$, and observe that $s$ varies in the range $2\le s <6$ being $4\le q\le\frac{28}{5}$ and $4\le r<\infty$. An application of~\eqref{mu0}, of the convergence $\mu(\Pi_k\tilde{\bm{\Theta}}_{\ell,k})\to\mu(\tilde{\Theta}_\ell)$ in $L^q(0,T)$ as $N\to\infty$ for any $4\le q\le\frac{28}{5}$ recalled beforehand, and of the generalised H\"{o}lder inequality give:
	\begin{equation*}
		\begin{aligned}
			&\left(\int_{0}^{T}\left\|\mu(\Pi_k\tilde{\bm{\Theta}}_{\ell,k})\psi_{k,\ell}(t)v_{N,\ell}-\mu(\tilde{\Theta}_\ell)u_\ell(t)\right\|_{W^{1,p}(\Omega)}^s\dd t\right)^{1/s}\\
			&\le\mu_2\left(\int_{0}^{T}\|\psi_{k,\ell}(t)v_{N,\ell}-u_\ell(t)\|_{W^{1,p}(\Omega)}^s\dd t\right)^{1/s}\\
			&\quad+\left(\int_{0}^{T}\left|\mu(\Pi_k\tilde{\bm{\Theta}}_{\ell,k})-\mu(\tilde{\Theta}_\ell)\right|^q \dd t\right)^{1/q}
			\left(\int_{0}^{T}\|u_\ell(t)\|_{W^{1,p}(\Omega)}^r\right)^{1/r}.
		\end{aligned}
	\end{equation*}
	
	Thanks to~\eqref{mu0}, \eqref{step:14-2}, \eqref{conv-mu-psi}and the boundedness of $\{u_\ell\}_{\ell>0}$ in $L^\infty(0,T;W^{1,p}_0(\Omega))$ established in~\eqref{conv-proc}, the right-hand side of the estimate above tends to zero as $N\to\infty$, showing that
	\begin{equation}
		\label{step:14-3}
		\mu(\Pi_k\tilde{\bm{\Theta}}_{\ell,k})\psi_{k,\ell}v_{N,\ell}\to\mu(\tilde{\Theta}_\ell)u_\ell,\quad\textup{ in }L^s(0,T;W^{1,p}_0(\Omega))\textup{ as }N\to\infty,
	\end{equation}
	for some $2\le s<6$.
	An application of Lemma~\ref{lem:8} and~\eqref{step:14-3} to~\eqref{step:4} gives:
	\begin{equation}
		\label{step:15}
		\begin{aligned}
			&\dfrac{1}{\alpha'}\|u_\ell(T)\|_{L^\alpha(\Omega)}^\alpha +\int_{0}^{T} \mu(\tilde{\Theta}_\ell)\langle g_\ell(t), u_\ell(t) \rangle_{W^{-1,p'}(\Omega), W_0^{1,p}(\Omega)} \dd t\\
			&=\int_{0}^{T} \int_{\Omega} \tilde{a}(t) u_\ell(t) \dd x \dd t+\dfrac{1}{\alpha'}\|u_0\|_{L^\alpha(\Omega)}^\alpha-\dfrac{1}{\ell}\int_{0}^{T}\mu(\tilde{\Theta}_\ell)\|\{u_\ell(t)\}^{-}\|_{L^2(\Omega)}^2\dd x\dd t.
		\end{aligned}
	\end{equation}
	
	Combining~\eqref{step:14} and~\eqref{step:15} gives:
	\begin{equation}
		\label{step:16}
		\begin{aligned}
			&\liminf_{N\to\infty}\int_{0}^{T} \mu(\Pi_k\tilde{\bm{\Theta}}_{\ell,k})\left\langle B_\ell(\Pi_k \bm{u}_{\ell,k}), \Pi_k \bm{u}_{\ell,k} \right\rangle_{W^{-1,p'}(\Omega), W_0^{1,p}(\Omega)}\\
			&\le\int_{0}^{T} \mu(\tilde{\Theta}_\ell)\langle g_\ell(t), u_\ell(t) \rangle_{W^{-1,p'}(\Omega), W_0^{1,p}(\Omega)} \dd t.
		\end{aligned}
	\end{equation}
	
	We now exploit the method of Minty, known as Minty's trick~\cite{Minty1962}. Let $w\in L^p(0,T;W_0^{1,p}(\Omega))$. An application of~\eqref{mu0}, of the (strict) monotonicity of the operator $B_\ell$ defined in~\eqref{Bkappa}, and~\eqref{step:16} give:
	\begin{equation}
		\label{step:17}
		\begin{aligned}
			&\int_{0}^{T} \mu(\tilde{\Theta}_\ell)\langle g_\ell-B_\ell(w), u_\ell-w\rangle_{W^{-1,p'}(\Omega), W_0^{1,p}(\Omega)} \dd t\\
			&\ge \mu_1\liminf_{N\to\infty} \int_{0}^{T} \langle B_\ell(\Pi_k \bm{u}_{\ell,k})-B_\ell(w), \Pi_k \bm{u}_{\ell,k}-w \rangle_{W^{-1,p'}(\Omega), W_0^{1,p}(\Omega)} \dd t \ge 0.
		\end{aligned}
	\end{equation}
	
	Let $\lambda>0$ and specialise $w=u_\ell -\lambda v$ in~\eqref{step:17}, where $v$ is arbitrarily chosen in $L^p(0,T;W_0^{1,p}(\Omega))$. We obtain that:
	\begin{equation}
		\label{step:18}
		\int_{0}^{T} \mu(\tilde{\Theta}_\ell)\langle g_\ell-B_\ell(u_\ell -\lambda v), \lambda v \rangle_{W^{-1,p'}(\Omega), W_0^{1,p}(\Omega)} \dd t \ge 0.
	\end{equation}
	
	Dividing~\eqref{step:18} by $\lambda >0$, letting $\lambda \to 0^+$ and exploiting the hemi-continuity of $B_\ell$ gives:
	\begin{equation*}
		\label{step:19}
		\int_{0}^{T} \mu(\tilde{\Theta}_\ell)\langle g_\ell-B_\ell(u_\ell), v \rangle_{W^{-1,p'}(\Omega), W_0^{1,p}(\Omega)} \dd t \ge 0.
	\end{equation*}
	
	Specialising $v=\psi(t)\tilde{v}$, with $\psi\in\mathcal{D}(0,T)$ and $\tilde{v}\in\mathcal{D}(\Omega)$, we obtain that:
	\begin{equation*}
		\mu(\tilde{\Theta}_\ell)\langle g_\ell-B_\ell(u_\ell), \tilde{v} \rangle_{W^{-1,p'}(\Omega), W_0^{1,p}(\Omega)}=0,\quad\textup{ for a.a. }t\in (0,T).
	\end{equation*}
	
	Since $\mu_1\le\mu(\tilde{\Theta}_\ell)<\mu_2$ independently of $\ell$ by~\eqref{mu0} and since $\tilde{v}$ is chosen arbitrarily in $\mathcal{D}(\Omega)$, the latter implies:
	\begin{equation*}
		\label{step:20}
		g_\ell = B_\ell(u_\ell) \in L^\infty(0,T;W^{-1,p'}(\Omega)).
	\end{equation*}
	
	$(iv)$ \emph{Preliminary remarks for the recovery of the weak formulation~\eqref{penalty-thetatilde}}. Let $\Xi\in\mathcal{D}(\Omega\times (0,L))$ and let $\psi\in\mathcal{C}^1([0,T])$. For each $0 \le n \le N-1$, multiply~\eqref{FD:temp} by $\{\psi(nk)\Xi\}$, getting:
	\begin{equation}
		\label{step:22}
		\begin{aligned}
			&\rho c\int_{\Omega\times (0,L)}\dfrac{\tilde{\Theta}_{\ell,k}^{n+1}-\tilde{\Theta}_{\ell,k}^{n}}{k} \psi(nk)\Xi\dd z\dd x
			+\rho c\int_{\Omega\times (0,L)}(\bm{U}(0),v_z(0))\cdot\nabla\tilde{\Theta}_{\ell,k}^{n+1} (\psi(nk)\Xi)\dd z\dd x\\
			&\quad+\kappa\int_{\Omega\times (0,L)}\nabla\tilde{\Theta}_{\ell,k}^{n+1}\cdot\nabla(\psi(nk)\Xi)\dd z\dd x
			-\dfrac{1}{\ell}\int_{\Omega\times (0,L)}\{\tilde{\Theta}_{\ell,k}^{n+1}\}^{-}(\psi(nk)\Xi)\dd z\dd x\\
			&\quad+\dfrac{1}{\ell}\int_{\Omega\times (0,L)}\{\tilde{\Theta}_{\ell,k}^{n+1}-\Theta_{\textup{m}}\}^{+}(\psi(nk)\Xi)\dd z\dd x\\
			&=\int_{\Omega\times (0,L)}h_{\ell,k}^{n}(\psi(nk)\Xi)\dd z\dd x-\int_{\Omega}q_{\textup{geo}}^\perp(\psi(nk)\Xi)\dd x.
		\end{aligned}
	\end{equation}
	
	Define the mapping $\psi_k:[0,T]\to\mathbb{R}$ in a way that $\psi_k(t):=\psi(nk)$, for all $nk\le t< (n+1)k$.
	Multiplying~\eqref{step:22} by $k$ and summing over $n=0,\dots, N-1$ gives:
	\begin{equation}
		\label{step:23}
		\begin{aligned}
			&\rho c\int_{0}^{T}\int_{\Omega\times (0,L)}D_k(\Pi_k\tilde{\bm{\Theta}}_{\ell,k})\psi_k(t)\Xi\dd z\dd x\dd t\\
			&\quad+\rho c\int_{0}^{T}\int_{\Omega\times (0,L)}(\bm{U}(0),v_z(0))\cdot\nabla(\Pi_k\tilde{\bm{\Theta}}_{\ell,k}) (\psi_k(t)\Xi) \dd z\dd x\dd t\\
			&\quad+\kappa\int_{0}^{T}\int_{\Omega\times (0,L)}\nabla(\Pi_k\tilde{\bm{\Theta}}_{\ell,k})\cdot\nabla(\psi_k(t)\Xi)\dd z\dd x\dd t\\
			&\quad-\dfrac{1}{\ell}\int_{0}^{T}\int_{\Omega\times (0,L)}\{\Pi_k\tilde{\bm{\Theta}}_{\ell,k}\}^{-} (\psi_k(t)\Xi)\dd z \dd x \dd t\\
			&\quad+\dfrac{1}{\ell}\int_{0}^{T}\int_{\Omega\times (0,L)}\{\Pi_k\tilde{\bm{\Theta}}_{\ell,k}-\Theta_{\textup{m}}\}^{+}(\psi_k(t)\Xi)\dd z\dd x\dd t\\
			&=\int_{0}^{T}\int_{\Omega\times (0,L)}\Lambda_k\bm{h}_{\ell,k}(\psi_k(t)\Xi)\dd z\dd x\dd t-\int_{0}^{T}\int_{\Omega}q_{\textup{geo}}^\perp(\psi_k(t)\Xi)\dd x\dd t.
		\end{aligned}
	\end{equation}
	
	An application of Theorem~\ref{rhs-conv}, the properties of finite difference quotients (viz., e.g., Section~5.8.2 in~\cite{Evans2010}), and the affinity and continuity of the $-\Delta_b:L^\infty(0,T;V)\to L^\infty(0,T;V')$ (both inherited by the affinity and continuity of the original \emph{elliptic} operator $-\Delta_b:V\to V'$) allow us to pass to the limit as $N\to\infty$ in~\eqref{step:23}, so as to obtain
	\begin{equation}
		\label{step:24}
		\begin{aligned}
			&\rho c\int_{0}^{T}\int_{\Omega\times (0,L)}\dfrac{\dd\tilde{\Theta}_\ell}{\dd t}(\psi(t)\Xi) \dd z\dd x\dd t
			+\rho c\int_{0}^{T}\int_{\Omega\times (0,L)}(\bm{U}(0),v_z(0))\cdot\nabla\tilde{\Theta}_\ell (\psi(t)\Xi)\dd z\dd x\dd t\\
			&\quad+\kappa\int_{0}^{T}\langle-\Delta_b(\tilde{\Theta}_\ell),\Xi\rangle_{V',V}\psi(t)\dd t
			-\dfrac{1}{\ell}\int_{0}^{T}\int_{\Omega\times (0,L)}\{\tilde{\Theta}_\ell\}^{-}(\psi(t)\Xi)\dd z\dd x\dd t\\
			&\quad+\dfrac{1}{\ell}\int_{0}^{T}\int_{\Omega\times (0,L)}\{\tilde{\Theta}_\ell-\Theta_{\textup{m}}\}^{+}(\psi(t)\Xi)\dd z\dd x\dd t\\
			&=\int_{0}^{T}\int_{\Omega\times (0,L)}h_\ell (\psi(t)\Xi)\dd z\dd x\dd t,
		\end{aligned}
	\end{equation}
	on the one hand.
	On the other hand, observe that the first term in~\eqref{step:22} can be written as:
	\begin{equation*}
		\begin{aligned}
			&\dfrac{\rho c}{k}\int_{\Omega\times (0,L)}\tilde{\Theta}_{\ell,k}^{N}\psi(T-k)\Xi\dd z\dd x-\dfrac{\rho c}{k}\int_{\Omega\times (0,L)}\tilde{\Theta}_0\psi(0)\Xi\dd z\dd x\\
			&\quad-\rho c\sum_{n=0}^{N-2}\int_{\Omega\times (0,L)}\tilde{\Theta}_{\ell,k}^{n+1}\left[\dfrac{\psi((n+1)k)-\psi(nk)}{k}\right]\Xi\dd z\dd x.
		\end{aligned}
	\end{equation*}
	
	Multiplying~\eqref{step:22} by $k$, summing over $n=0,\dots, N-1$, and applying the Taylor--MacLaurin formula (cf., e.g., Theorem~9.10-1$(c)$ in~\cite{Ciarlet2025}) to the latter leads to:
	\begin{equation}
		\label{step:25}
		\begin{aligned}
			&\rho c\int_{\Omega\times (0,L)}\tilde{\Theta}_{\ell,k}^{N}\psi(T-k)\Xi\dd z\dd x-\rho c\int_{\Omega\times (0,L)}\tilde{\Theta}_0\psi(0)\Xi\dd z\dd x\\
			&\quad-\rho c\sum_{n=0}^{N-2}\int_{nk}^{(n+1)k}\int_{\Omega\times (0,L)}\Pi_k\tilde{\bm{\Theta}}_{\ell,k}(t)\psi'(\theta n k+(1-\theta)(1+n)k)\Xi\dd z\dd x\dd t\\
			&\quad+\rho c\sum_{n=0}^{N-1}\int_{nk}^{(n+1)k}\int_{\Omega\times (0,L)}(\bm{U}(0),v_z(0))\cdot\nabla(\Pi_k\tilde{\bm{\Theta}}_{\ell,k}(t)) (\psi_k(t)\Xi)\dd z\dd x \dd t\\
			&+\kappa\sum_{n=0}^{N-1}\int_{nk}^{(n+1)k}\int_{\Omega\times (0,L)}\nabla(\Pi_k\tilde{\bm{\Theta}}_{\ell,k}(t))\cdot\nabla(\psi_k(t)\Xi)\dd z\dd x \dd t\\
			&\quad-\dfrac{1}{\ell}\sum_{n=0}^{N-1}\int_{nk}^{(n+1)k}\int_{\Omega\times (0,L)}\{\Pi_k\tilde{\bm{\Theta}}_{\ell,k}(t)\}^{-}(\psi_k(t)\Xi)\dd z\dd x \dd t\\
			&\quad+\dfrac{1}{\ell}\sum_{n=0}^{N-1}\int_{nk}^{(n+1)k}\int_{\Omega\times (0,L)}\{\Pi_k\tilde{\bm{\Theta}}_{\ell,k}(t)-\Theta_{\textup{m}}\}^{+}\psi_k(t)\Xi\dd z\dd x \dd t\\
			&=\sum_{n=0}^{N-1}\int_{nk}^{(n+1)k}\int_{\Omega\times (0,L)}\Lambda_k\bm{h}_{\ell,k}(t)(\psi_k(t)\Xi)\dd z\dd x \dd t-\sum_{n=0}^{N-1}\int_{nk}^{(n+1)k}\int_{\Omega}q_{\textup{geo}}^\perp (\psi_k(t)\Xi)\dd x\dd t.
		\end{aligned}
	\end{equation}
	
	An application of Theorem~\ref{rhs-conv} and~\eqref{conv-proc} allow us to pass to the limit as $N\to\infty$ in~\eqref{step:25}, so as to obtain:
	\begin{equation}
		\label{step:26}
		\begin{aligned}
			&\rho c\int_{\Omega\times (0,L)}\tilde{\tau}_\ell\psi(T)\Xi \dd z\dd x-\rho c\int_{\Omega\times (0,L)}\tilde{\Theta}_0\psi(0)\Xi\dd z\dd x
			-\rho c\int_{0}^{T}\int_{\Omega\times (0,L)}\tilde{\Theta}_\ell \Xi \dfrac{\dd\psi}{\dd t}\dd z\dd x\dd t\\
			&\quad+\rho c\int_{0}^{T}\int_{\Omega\times (0,L)}(\bm{U}(0),v_z(0))\cdot\nabla\tilde{\Theta}_\ell(\psi(t)\Xi)\dd z\dd x\dd t
			+\kappa\int_{0}^{T}\langle-\Delta_b(\tilde{\Theta}_\ell(t)),\Xi\rangle_{V',V}\psi(t)\dd t\\
			&\quad-\dfrac{1}{\ell}\int_{0}^{T}\int_{\Omega\times (0,L)}\{\tilde{\Theta}_\ell\}^{-}(\psi(t)\Xi)\dd z\dd x\dd t
			+\dfrac{1}{\ell}\int_{0}^{T}\int_{\Omega\times (0,L)}\{\tilde{\Theta}_\ell-\Theta_{\textup{m}}\}^{+}(\psi(t)\Xi)\dd z\dd x\dd t\\
			&=\int_{0}^{T}\int_{\Omega\times (0,L)}h_\ell(\psi(t)\Xi)\dd z\dd x\dd t.
		\end{aligned}
	\end{equation}
	
	Taking an integration by parts in time of~\eqref{step:24} leads to:
	\begin{equation}
		\label{step:27}
		\begin{aligned}
			&\rho c\int_{\Omega\times (0,L)}\tilde{\Theta}_\ell(T)\psi(T)\Xi\dd z\dd x-\rho c\int_{\Omega\times (0,L)}\tilde{\Theta}_\ell(0)\psi(0)\Xi\dd z\dd x\\
			&\quad-\rho c\int_{0}^{T}\int_{\Omega\times (0,L)}\tilde{\Theta}_\ell\left(\dfrac{\dd\psi}{\dd t}\Xi\right) \dd z\dd x\dd t\\
			&\quad+\rho c\int_{0}^{T}\int_{\Omega\times (0,L)}(\bm{U}(0),v_z(0))\cdot\nabla\tilde{\Theta}_\ell (\psi(t)\Xi)\dd z\dd x\dd t\\
			&\quad+\kappa\int_{0}^{T}\langle-\Delta_b\tilde{\Theta}_\ell,\Xi\rangle_{V',V}\psi(t)\dd t
			-\dfrac{1}{\ell}\int_{0}^{T}\int_{\Omega\times (0,L)}\{\tilde{\Theta}_\ell\}^{-}(\psi(t)\Xi)\dd z\dd x\dd t\\
			&\quad+\dfrac{1}{\ell}\int_{0}^{T}\int_{\Omega\times (0,L)}\{\tilde{\Theta}_\ell-\Theta_{\textup{m}}\}^{+}(\psi(t)\Xi)\dd z\dd x\dd t\\
			&=\int_{0}^{T}\int_{\Omega\times (0,L)}h_\ell (\psi(t)\Xi)\dd z\dd x\dd t.
		\end{aligned}
	\end{equation}
	
	Therefore, subtracting~\eqref{step:27} from~\eqref{step:26} gives:
	\begin{equation}
		\label{step:28}
		\int_{\Omega\times (0,L)}(\tilde{\tau}_\ell-\tilde{\Theta}_\ell(T))\psi(T)\Xi\dd z\dd x=\int_{\Omega\times (0,L)}(\tilde{\Theta}_0-\tilde{\Theta}_\ell(0))\psi(0)\Xi\dd z\dd x.
	\end{equation}
	
	Specialising $\psi$ in~\eqref{step:28} first in a way that $\psi(T)=0$ and, second, in a way that $\psi(0)=0$ leads to, respectively:
	\begin{equation*}
		\begin{aligned}
			\tilde{\Theta}_\ell(0)&=\tilde{\Theta}_0,\\
			\tilde{\Theta}_\ell(T)&=\tilde{\tau}_\ell.
		\end{aligned}
	\end{equation*}
	
	$(v)$ \emph{Completion of the proof}. Putting together steps $(i)$--$(iv)$ shows that Problem~\ref{Pkappa} admits at least one solution $(u_\ell,\tilde{\Theta}_\ell)$, as it was to be proved.
\end{proof}

\section{Recovery of the limit model~\eqref{weak-u}--\eqref{theta:ic}}
\label{Sec:4}

The next step, which constitutes the final step of the existence result we want to establish, is the passage to the limit as the penalty parameter approaches zero, as well as the recovery of a sound \emph{concept of solution} for the model governing the evolution of time-dependent thermally-coupled shallow ice sheets. We mention that the strategy proposed here solves a mistake in the proof proposed in~\cite{PT23} for the recovery for the limit model. While the mistake in~\cite{PT23} affects the identification of the concept of solution, the idea itself is valid and the correct result can still be reached by arguing in a different fashion.

The estimates~\eqref{est:1}--\eqref{est:9}, the estimates~\eqref{penalty-est-1}, \eqref{penalty-est-2}, \eqref{penalty-est-3} and Theorem~\ref{thm:3} imply that there exists a constant $C>0$ independent of $\ell$ such that:
\begin{equation}
	\label{boundukappa}
	\begin{aligned}
		\|u_\ell\|_{L^\infty(0,T;W_0^{1,p}(\Omega))} &\le C,\\
		\left\||u_\ell|^\frac{\alpha-2}{2}u_\ell\right\|_{L^\infty(0,T;L^2(\Omega))} &\le C,\\
		\left\|\dfrac{\dd }{\dd t}(|u_\ell|^\frac{\alpha-2}{2} u_\ell)\right\|_{L^2(0,T;L^2(\Omega))} &\le C,\\
		\left\||u_\ell|^{\alpha-2} u_\ell\right\|_{L^\infty(0,T;L^{\alpha'}(\Omega))} &\le C,\\
		\|u_\ell(T)\|_{L^\alpha(\Omega)}&\le C,\\
		\|\{u_\ell\}^{-}\|_{L^2(0,T;L^2(\Omega))}&\le C\sqrt{\ell},\\
		\|\tilde{\Theta}_\ell\|_{L^\infty(0,T;V)} &\le C,\\
		\left\|\dfrac{\dd\tilde{\Theta}_\ell}{\dd t}\right\|_{L^2(0,T;L^2(\Omega\times (0,L)))} &\le C,\\
		\|\{\tilde{\Theta}_\ell\}^{-}\|_{L^2(0,T;L^2(\Omega\times (0,L)))}&\le C\sqrt{\ell},\\
		\|\{\tilde{\Theta}_\ell-\Theta_{\textup{m}}\}^{+}\|_{L^2(0,T;L^2(\Omega\times (0,L)))}&\le C\sqrt{\ell}.
	\end{aligned}
\end{equation}

By the Banach--Alaoglu--Bourbaki theorem (cf., e.g., Theorem~3.6 of~\cite{Brez11}) we infer that, up to passing to subsequences $\{u_{\ell_n}\}_{n=1}^\infty$ and $\{\tilde{\Theta}_{\ell_n}\}_{n=1}^\infty$, with $\ell_n\to 0^+$ as $n\to\infty$, it results:
\begin{equation}
	\label{conv-proc-kappa}
	\begin{aligned}
		u_{\ell_n} \wsc u, &\textup{ in } L^\infty(0,T;W_0^{1,p}(\Omega))\textup{ as }n\to\infty,\\
		|u_{\ell_n}|^\frac{\alpha-2}{2} u_{\ell_n} \wsc v, &\textup{ in } L^\infty(0,T;L^2(\Omega))\textup{ as }n\to\infty,\\
		\dfrac{\dd }{\dd t}\left(|u_{\ell_n}|^\frac{\alpha-2}{2} u_{\ell_n}\right) \rightharpoonup \dfrac{\dd v}{\dd t}, &\textup{ in } L^2(0,T;L^2(\Omega))\textup{ as }n\to\infty,\\
		|u_{\ell_n}|^{\alpha-2} u_{\ell_n} \wsc w, &\textup{ in } L^\infty(0,T;L^{\alpha'}(\Omega))\textup{ as }n\to\infty,\\
		u_{\ell_n}(T)\rightharpoonup\chi,&\textup{ in }L^\alpha(\Omega)\textup{ as }n\to\infty,\\
		\{u_{\ell_n}\}^{-} \to 0,&\textup{ in }L^2(0,T;L^2(\Omega))\textup{ as }n\to\infty,\\
		\tilde{\Theta}_{\ell_n}\wsc\tilde{\Theta},&\textup{ in }L^\infty(0,T;V)\textup{ as }n\to\infty,\\
		\dfrac{\dd\tilde{\Theta}_{\ell_n}}{\dd t}\rightharpoonup\dfrac{\dd\tilde{\Theta}}{\dd t},&\textup{ in }L^2(0,T;L^2(\Omega\times(0,L)))\textup{ as }n\to\infty,\\
		\{\tilde{\Theta}_{\ell_n}\}^{-}\to 0,&\textup{ in }L^2(0,T;L^2(\Omega\times(0,L)))\textup{ as }n\to\infty,\\
		\{\tilde{\Theta}_{\ell_n}-\Theta_{\textup{m}}\}^{+}\to 0,&\textup{ in }L^2(0,T;L^2(\Omega\times(0,L)))\textup{ as }n\to\infty.
	\end{aligned}
\end{equation}

The first convergence in~\eqref{conv-proc-kappa} implies that:
\begin{equation*}
	u_{\ell_n} \rightharpoonup u,\quad\textup{ in }L^2(0,T;L^2(\Omega)) \simeq L^2((0,T)\times \Omega) \textup{ as }n\to\infty.
\end{equation*}

The continuity and the monotonicity of the negative part established in Lemma~\ref{lem:1}$(a)$ allow us to apply Theorem~12.5-2$(a)$ of~\cite{Ciarlet2025}, getting
\begin{equation}
	\label{sign}
	\{u\}^{-} =0 \quad \textup{ in } L^2(0,T;L^2(\Omega)),
\end{equation}
which means that $u(t) \ge 0$ a.e. in $\Omega$, for a.e. $t \in (0,T)$.

Using the same compactness argument as in Lemma~\ref{Dub:appl-1}, an application of Lemma~\ref{Dub}$(b)$ with $A_0:=L^\frac{2p}{\alpha}(\Omega)$, $A_1:=L^2(\Omega)$, $q_0=\infty$ and $q_1=2$ gives that:
\begin{equation}
	\label{Dub:strong}
	\left\{|u_{\ell_n}|^\frac{\alpha-2}{2}u_{\ell_n}\right\}_{n=1}^\infty \textup{ strongly converges in } \mathcal{C}^0([0,T];L^\frac{2p}{\alpha}(\Omega))\textup{ as }n\to\infty.
\end{equation}

The seventh and eighth convergences in~\eqref{conv-proc-kappa} put us in position to apply the classical Aubin--Lions--Simon compactness theorem, so as to infer that, up to passing to an ulterior subsequence
\begin{equation}
	\label{ALS-temp}
	\tilde{\Theta}_{\ell_n}\to\tilde{\Theta},\quad\textup{ in }\mathcal{C}^0([0,T];L^{q}(\Omega\times(0,L))),
\end{equation}
as $n\to\infty$, for all $2\le q<6$.
Since, by Lemma~\ref{A-mu-properties}$(c)$, the operator $\mathcal{A}$ is continuous from $L^q(0,T;L^q(\Omega\times(0,L)))$ to itself in the particular case where $\frac{10}{3}\le q\le\frac{28}{5}$, we obtain that an application of~\eqref{ALS-temp} gives:
\begin{equation}
	\label{A-conv}
	\mathcal{A}(\tilde{\Theta}_{\ell_n})\to\mathcal{A}(\tilde{\Theta}),\quad\textup{ in }L^q(0,T;L^q(\Omega\times(0,L))) \textup{ as }n\to\infty.
\end{equation}

In light of~\eqref{Dub:strong} and~\eqref{A-conv}, we can systematically apply, in this order, Lemma~\ref{lem:5:improved}, Lemma~\ref{lem:11}, \eqref{conv-rhs-temp} and Theorem~\ref{rhs-conv} (applied to $\ell$ instead of $N$) to obtain that:
\begin{equation}
	\label{key}
	h_\ell\to h:=\mathcal{A}(\tilde{\Theta})\left(\left\{u^\frac{p-1}{2p}-z\right\}^{+}\right)^p,\quad\textup{ in }L^\sigma(0,T;L^\sigma(\Omega\times (0,L))) \textup{ as }n\to\infty,
\end{equation}
where $\sigma$ is the real number defined in Theorem~\ref{rhs-conv}.

The same monotonicity argument as in Theorem~\ref{thm:3} and the third convergence in~\eqref{conv-proc-kappa} in turn imply that the limit $v$ of the sequence in~\eqref{Dub:strong} takes the following form:
\begin{equation*}
	v=|u|^\frac{\alpha-2}{2} u \in H^1(0,T;L^2(\Omega)).
\end{equation*}

For each $v \in L^2(0,T;L^2(\Omega))$ we have
\begin{equation*}
	\begin{aligned}
		&\int_{0}^{T} \int_{\Omega} v (|u_{\ell_n}|-u)\dd x \dd t
		=\int_{0}^{T} \int_{\Omega} v (\{u_{\ell_n}\}^{+}+\{u_{\ell_n}\}^{-}-u)\dd x \dd t\\
		&=\int_{0}^{T} \int_{\Omega} v \left((\{u_{\ell_n}\}^{+}-\{u_{\ell_n}\}^{-})+2\{u_{\ell_n}\}^{-}-u\right)\dd x \dd t\\
		&=\int_{0}^{T} \int_{\Omega} v (u_{\ell_n} -u)\dd x \dd t+2\int_{0}^{T} \int_{\Omega} v\{u_{\ell_n}\}^{-}\dd x \dd t\to 0,
	\end{aligned}
\end{equation*}
as $n\to\infty$ since the first term converges to zero by the first convergence in~\eqref{conv-proc-kappa} and the second term converges to zero thanks to~\eqref{sign}. In conclusion, we have established the following convergence:
\begin{equation}
	\label{conv-abs}
	|u_{\ell_n}| \rightharpoonup u,\quad\textup{ in }L^2(0,T;L^2(\Omega)) \textup{ as }n\to\infty.
\end{equation}

In light of~\eqref{Dub:strong} and~\eqref{conv-abs}, we are in position to apply Lemma~\ref{lem:5} so as to obtain
\begin{equation*}
	\left\{|u_{\ell_n}|\right\}_{n=1}^\infty \textup{ strongly converges in } \mathcal{C}^0([0,T];L^\alpha(\Omega))\textup{ as }n\to\infty,
\end{equation*}
and, so, it also strongly converges in $L^2(0,T;L^\alpha(\Omega))$. Combining this with the convergences~\eqref{sign} and~\eqref{conv-abs} gives that:
\begin{equation}
	\label{cont}
	|u_{\ell_n}| \to u,\quad\textup{ in } \mathcal{C}^0([0,T];L^\alpha(\Omega))\textup{ as }n\to\infty.
\end{equation}

Define the linear and continuous operator $L_0:\mathcal{C}^0([0,T];L^\alpha(\Omega)) \to L^\alpha(\Omega)$ by:
\begin{equation*}
	L_0(v):=v(0),\quad\textup{ for all }v \in \mathcal{C}^0([0,T];L^\alpha(\Omega)).
\end{equation*}

By~\eqref{cont} and the continuity of $L_0$, we have that $|u_{\ell_n}(0)| \to u(0)$ in $L^\alpha(\Omega)$ as $n\to\infty$. However, since $u_{\ell_n}(0)=u_0\in K_{\textup{surf}}$ for all $n\in\mathbb{N}$, we immediately deduce that $u(0)=u_0 \in K_{\textup{surf}}$ and so that the weak-star limit $u$ in~\eqref{conv-proc-kappa} satisfies the expected initial condition.

An application of the first and fourth convergences in~\eqref{conv-proc-kappa} give $u_{\ell_n}\rightharpoonup u$ in $L^2(0,T;L^\alpha(\Omega))$ as $n\to\infty$ and $(|u_{\ell_n}|^{\alpha-2}u_{\ell_n})\rightharpoonup w$ in $L^2(0,T; L^{\alpha'}(\Omega))$ as $n\to\infty$, respectively. Observing that an application of~\eqref{cont} gives $(|u|^{\alpha-2}u)\in L^2(0,T;L^{\alpha'}(\Omega))$ as well as
\begin{equation*}
	\int_{0}^{T}\int_{\Omega}\left(|u_{\ell_n}|^{\alpha-2}u_{\ell_n}\right) u_{\ell_n}\dd x\dd t=\int_{0}^{T}\int_{\Omega}|u_{\ell_n}|^\alpha\dd x\dd t\to\int_{0}^{T}\int_{\Omega}|u|^\alpha\dd x\dd t=\int_{0}^{T}\int_{\Omega}\left(|u|^{\alpha-2}u\right) u\dd x\dd t,
\end{equation*}
as $n\to\infty$, we are in position to apply Theorem~12.5-2$(a)$ in~\cite{Ciarlet2025} and conclude that:
\begin{equation}
	\label{eq:10}
	w=|u|^{\alpha-2}u,\quad\textup{ in }L^2(0,T;L^{\alpha'}(\Omega)).
\end{equation}

Thanks to Lemma~\ref{lem:3}, Lemma~\ref{lem:5} and Lemma~\ref{lem:8}, it results:
\begin{equation}
	\label{eq:11}
	\begin{aligned}
		&\int_{0}^{T} \left\langle\dfrac{\dd}{\dd t}\left(|u_{\ell_n}(t)|^{\alpha-2} u_{\ell_n}(t)\right),u_{\ell_n}(t)\right\rangle_{W^{-1,p'}(\Omega), W_0^{1,p}(\Omega)} \dd t
		=\dfrac{\|u_{\ell_n}(T)\|_{L^\alpha(\Omega)}^\alpha}{\alpha'}-\dfrac{\|u_0\|_{L^\alpha(\Omega)}^\alpha}{\alpha'}\\
		&\to \dfrac{\|u(T)\|_{L^\alpha(\Omega)}^\alpha}{\alpha'}-\dfrac{\|u_0\|_{L^\alpha(\Omega)}^\alpha}{\alpha'},\quad\textup{ as }n\to\infty.
	\end{aligned}
\end{equation}

Define the operator $B:L^1(0,T;W_0^{1,p}(\Omega))\to L^\infty(0,T;W^{-1,p'}(\Omega))$ as follows:
\begin{equation*}
	\begin{aligned}
		\langle B(u) , v \rangle_{L^\infty(0,T;W^{-1,p'}(\Omega)),L^1(0,T;W_0^{1,p}(\Omega))}:=\int_{0}^{T} \int_{\Omega} \mu(\tilde{\Theta}(t))|\nabla(u(t))|^{p-2} \nabla(u(t)) \cdot \nabla(v(t)) \dd x \dd t,
	\end{aligned}
\end{equation*}
for all $u,v \in L^1(0,T;W_0^{1,p}(\Omega))$.
For any $v \in L^1(0,T;W^{1,p}_0(\Omega))$, an application of~\eqref{mu0} and H\"older's inequality gives:
\begin{equation*}
	\begin{aligned}
		&\left|\int_{0}^{T} \int_{\Omega} \mu(\tilde{\Theta}) |\nabla (u_{\ell_n}(t))|^{p-2} \nabla(u_{\ell_n}(t)) \cdot\nabla(v(t)) \dd x \dd t\right|\\
		&=\left|\int_{0}^{T} \sum_{i=1}^2\int_{\Omega} \mu(\tilde{\Theta}) |\nabla (u_{\ell_n}(t))|^{p-2} \partial_i(u_{\ell_n}(t)) \partial_i(v(t)) \dd x \dd t\right|\\
		& \le \mu_2\Bigg|\sum_{i=1}^2\int_{0}^{T} \left(\int_{\Omega}\left| |\nabla(u_{\ell_n}(t))|^{p-2} \partial_i(u_{\ell_n}(t))\right|^{p/(p-1)} \dd x\right)^{(p-1)/p}\\
		&\qquad\qquad\times\left(\int_{\Omega} |\partial_i(v(t))|^p \dd x\right)^{1/p} \dd t\Bigg|\\
		& \le \mu_2\Bigg|\sum_{i=1}^2\int_{0}^{T} \left(\int_{\Omega}|\nabla(u_{\ell_n}(t))|^{p(p-2)/(p-1)} |\partial_i(u_{\ell_n}(t))|^{p/(p-1)} \dd x\right)^{(p-1)/p}\\
		&\qquad\qquad\times\left(\int_{\Omega} |\partial_i(v(t))|^p \dd x\right)^{1/p} \dd t\Bigg|\\
		&\le 2\mu_2\left|\int_{0}^{T} \|\nabla(u_{\ell_n}(t))\|_{\bm{L}^p(\Omega)}^{p-1} \|\nabla(v(t))\|_{\bm{L}^p(\Omega)} \dd t\right|\\
		&\le 2\mu_2 \left(\sup_{t\in (0,T)} \|u_{\ell_n}(t)\|_{W^{1,p}(\Omega)}\right)^{p-1} \|v\|_{L^1(0,T;W_0^{1,p}(\Omega))}.
	\end{aligned}
\end{equation*}

This estimate shows that the sequence $\{B(u_{\ell_n})\}_{n=1}^\infty$ is bounded in $L^\infty(0,T;W^{-1,p'}(\Omega))$ independently of $n$. Since $\mu$ is bounded below by $\mu_1>0$ (viz.~\eqref{mu0}), we obtain that the sequence $\{-\nabla\cdot(|\nabla u_{\ell_n}|^{p-2}\nabla u_{\ell_n})\}_{n=1}^\infty$ is bounded in $L^\infty(0,T;W^{-1,p'}(\Omega))$ independently of $n$. A subsequent application of the Banach--Alaoglu--Bourbaki theorem (cf., e.g., Theorem~3.6 of~\cite{Brez11}) ensures the existence of an element $G \in L^\infty(0,T;W^{-1,p'}(\Omega))$ such that:
\begin{equation}
	\label{eq:13-bis}
	-\nabla\cdot\left(|\nabla u_{\ell_n}|^{p-2}\nabla u_{\ell_n}\right)\wsc G,\quad\textup{ in }L^\infty(0,T;W^{-1,p'}(\Omega)) \textup{ as }n\to\infty.
\end{equation}

The continuity of $\mu$ asserted in Lemma~\ref{A-mu-properties}$(e)$, the boundedness of the sequence $\{u_{\ell_n}\}_{n=1}^\infty$ in $L^\infty(0,T;W^{1,p}_0(\Omega))$ stated in the first line of~\eqref{conv-proc-kappa}, and the sequential weak lower semicontinuity of the convex functional $w\in L^p(0,T;W^{1,p}(\Omega))\mapsto\int_{0}^{T}\mu(\tilde{\Theta})\int_{\Omega}|\nabla(w(t))|^p\dd x\dd t$ give:
\begin{equation}
	\label{eq:16}
	\begin{aligned}
		&\limsup_{n\to\infty}\left(-\int_{0}^{T} \int_{\Omega} \mu(\tilde{\Theta}_{\ell_n}) |\nabla (u_{\ell_n}(t))|^p \dd x \dd t\right)
		=-\liminf_{n\to\infty}\int_{0}^{T} \int_{\Omega} \mu(\tilde{\Theta}) |\nabla (u_{\ell_n}(t))|^p \dd x \dd t\\
		&\le -\int_{0}^{T}\int_{\Omega} \mu(\tilde{\Theta}) |\nabla(u(t))|^p \dd x \dd t.
	\end{aligned}
\end{equation}

Thanks to Lemma~\ref{A-mu-properties}$(e)$, the first bound in~\eqref{boundukappa}, H\"{o}lder's inequality, and~\eqref{eq:13-bis} we obtain that, up to passing to a suitable subsequence, the following convergence holds as $n\to\infty$:
\begin{equation}
	\label{c:1}
	\begin{aligned}
		&\int_{0}^{T}\mu(\tilde{\Theta}_{\ell_n})\int_{\Omega}|\nabla(u_{\ell_n}(t))|^{p-2}\nabla(u_{\ell_n}(t))\cdot\nabla(v(t))\dd x\dd t\\
		&=\int_{0}^{T}\left(\mu(\tilde{\Theta}_{\ell_n})-\mu(\tilde{\Theta})\right)\int_{\Omega}|\nabla(u_{\ell_n}(t))|^{p-2}\nabla(u_{\ell_n}(t))\cdot\nabla(v(t))\dd x\dd t\\
		&\quad+\int_{0}^{T}\int_{\Omega}|\nabla(u_{\ell_n}(t))|^{p-2}\nabla(u_{\ell_n}(t))\cdot\nabla\left(\mu(\tilde{\Theta})v(t)\right)\dd x\dd t\\
		&\to\int_{0}^{T}\mu(\tilde{\Theta})\langle G(t),v(t)\rangle_{W^{-1,p'}(\Omega), W_0^{1,p}(\Omega)}\dd t.
	\end{aligned}
\end{equation}

Thanks to Lemma~\ref{th:gradBochner}, we infer that the seventh convergence in~\eqref{conv-proc-kappa} implies:
\begin{equation}
	\label{grad-temp-conv}
	\widetilde{\nabla\tilde{\Theta}_{\ell_n}}\rightharpoonup\widetilde{\nabla\tilde{\Theta}},\quad\textup{ in }L^2(0,T;\bm{L}^2(\Omega\times (0,L))) \textup{ as }n\to\infty.
\end{equation}

Let $v \in W^{1,2,2}(0,T;W^{1,p}_0(\Omega);L^\alpha(\Omega))$ be such that $v(t) \ge 0$ in $\overline{\Omega}$ (since $p>2$), multiply equation~\eqref{penalty-u} by $(v(t)-u_{\ell_n}(t))$ in the sense of the duality between $W^{-1,p'}(\Omega)$ and $W^{1,p}_0(\Omega)$. A subsequent integration in $(0,T)$ gives:
\begin{equation}
	\label{preliminary}
	\begin{aligned}
		&\int_{0}^{T} \left\langle \dfrac{\dd }{\dd t}(|u_{\ell_n}|^{\alpha-2} u_{\ell_n}), v(t)-u_{\ell_n}(t) \right\rangle_{W^{-1,p'}(\Omega), W_0^{1,p}(\Omega)} \dd t\\
		&\quad+\int_{0}^{T} \int_{\Omega} \mu(\tilde{\Theta}_{\ell_n}) |\nabla(u_{\ell_n}(t))|^{p-2} \nabla(u_{\ell_n}(t)) \cdot (\nabla(v(t))-\nabla(u_{\ell_n}(t))) \dd x \dd t\\
		&\ge \int_{0}^{T} \int_{\Omega} \tilde{a}(t) (v(t)-u_{\ell_n}(t)) \dd x\dd t,
	\end{aligned}
\end{equation}
where the inequality holds thanks to the monotonicity of $(-\{\cdot\}^{-})$ established in Lemma~\ref{lem:1}$(a)$ and~\eqref{mu0}.

Passing to the $\limsup$ as $n\to\infty$ in~\eqref{preliminary}, an application of~\eqref{eq:10} (note that $(\alpha'-1)p>\alpha'$), \eqref{eq:11} and~\eqref{eq:16} and~\eqref{c:1} give that the weak-star limit $u$ satisfies:
\begin{equation}
	\label{limit-problem-u}
	\begin{aligned}
		&\int_{\Omega}|u(T)|^{\alpha-2}u(T)v(T)-u_0^{\alpha-1}v(0)\dd x-\int_{0}^{T}\int_{\Omega}|u(t)|^{\alpha-2}u(t)\dfrac{\dd v}{\dd t}(t)\dd x\dd t\\
		&\quad-\dfrac{\|u(T)\|_{L^\alpha(\Omega)}^\alpha}{\alpha'}+\dfrac{\|u_0\|_{L^\alpha(\Omega)}^\alpha}{\alpha'}
		+\int_{0}^{T} \mu(\tilde{\Theta})\langle \gls{G}(t),v(t)\rangle_{W^{-1,p'}(\Omega), W_0^{1,p}(\Omega)} \dd t\\
		&\quad-\int_{0}^{T} \mu(\tilde{\Theta})\int_{\Omega} |\nabla(u(t))|^p \dd x \dd t
		\ge \int_{0}^{T} \int_{\Omega} \tilde{a}(t) (v(t)-u(t)) \dd x\dd t.
	\end{aligned}
\end{equation}

Choose $\Xi\in L^\infty(0,T;V)$ such that $0\le\Xi(t)\le\Theta_{\textup{m}}$ for a.a. $t\in (0,T)$ and multiply equation~\eqref{penalty-thetatilde} by $(\Xi(t)-\tilde{\Theta}_{\ell_n}(t))$ in the sense of duality between $V'$ and $V$. A subsequent integration in $(0,T)$ gives:
\begin{equation*}
	\begin{aligned}
		&\rho c\int_{0}^{T}\int_{0}^{L}\int_{\Omega}\left(\dfrac{\dd \tilde{\Theta}_{\ell_n}}{\dd t}(t)+U_1(0)\dfrac{\partial\tilde{\Theta}_{\ell_n}}{\partial x_1}(t)+U_2(0)\dfrac{\partial\tilde{\Theta}_{\ell_n}}{\partial x_2}(t)+v_z(0)\dfrac{\partial\tilde{\Theta}_{\ell_n}}{\partial z}(t)\right) (\Xi(t)-\tilde{\Theta}_{\ell_n}(t)) \dd x \dd z \dd t\\
		&\quad+\kappa\int_{0}^{T}\int_{0}^{L}\int_{\Omega}\nabla\tilde{\Theta}_{\ell_n}(t)\cdot\nabla(\Xi(t)-\tilde{\Theta}_{\ell_n}(t))\dd x \dd z\dd t-\dfrac{1}{\ell}\int_{0}^{T}\int_{0}^{L}\int_{\Omega}\{\tilde{\Theta}_{\ell_n}(t)\}^{-}(\Xi(t)-\tilde{\Theta}_{\ell_n}(t)) \dd x \dd z \dd t\\
		&\quad+\dfrac{1}{\ell}\int_{0}^{T}\int_{0}^{L}\int_{\Omega}\{\tilde{\Theta}_{\ell_n}(t)-\Theta_{\textup{m}}\}^{+}(\Xi(t)-\tilde{\Theta}_{\ell_n}(t)) \dd x \dd z \dd t\\
		&= 2(\rho g C_0)^p\int_{0}^{T}\int_{0}^{L}\int_{\Omega}A(\cdot,\cdot,\tilde{\Theta}_{\ell_n}(t))\left(\left\{|u_{\ell_n}(t)|^{\frac{p-1}{2p}}-z\right\}^{+}\right)^p (\Xi(t)-\tilde{\Theta}_{\ell_n}(t))\dd x \dd z \dd t\\
		&\quad-\int_0^T\int_{\Omega}q_{\textup{geo}}^\perp (\Xi(t)-\tilde{\Theta}_{\ell_n}(t))\dd x\dd t.
	\end{aligned}
\end{equation*}

An application of the seventh and eight convergences in~\eqref{conv-proc-kappa}, \eqref{ALS-temp}, \eqref{key}, \eqref{grad-temp-conv}, the monotonicity of $(-\{\cdot\}^{-})$ and $\{\cdot\}^{+}$ established in Lemma~\ref{lem:1}$(a)$ and Remark~\ref{rem:neg}, and the fact that
\begin{equation*}
	\int_{0}^{T}\int_{0}^{L}\int_{\Omega}|\nabla\tilde{\Theta}(t)|^2\dd x\dd z\dd t\le\liminf_{n\to\infty}\int_{0}^{T}\int_{0}^{L}\int_{\Omega}|\nabla\tilde{\Theta}_{\ell_n}(t)|^2\dd x\dd z\dd t,
\end{equation*}
give that the limit $\tilde{\Theta}$ satisfies:
\begin{equation}
	\label{limit-problem-tildetheta}
	\begin{aligned}
		&\rho c\int_{0}^{T}\int_{0}^{L}\int_{\Omega}\left(\dfrac{\dd \tilde{\Theta}}{\dd t}(t)+U_1(0)\dfrac{\partial\tilde{\Theta}}{\partial x_1}(t)+U_2(0)\dfrac{\partial\tilde{\Theta}}{\partial x_2}(t)+v_z(0)\dfrac{\partial\tilde{\Theta}}{\partial z}(t)\right) (\Xi(t)-\tilde{\Theta}(t)) \dd x \dd z \dd t\\
		&\quad+\kappa\int_{0}^{T}\int_{0}^{L}\int_{\Omega}\nabla\tilde{\Theta}(t)\cdot\nabla(\Xi(t)-\tilde{\Theta}(t))\dd x \dd z \dd t\\
		&\ge 2(\rho g C_0)^p\int_{0}^{T}\int_{0}^{L}\int_{\Omega}A(\cdot,\cdot,\tilde{\Theta}(t))\left(\left\{u(t)^{\frac{p-1}{2p}}-z\right\}^{+}\right)^p (\Xi(t)-\tilde{\Theta}(t))\dd x \dd z \dd t\\
		&\quad-\int_0^T\int_{\Omega}q_{\textup{geo}}^\perp (\Xi(t)-\tilde{\Theta}(t))\dd x\dd t.
	\end{aligned}
\end{equation}

In the end, we are in position to write down the model governing the evolution of the thickness of a shallow ice sheet, and to assert that this model admits at least one solution.
This is the main result of this article, the proof of which has been given above.
\begin{theorem}
	\label{thm:4}
	Let $T>0$, let $\Omega \subset \mathbb{R}^2$ be a Lipschitz domain, let $2.8\le p\le 5$ and let $\alpha$ be defined in~\eqref{alpha}.
	Assume that ($H$\ref{H1})--($H$\ref{H4}) hold.
	Then, the family $\{u_\ell\}_{\ell>0}$ of solutions of Problem~\ref{Pkappa} generated by Theorem~\ref{thm:3} satisfies the following convergences up to passing to a subsequence $\{u_{\ell_n}\}_{n=1}^\infty$, where $\{\ell_n\}_{n=1}^\infty$ is such that $\ell_n\to 0^+$ as $n\to\infty$:
	\begin{equation*}
		\begin{aligned}
			u_{\ell_n} \wsc u, &\textup{ in } L^\infty(0,T;W_0^{1,p}(\Omega))\textup{ as }n\to\infty,\\
			|u_{\ell_n}|^\frac{\alpha-2}{2} u_{\ell_n} \wsc |u|^\frac{\alpha-2}{2} u, &\textup{ in } L^\infty(0,T;L^2(\Omega))\textup{ as }n\to\infty,\\
			\dfrac{\dd }{\dd t}\left(|u_{\ell_n}|^\frac{\alpha-2}{2} u_{\ell_n}\right) \rightharpoonup \dfrac{\dd }{\dd t}\left(|u|^\frac{\alpha-2}{2} u\right), &\textup{ in } L^2(0,T;L^2(\Omega))\textup{ as }n\to\infty,\\
			|u_{\ell_n}|^{\alpha-2} u_{\ell_n} \wsc |u|^{\alpha-2} u, &\textup{ in } L^\infty(0,T;L^{\alpha'}(\Omega))\textup{ as }n\to\infty,\\
			\tilde{\Theta}_{\ell_n}\wsc\tilde{\Theta},&\textup{ in }L^\infty(0,T;V)\textup{ as }n\to\infty,\\
			\dfrac{\dd\tilde{\Theta}_{\ell_n}}{\dd t}\rightharpoonup\dfrac{\dd\tilde{\Theta}}{\dd t},&\textup{ in }L^2(0,T;L^2(\Omega\times(0,L)))\textup{ as }n\to\infty.
		\end{aligned}
	\end{equation*}
	
	Moreover, the limit $(u,\tilde{\Theta})$ is a solution to the following variational problem:
	\begin{customprob}{$\mathcal{P}$}
		\label{P}
		Find $u\in\gls{Ksurf2}:=\{v\in L^\infty(0,T;W_0^{1,p}(\Omega)); v(t)\in K_{\textup{surf}} \textup{ a.e. in } (0,T)\}$ and $\tilde{\Theta}\in\gls{Ktemp2}:=\{\Xi\in L^\infty(0,T;V);0\le\Xi(t)\le\Theta_{\textup{m}} \textup{ a.e. in }(0,T)\}$ such that
		\begin{align*}
			u &\in L^\infty(0,T;W_0^{1,p}(\Omega)),\\
			\dfrac{\dd }{\dd t}\left(|u|^\frac{\alpha-2}{2} u\right) &\in L^2(0,T;L^2(\Omega)),\\
			\tilde{\Theta}&\in L^\infty(0,T;V),\\
			\dfrac{\dd\tilde{\Theta}}{\dd t}&\in L^2(0,T;L^2(\Omega\times (0,L))),
		\end{align*}
		and find $G \in L^\infty(0,T;W^{-1,p'}(\Omega))$ satisfying the variational inequalities~\eqref{limit-problem-u} and~\eqref{limit-problem-tildetheta}, namely,
		\begin{equation*}
			\begin{aligned}
				&\int_{\Omega}|u(T)|^{\alpha-2}u(T)v(T)-u_0^{\alpha-1}v(0)\dd x-\int_{0}^{T}\int_{\Omega}|u(t)|^{\alpha-2}u(t)\dfrac{\dd v}{\dd t}(t)\dd x\dd t\\
				&\quad-\dfrac{\|u(T)\|_{L^\alpha(\Omega)}^\alpha}{\alpha'}+\dfrac{\|u_0\|_{L^\alpha(\Omega)}^\alpha}{\alpha'}
				+\int_{0}^{T} \mu(\tilde{\Theta})\langle G(t),v(t)\rangle_{W^{-1,p'}(\Omega), W_0^{1,p}(\Omega)} \dd t\\
				&\quad-\int_{0}^{T} \mu(\tilde{\Theta})\int_{\Omega} |\nabla(u(t))|^p \dd x \dd t
				\ge \int_{0}^{T} \int_{\Omega} \tilde{a}(t) (v(t)-u(t)) \dd x\dd t,
			\end{aligned}
		\end{equation*}
		and
		\begin{equation*}
			\begin{aligned}
				&\rho c\int_{0}^{T}\int_{0}^{L}\int_{\Omega}\left(\dfrac{\dd \tilde{\Theta}}{\dd t}(t)+U_1(0)\dfrac{\partial\tilde{\Theta}}{\partial x_1}(t)+U_2(0)\dfrac{\partial\tilde{\Theta}}{\partial x_2}(t)+v_z(0)\dfrac{\partial\tilde{\Theta}}{\partial z}(t)\right) (\Xi(t)-\tilde{\Theta}(t)) \dd x \dd z \dd t\\
				&\quad+\kappa\int_{0}^{T}\int_{0}^{L}\int_{\Omega}\nabla\tilde{\Theta}(t)\cdot\nabla(\Xi(t)-\tilde{\Theta}(t))\dd x \dd z \dd t\\
				&\ge 2(\rho g C_0)^p\int_{0}^{T}\int_{0}^{L}\int_{\Omega}A(\cdot,\cdot,\tilde{\Theta}(t))\left(\left\{u(t)^{\frac{p-1}{2p}}-z\right\}^{+}\right)^p (\Xi(t)-\tilde{\Theta}(t))\dd x \dd z \dd t\\
				&\quad-\int_0^T\int_{\Omega}q_{\textup{geo}}^\perp (\Xi(t)-\tilde{\Theta}(t))\dd x\dd t,
			\end{aligned}
		\end{equation*}
		for all $v\in W^{1,2,2}(0,T;W_0^{1,p}(\Omega); L^\alpha(\Omega))$ such that $v(t)\ge 0$ in $\overline{\Omega}$ for all $t\in [0,T]$, and for all $\Xi\in L^\infty(0,T;V)$ such that $0\le\Xi(t)\le\Theta_{\textup{m}}$ for a.a. $t\in (0,T)$.
		
		Moreover, the functions $u$ and $\tilde{\Theta}$ satisfy the following initial conditions
		\begin{equation*}
			\begin{aligned}
				u(0)&=u_0,\\
				\tilde{\Theta}(0)&=\tilde{\Theta}_0,
			\end{aligned}
		\end{equation*}
		for some prescribed $u_0 \in K_{\textup{surf}}$ and $\tilde{\Theta}_0\in K_{\textup{temp}}$.
	\end{customprob}
	\qed
\end{theorem}

It is worth observing that a potential identification of the weak-star limit $G$ with the \emph{expected} term $-\nabla\cdot(|\nabla u|^{p-2}\nabla u)$ hinges, for instance, on the higher regularity for the term $\frac{\dd}{\dd t}(|u|^{\alpha-2}u)$. In general, we cannot say that the term $\frac{\dd}{\dd t}(|u|^{\alpha-2}u)$ is the weak-star limit of the sequence $\{\frac{\dd}{\dd t}(|u_{\ell_n}|^{\alpha-2}u_{\ell_n})\}_{n=1}^\infty$ in $L^\infty(0,T;W^{-1,p'}(\Omega))$ as $n\to\infty$, as this sequence is in general not bounded independently of $n$ (viz.~\eqref{prelim:6}; where $n$ refer to the index of the subsequence $\{\ell_n\}_{n=1}^\infty$).

\addtocontents{toc}{\protect\setcounter{tocdepth}{-1}}

\section*{Conclusions}
In this article we proposed a new model governing the simultaneous evolution of the surface elevation in a shallow ice sheet lying on a flat bedrock and its internal temperature.
Since the ice internal temperature is defined over a \emph{moving domain}, whose shape is determined by the surface elevation of the shallow ice sheet under consideration, and since both the ice surface elevation and the ice internal temperature are subjected to physical constraints, the model takes the form of a coupled system of evolutionary variational inequalities. In order to remove the presence of a boundary term, in the \emph{thermal} variational inequalities, depending on the surface elevation of the shallow ice sheet under consideration, we apply a \emph{corrector} that cuts the ice internal temperature off near the ice surface. This gives rise to a model resembling to the diffuse-interface ones that have been widely studied in the literature for the Cahn--Hilliard equation. The dependence of the approximate model on the parameter $r_0$ still remains, as the initial condition is chosen in a way that it complies with the construction performed in Section~\ref{Sec:1bis}.

The intrinsic degeneracy for the model governing the evolution of the surface elevation in the shallow ice sheet under consideration inserts an additional layer of difficulty, that we overcome by resorting to a change of variables originally suggested by P.-A. Raviart~\cite{Raviart1970}.

In order to establish the existence of solutions for the problem under consideration, we first define a \emph{penalised} version for the formal model we derived, and we establish the existence of solutions for this penalised model by finite differences in time. Finally, we pass to the limit in the penalty parameter and we recover a model that agrees with the one we had recovered formally.

\section*{Acknowledgements}

The research of P.P. is currently being supported by the following agencies:
\begin{itemize}
	\item Natural Science Foundation of China (NSFC), grant number W2533011;
	\item ZJ Talent Program of the Guangdong Province, grant number 2024QN11X057;
	\item Peacock Talent Program - Type~C of the Shenzhen Municipality.
\end{itemize}

P.P. is grateful to Professor Maurizio Grasselli (Politecnico di Milano, Italy) for the insightful discussion on the literature concerning the Cahn--Hilliard equation and diffuse interface models.

\addtocontents{toc}{\protect\setcounter{tocdepth}{1}}

\bibliographystyle{abbrvnat} %unsrt	
\bibliography{references.bib}

\addresseshere

\newpage
% 5. Print the two glossaries separately
\printglossary[type=mathsymbols, title={Mathematical Symbols}]
\printglossary[type=physsymbols, title={Physics Symbols}]

%\newpage
%\appendix
%\input{appendix}

\end{document}